\documentclass[11pt]{article}

\usepackage{latexsym,amssymb,amsmath,amsthm,graphicx,float,bm}
\usepackage{gpeyre}
\usepackage{epstopdf}
\usepackage{epsfig}

\renewcommand{\pd}{\partial}

\newcommand{\R}{\mathbb{R}}

\newcommand{\<}{\langle}
\renewcommand{\>}{\rangle}

\newcommand{\eps}{\epsilon}
\newcommand{\spc}{\sigma_{\scriptsize\mbox{pc}}}
\renewcommand{\Pr}{\mbox{Pr}}

\newcommand{\punif}{p_{\mbox{\scriptsize unif}}}
\newcommand{\Kunif}{K_{\mbox{\scriptsize unif}}}
\newcommand{\unif}{\mbox{\scriptsize unif}}

\newcommand{\sigmaj}{\sigma[j]}
\newcommand{\sigmai}{\sigma[i]}

\graphicspath{{./images/}}

\title{Compressive Wave Computation}
\author{
\begin{tabular}{c}
    Laurent Demanet\\
     Department of Mathematics \\
     Stanford University\\
     Stanford, CA 94305
\end{tabular}
\qquad
\begin{tabular}{c}
    Gabriel Peyr\'{e}\\
     CNRS and Ceremade \\
     Universit\'e Paris-Dauphine\\
     75775 Paris Cedex 16
\end{tabular}
}
\date{August 2008}        

\begin{document}
\maketitle

\begin{abstract}

This paper considers large-scale simulations of wave propagation phenomena. We argue that it is possible to accurately compute a wavefield by decomposing it onto a largely incomplete set of eigenfunctions of the Helmholtz operator, chosen at random, and that this provides a natural way of parallelizing wave simulations for memory-intensive applications.

Where a standard eigenfunction expansion in general fails to be accurate if a single term is missing, a sparsity-promoting $\ell_1$ minimization problem can vastly enhance the quality of synthesis of a wavefield from low-dimensional spectral information. This phenomenon may be seen as "compressive sampling in the Helmholtz domain", and has recently been observed to have a bearing on the performance of data extrapolation techniques in seismic imaging [T. Lin and F. Herrmann, Geophysics, 2007].

This paper shows that $\ell_1$-Helmholtz recovery also makes sense for wave computation, and identifies a regime in which it is provably effective: the one-dimensional wave equation with coefficients of small bounded variation. Under suitable assumptions we show that the number of eigenfunctions needed to evolve a sparse wavefield defined on $N$ points, accurately with very high probability, is bounded by
\[
C(\eta) \cdot \log N \cdot \log \log N,
\]
where $C(\eta)$ is related to the desired accuracy $\eta$ and can be made to grow at a much slower rate than $N$ when the solution is sparse. The PDE estimates that underlie this result are new to the authors' knowledge and may be of independent mathematical interest; they include an $L^1$ estimate for the wave equation, an $L^\infty - L^2$ estimate of extension of eigenfunctions, and a bound for eigenvalue gaps in Sturm-Liouville problems.

In practice, the compressive strategy makes sense because the computation of eigenfunctions can be assigned to different nodes of a cluster in an embarrassingly parallel way. Numerical examples are presented in one spatial dimension and show that as few as 10 percents of all eigenfunctions can suffice for accurate results. Availability of a good preconditioner for the Helmholtz equation is important and also discussed in the paper. Finally, we argue that the compressive viewpoint suggests a competitive parallel algorithm for an adjoint-state inversion method in reflection seismology.

\end{abstract}

{\bf Acknowledgements.}

We would like to thank Ralph Smith and Jim Berger of the Statistical and Applied Mathematical Sciences Institute (SAMSI) for giving us the opportunity to visit the Institute, which catalyzed the completion of this project. We are indebted to Felix Herrmann for early discussions, Paul Rubin for help on a probability question related to Proposition \ref{teo:nonuniform}, and Lexing Ying for help with a point in the numerical implementation of our preconditioner. Thanks are also due to Emmanuel Cand\`{e}s, Mark Embree, Jalal Fadili, Josselin Garnier, Mauro Maggioni, Justin Romberg, and William Symes for useful discussions. L.D. is supported in part by a grant from the National Science Foundation.

\pagebreak

\section{Introduction}

In this paper we consider a simple model for acoustic waves,
\begin{equation}\label{eq:wave}
\sigma^2(x) \frac{\pd^2 u}{\pd t^2}(x,t) - \frac{\pd^2 u}{\pd x^2}(x,t) = 0, \qquad x \in [0,1],
\end{equation}
\[
u(x,0) = u_0(x), \qquad \frac{\pd u}{\pd t}(x,0) = u_1(x),
\]
with Dirichlet ($u(0,t) = u(1,t) = 0$) or Neumann ($u'(0,t) = u'(1,t) = 0$) boundary conditions. The parameter $\sigma(x)$ is viewed as the acoustic impedance of the medium, we will assume at least that $0 < \sigma_{\min} \leq \sigma(x) \leq \sigma_{\max}$, almost everywhere. In the sequel we will see how the classical equation of a vibrating string, with parameters $\rho(x)$ and $\mu(x)$, can be transformed into (\ref{eq:wave}) in such a way that all the results of this paper hold unaffected in the more general case.

Most numerical methods for solving (\ref{eq:wave}) fall into two categories: either simulation in the time domain, by timestepping from initial and boundary data; or in the frequency domain using an expansion in terms of time-harmonic solutions. For the latter approach, when $u(x,t) = v_\omega(x) e^{i \omega t}$, then $v_\omega$ solves the Helmholtz equation
\begin{equation}\label{eq:Helmholtz}
v''_\omega(x) + \omega^2 \sigma^2(x) v_\omega(x) = 0, \qquad x \in [0,1],
\end{equation}
with Dirichlet or Neumann boundary conditions. The special values of $\omega$ for which (\ref{eq:Helmholtz}) has a solution correspond to eigenvalues $\lambda$ of the operator $\mathcal{L} = \sigma^{-2} d^2/dx^2$ with the same boundary boundary conditions, through $\lambda = -\omega^2$. The natural inner product that makes the operator $\mathcal{L}$ self-adjoint is the one of the weighted space $L^2_{\sigma^2}([0,1], \R)$,
\begin{equation}\label{eq:weighted-innerprod}
\< f,g \> = \int_0^1 f(x) g(x) \sigma^2(x) dx.
\end{equation}
(The regular $L^2$ inner product will not be used in the sequel.) Self-adjointness of $\mathcal{L}$ classically implies orthogonality of the eigenfunctions, in the sense that $\< v_{\omega_1}, v_{\omega_2} \> = \delta_{\omega_1, \omega_2}$. The operator $\mathcal{L}$ is also negative definite (Dirichlet) or negative semi-definite (Neumann) with respect to the weighted inner product, which justifies the choice of sign for the eigenvalues. As a result, any solution $u(x,t)$ of (\ref{eq:wave}) can be expanded as
\begin{equation}\label{eq:sumomega}
u(x,t) = \sum_{\omega} c_\omega(t) v_\omega(x),
\end{equation}
where
\[
c_\omega(t) = \< u_0, v_\omega \> \cos \omega t + \< u_1, v_\omega \> \frac{\sin \omega t}{\omega}.
\]

If we are ready to make assumptions on $u(x,t)$, however, equation (\ref{eq:sumomega}) may not be the only way of synthesizing $u(x,t)$ from its coefficients $c_\omega(t) = \< u(\cdot, t), v_\omega \>$. The situation of interest in this paper is when $u(x,t)$ is approximately sparse for each time $t$, i.e, peaks at a few places in $x$, and the impedance $\sigma(x)$ has minimal regularity properties. When these conditions are met, we shall see that \emph{the information of the solution $u(x,t)$ is evenly spread, and contained several times over in the full collection of coefficients $c_\omega(t)$.} In other words, there is a form of uncertainty principle to guarantee that a solution that peaks in $x$ cannot peak in $\omega$, and that it is possible to decimate $c_\omega(t)$ in such a way that a good approximation of $u(x,t)$ can still be recovered. In addition, there exists a nonlinear procedure for this recovery task, which is so simple that it may have nontrivial implications for the numerical analysis of wave equations.

\subsection{The Compressive Strategy}

How to recover a sparse, sampled function $f(j/N)$ from a set of discrete orthobasis coefficients $c_k$, which is incomplete but nevertheless ``contains the information" of $f(j/N)$, has been the concern of a vast body of recent work on compressed sensing \cite{candes-robust, donoho-cs}. In this approach, the proposed solution for recovering $f(j/N)$ is to find a vector with minimum $\ell_1$ norm among all vectors which have the observed $c_k$ as coefficients. Compressed sensing currently finds most of its applications in signal or image processing, but its philosophy turns out to be relevant for computational wave propagation as well, where the basis vectors are discretized eigenfunctions of the Helmholtz equation.

Accordingly, we formulate an $\ell_1$ problem for recovering $u(x,t)$ at fixed time $t$ from a restricted set of coefficients $c_\omega(t)$. To guarantee that these coefficients be informative about $u(x,t)$, we will \emph{sample $\omega$ at random}. For practical reasons explained later, we use the following sampling scheme: 1) draw a number $w$ uniformly at random between $0$ and some maximum value $\omega_{\max}$, 2) find the closest $\omega_{[k]}$ to $w$, such that $\lambda_{[k]} = - \omega^2_{[k]}$ is an eigenvalue, 3) add this eigenvalue to a list provided it does not already belong to it, and 4) repeat until the size of the list reaches a preset value $K$. Eventually, denote this list as $\Omega_K = \{ \omega_{[k]} : k = 1, \ldots, K \}$.

Putting aside questions of discretization for the time being, the main steps of compressive wave computation are the following. Fix $t > 0$.

\bigskip

\begin{listing}
\bigskip
\begin{enumerate}
\item Form the randomized set $\Omega_K$ and compute the eigenvectors $v_\omega(x)$ for $\omega \in \Omega_K$;
\item Obtain the coefficients of the initial wavefields as $\< u_0, v_\omega \>$ and $\< u_1, v_\omega \>$, for $\omega \in \Omega_K$;
\item Form the coefficients of the solution at time $t$ as
\[
c_\omega(t) = \< u_0, v_\omega \> \cos \omega t + \< u_1, v_\omega \> \frac{\sin \omega t}{\omega},  \quad \omega \in \Omega_K;
\]
\item Solve the minimization problem
\begin{equation}\label{eq:minell1}
\min_{\tilde{u}} \int_0^1 \sigma(x) |\tilde{u}(x)| \, dx, \quad \mbox{such that } \< \tilde{u}, v_\omega \> = c_\omega(t), \quad \omega \in \Omega_K.
\end{equation}
\end{enumerate}
\end{listing}

\bigskip

The algorithmic specifics of points 1 and 4 will be discussed at length in the sequel. The main result of this paper concerns the number $K = |\Omega_K|$ of coefficients needed to ensure that the minimizer $\tilde{u}(x)$ just introduced approximates $u(x,t)$ within a controlled error, and controlled probability of failure.

\subsection{Main Result}

Specific notions of \emph{sparsity of the solution}, and \emph{smoothness of the medium}, will be used in the formulation of our main result.

\begin{itemize}

\item {\em Sparsity of the solution}. Assume that $u[j](t)$ is the solution of a discretization of the problem (\ref{eq:wave}) on $N$ equispaced spatial points, i.e., $u[j](t) \simeq u(j/N,t)$. A central quantity in the analysis is the ``size of the essential support" $S_\eta$ of this discretized solution, i.e., in a strategy that approximates $u[j](t)$ by its $S$ largest entries in modulus, how large would $S$ need to be so that the error made is less than a threshold $\eta$ in a weighted $\ell_1$ sense. Then this particular choice of $S$ is called $S_\eta$. More precisely, for fixed $t > 0$ consider a tolerance $\eta > 0$, and the largest level $\gamma$ at which
\begin{equation}\label{eq:Seta}
\sum_{j: |u[j](t)| \leq \gamma} \sigmaj | u[j](t) |  \leq \eta,
\end{equation}
where $\sigmaj$ should be thought of as $\sigma(j/N)$, or possibly a more sophisticated variant involving cell averages. Then $S_\eta$ is the number of terms in this sum, i.e., the number of grid points indexed by $j$ such that $|u[j](t)| \leq \gamma$. In all cases $S_\eta \leq N$, but the sparser $u(x,t)$ the smaller $S_\eta$.

\item {\em Smoothness of the medium}. We consider acoustic impedances $\sigma(x)$ with small \emph{total variation} Var $\log(\sigma)$. In particular, we will require Var$(\log \sigma) < \pi$. See Section \ref{sec:three} for a discussion of total variation. Note that $\sigma(x)$ is permitted to have discontinuities in this model.

\end{itemize}


Three sources of error arise in the analysis: 1) the discretization error
\[
\tau = \| u[j](t) - u(j/N,t) \|_{\ell_2},
\]
where the $\ell_2$ norm is over $j$; 2) the truncation error $\eta$ just introduced, due to the lack of exact sparsity; and 3) a numerical error $\epsilon$ made in computing the discrete eigenvalues and eigenvectors, which shows as a discrepancy
\[
\eps = \| \tilde{c}_{\tilde{\omega}}(t) - c_\omega(t) \|_{\ell_2},
\]
for fixed $t$, and where the $\ell_2$ norm is over $\omega \in \Omega_K$. In addition to the \emph{absolute} accuracy requirement that $\epsilon$ be small, we also need the following minor \emph{relative} accuracy requirements for the eigenvalues and eigenvectors. Two quantities $A$, $B$ are said to be within a factor 2 of each other if $A/2 \leq B \leq 2A$.

\begin{definition}\label{def:faith} (Faithful discretization)
A spatial discretization of the Helmholtz equation (\ref{eq:Helmholtz}) is called \emph{faithful} if the following two properties are satisfied.
\begin{itemize}
\item The gap $|\tilde{\omega}_{n+1} - \tilde{\omega}_n|$ between two consecutive eigenvalues of the discrete equation is within a factor 2 of the gap $|\omega_{n+1} - \omega_n|$ between the corresponding exact eigenvalues;
\item There exists a normalization of the computed eigenvectors $\tilde{v}_{\tilde{\omega}}[j]$ and exact eigenvectors $v_\omega(x)$ such that the discrete norm $\left( \frac{1}{N} \sum_j \sigmaj^2 \tilde{v}_{\tilde{\omega}}[j]^2 \right)^{1/2}$ is within a factor 2 of $\| v_\omega \|_{L^2_{\sigma^2}}$, and simultaneously $\max_j | \sigmaj \tilde{v}_{\tilde{\omega}}[j]|$ is within a factor 2 of $\| \sigma v_\omega \|_{L^\infty}$.
\end{itemize}
\end{definition}


Finally, we assume exact arithmetic throughout. The following theorem is our main result.

\begin{theorem}\label{teo:main} Assume that Var$(\log \sigma) < \pi$, and that the discretization of (\ref{eq:Helmholtz}) is faithful in the sense of Definition \ref{def:faith}. Assume that $K$ eigenvectors are drawn at random according to the procedure outlined above. There exists $C(\sigma)$ such that if $K$ obeys
\begin{equation}\label{eq:boundK}
K \geq C(\sigma) \cdot S_\eta \log N \cdot \log^2(S_\eta) \log (S_\eta \log N),
\end{equation}
(with $N$ sufficiently large so that all the logarithms are greater than 1), then with very high probability the solution $\tilde{u}[j](t)$ of a discrete version of the minimization problem (\ref{eq:minell1}) obeys
\begin{equation}\label{eq:accuracy}
\| u(j/N,t) - \tilde{u}[j](t) \|_{\ell_2} \leq \frac{C_1}{\sqrt{S_\eta}} \eta + C_2 \eps + \tau.
\end{equation}
Furthermore, $C(\sigma)$ can be taken to obey
\begin{equation}\label{eq:Csigma}
C(\sigma) \leq C_3 \cdot \frac{\pi + \mbox{Var}(\log \sigma)}{\pi - \mbox{Var}(\log \sigma)} \cdot \, \exp ( 2 \mbox{Var} (\log \sigma)),
\end{equation}
where $C_3$ is a numerical constant. ``Very high probability" here means $1 - O(N^{-\alpha})$ where $\alpha$ is an affine, increasing function of $C_3$.
\end{theorem}

The discrete CWC algorithm is described in Section \ref{sec:disc}. The proof of Theorem \ref{teo:main} is in Sections \ref{sec:proof} and \ref{sec:proof2}.


Contrast this behavior of the $\ell_1$ problem by considering instead an $\ell_2$ regularization; by Plancherel the solution $u^\sharp$ would then simply be the truncated sum
\[
u^\sharp(x,t) = \sum_{\omega \in \Omega_K} c_\omega(t) v_\omega(x),
\]
with error
\[
\| u(\cdot, t) - u^\sharp(\cdot, t) \|^2_{L^2_{\sigma^2}} = \sum_{\omega \notin \Omega_K} |c_\omega(t)|^2.
\]
This shows that $u^\sharp$ may be very different from $u$ as soon as $\Omega_K$ is not the complete set.
In fact, the uncertainty principles discussed in this paper would show that the resulting error for sparse $u(x,t)$ and random $\Omega_K$ is large with very high probability.

The estimate (\ref{eq:accuracy}) is a statement of \emph{robustness} of the $\ell_1$ recovery method, and is intimately related to a well-established stability result in compressed sensing \cite{candes-stable}. If errors are made in the discretization, in particular in computing eigenvectors on which the scheme is based $(\eps \ne 0)$, then equation (\ref{eq:accuracy}) shows that this error will carry over to the result without being overly amplified. This robustness is particularly important in our context since in practice a compressive numerical scheme is bound to be the ``driver" of a legacy code for the Helmholtz equation.

Theorem \ref{teo:main} also states that, if the discrete solution happens to be compactly supported on a small set ($S_\eta << N$ for $\eta = 0$), and no error is made in solving for the Helmholtz equation, then the compressive scheme would recover the discrete solution \emph{exactly}, with high probability, and without using all the eigenvectors. It is not unreasonable to speak about compact support of the solution of a wave equation, since the speed of propagation is finite.

Another important point is that there exists no particular, ``optimized" choice of the eigenvectors that can essentially beat choosing them at random. The compressive strategy, whether in numerical analysis or signal processing, is intrinsically probabilistic. In fact, making a deterministic choice can possibly void the reliability of recovery as there would exist counter-examples to Theorem \ref{teo:main}.

Finally, the estimate is quite important from the viewpoint of computational complexity. In situations where $S_\eta$ is small, computing $K \sim S_\eta \log N$ eigenvectors whose identity is unimportant as long as it is sufficiently randomized, can be much more advantageous than computing all $N$ of them. This leads to an easily parallelizable, frequency-based algorithm for the wave equation requiring at most $\sim S_\eta N^2 \log N$ operations instead of the usual $\sim N^3$ for a QR method computing all eigenvectors. Complexity and parallelism questions are further discussed below.

\subsection{Why It Works}\label{sec:why}

The $\ell_1$ regularization is chosen on the basis that it promotes sparsity while at the same time defining a convex optimization problem amenable to fast algorithms. It provides an exact relaxation of $\ell_0$ problems when the vector to be recovered is sufficiently sparse, as was identified by David Donoho and co-workers in the mid-nineties in the scope of work on basis pursuit \cite{chen-basis-pursuit}. The same $\ell_1$ minimization extracts information of a sparse vector remarkably well in the presence of incomplete data, provided the data correspond to inner products in a basis like Fourier, for which there exists a form of uncertainty principle. This observation was probably first treated mathematically in 1989 in \cite{donoho-uncertainty}, and was refined using probabilistic techniques in work of Cand\`{e}s, Romberg, and Tao \cite{candes-robust}, as well as Donoho \cite{donoho-cs}, both in 2004. This line of work received much attention and came to be known as compressed sensing or compressive sampling.

Consider the $\ell_1$ minimization problem in $\R^N$,
\[
\min \| f \|_1 \qquad s.t. \qquad   f \cdot \phi_k = c_k, \qquad k \in \Omega_K,
\]
where $\Omega_K$ is a subset of $1, \ldots, N$, and $\{ \phi_k \}$ is an orthobasis. The inner product $f \cdot \phi_k$ is called a ``measurement". The following three conditions are prerequisites for guaranteeing the success of recovery of $\ell_1$ minimization.
\begin{enumerate}
\item The vector $f$ to be recovered needs to be \emph{sparse};
\item The basis vectors $\phi_k$ are \emph{incoherent}; and
\item The actual $\phi_k$ used as measurement vectors are chosen \emph{uniformly at random} among the $\phi_k$.
\end{enumerate}

Sparsity of a vector can mean small support size, but in all realistic situations it is measured from the decay of its entries sorted in decreasing order. Typically, a vector is sparse when it has a small $\ell_p$ quasi-norm for $0 < p \leq 1$, and the smaller $p$ the stronger the measure of sparsity.

Incoherence of an orthobasis $\{ \phi_k \}$ of $\R^N$ means that
\begin{equation}\label{eq:incoherence}
\sup_{j,k=1,\ldots, N} | \phi_k [j] | \leq \frac{\mu}{\sqrt{N}},
\end{equation}
where the parameter $\mu \geq 1$, simply called incoherence, is a reasonably small constant. For instance, if $\phi_k$ are the column of the isometric discrete Fourier transform, then $\mu$ attains its lower bound of $1$, and the Fourier vectors are said to be maximally incoherent. The generalization where an orthobasis is incoherent with respect to another basis is often considered in the literature, in which case incoherence means small inner products of basis vectors from the two bases. Hence in our case we also speak of incoherence with respect to translated Diracs.

The condition of uniform random sampling of the measurement vectors is needed to avoid, with high probability,  complications of an algebraic nature that may prevent injectivity of the projection of sparse vectors onto the restricted set of measurements, or at least deteriorate the related conditioning. More quantitatively, randomness allows to promote the incoherence condition into a so-called restricted isometry property (RIP). We will have more to say about this in the sequel. Note that if the measurement basis is itself generated randomly, e.g. as the columns of a random matrix with i.i.d. gaussian entries, then the further randomization of the measurement vectors is of course not necessary.

When these three conditions are met, the central question is the number $K$ of measurements needed for the $\ell_1$ recovery to succeed. Recent papers \cite{candes-near-optimal, rudelson-sparse}  show that the best answer known to date is
\[
K \geq C \cdot \mu^2 \cdot S \log N \cdot \log^2 (S) \log (S \log N),
\]
where $S$ is the number of ``big" entries in the vector to be recovered, $\mu$ is the incoherence, and $C$ is a decent constant. (The trailing log factors to the right of $S \log N$ are conjectured to be unnecessary.) Stability estimates of the recovery accompany this result \cite{candes-stable}. All this will be made precise in Section \ref{sec:two}; let us only observe for now that (\ref{eq:boundK}) is manifestly a consequence of this body of theory.

The mathematical contribution of this paper is the verification that the conditions of 1) sparsity, 2) incoherence, and 3) uniform sampling are satisfied in presence of Helmholtz measurements for solutions of the wave equation, and in the context of a practical algorithm. For all three of these requirements, we will see that a single condition of bounded variation on $\log \sigma$ suffices in one spatial dimension.

\begin{itemize}
\item We show in Section \ref{sec:sparsity} that when Var$(\log \sigma) < 1$, the wave equation (\ref{eq:wave}) obeys a Strichartz-like $L^1$ estimate
\[
\int_0^1 \sigma(x) |u(x,t)| dx \leq D(\sigma) \cdot \left( \int_0^1 \sigma(x) |u_0(x)| dx + \int_0^1 \sigma^2 (x) | \int_0^x u_1(y) dy | dx \right),
\]
for $t \leq 1 / \sigma_{\min}$, and where
\[
D(\sigma) = C \sqrt{\frac{\sigma_{\max}}{\sigma_{\min}}} \frac{1}{1- \mbox{Var} (\log \sigma)}.
\]
This result shows that, in an $L^1$ sense, the sparser the initial conditions the sparser the solution at time $t$. Hence choosing sparse initial conditions gives a control over the quality of $\ell_1$ recovery through the discrete quantity $S_\eta$ introduced above---although we don't have a precise estimate to quantify this latter point.

\item We show in Section \ref{sec:incoherence} that when Var$(\log \sigma) < \infty$, solutions of the Helmholtz equation must be extended, which translates into a property of incoherence at the discrete level. If $v_\omega$ solves (\ref{eq:Helmholtz}), the extension estimate is
\[
\| \sigma v_\omega \|_{L^\infty} \leq \sqrt{2} \exp ( \mbox{Var} (\log \sigma)) \cdot \| v_\omega \|_{L^2_{\sigma^2}}.
\]
A quadrature of the integral on the right-hand side would reveal a factor $1/\sqrt{N}$, hence a comparison with the definition (\ref{eq:incoherence}) of incoherence shows that the leading factor $\sqrt{2} \exp ( \mbox{Var} (\log \sigma))$ is in fact---modulo discretization questions---an upper bound on the incoherence $\mu$ of eigenfunctions with translated Diracs.

\item We show in Section \ref{sec:egv-gaps} that when Var$(\log \sigma) < \pi$, two eigenvalues of the Helmholtz equation (\ref{eq:Helmholtz}) cannot be too close to each other: if $\lambda_{j} = - \omega_j^2$ for $j = 1,2$ are two distinct eigenvalues, then
\[
\frac{\pi - \mbox{Var}(\log \sigma)}{\int_0^1 \sigma(x) \, dx} \leq |\omega_1 - \omega_2| \leq \frac{\pi + \mbox{Var}(\log \sigma)}{\int_0^1 \sigma(x) \, dx},
\]
This gap estimate shows that if we draw numbers uniformly at random within $[0, \omega_{\max}]$, and round them to the nearest $\omega_{[k]}$ for which $\lambda_{[k]} = -\omega^2_{[k]}$ is an eigenvalue, then the probabilities $p_{n}$ of selecting the $n$-th eigenvalue $\lambda_n$ are of comparable size uniformly in $n$---again, modulo discretization questions. This provides control over departure from the uniform distribution, and quantifies the modest penalty incurred in the bound on $K$ as the ratio of probabilities
\[
\frac{\min_n  p_{n}}{\punif} \geq \frac{\pi - \mbox{Var}(\log \sigma)}{\pi + \mbox{Var}(\log \sigma)},
\]
where $\punif$ refers to the case of uniform $\sigma$ and equispaced $\omega_n$.

\end{itemize}

Bounded variation can be thought of as the minimum smoothness requirement of an one-dimensional acoustic medium, for which wave propagation is somewhat coherent, and no localization occurs. For instance, the incoherence result physically says that the localization length is independent of the (temporal) frequency of the wave, for media of bounded variation.

All the points discussed above will be properly integrated in the justification of Theorem \ref{teo:main}, in Section \ref{sec:two}. The proofs of the three PDE results are in Section \ref{sec:three}.



\subsection{How It Works}\label{sec:howitworks}

A classical algorithm for computing all $N$ eigenvectors of the discrete Helmholtz equation would be the QR method, with complexity an iterative $O(N^3)$. Since not all eigenvectors are required however, and since the discretized operator $\mathcal{L}$ is an extremely sparse matrix, all-purpose linear algebra methods like QR acting on matrix elements one-by-one are at a big disadvantage.

Instead, it is more natural to set up a variant of the power method with randomized shifts for computing the desired eigenvectors and corresponding eigenvalues. We have chosen the restarted Arnoldi method coded in Matlab's eigs command. In our context, the power method would compute the inverse $( \mathcal{L} + w^2 )^{-1}$ for a shift $w$ chosen at random, and apply it repeatedly to a starting vector (with random i.i.d. entries, say), to see it converge to the eigenvector with eigenvalue closest to $- w^2$. The probability distribution to place on $w$ should match the spectral density of $\mathcal{L}$ as closely as possible; one important piece of a priori information is the estimate (\ref{eq:egv-gaps}) on eigenvalue gaps.

Applying $( \mathcal{L} + w^2 )^{-1}$, or in practice solving the system $( \mathcal{L} + w^2 )u = f$, can be done in complexity $O(N^2)$ using an iterative method. However, the inversion needs not be very accurate at every step, and can be sped up using an adequate preconditioner. It is the subject of current research to bring the complexity of this step down to $O(N)$ with a constant independent of frequency $w$, see \cite{ErlNab} for some of the latest developments. In this paper we use a particularly efficient preconditioner based on discrete symbol calculus \cite{DSC} for pre-inverting $\mathcal{L} - w^2$ (there is no typo in the sign), although the resulting overall complexity is still closer to $O(N^2)$ than to $O(N)$.


The resolution of $\ell_1$ minimization can be done using standard convex optimization methods \cite{boyd-convex-optim}  such as interior point for linear programming in the noiseless setting, and second order cone programming in the noisy case \cite{chen-basis-pursuit}. One can however exploit the separability of the $\lun$ norm and design specific solvers such as exact \cite{donoho-homotopy} or approximate \cite{efron-lars}  path continuation and iterative thresholding \cite{figueiredo-nowak-em,daubechies-iterated,combettes-proximal,starck-mca}. In this paper we used a simple iterative thresholding algorithm.


The complexity count is as follows:
\begin{itemize}
\item It takes $O(K N^2)$ operations to solve for $K$ eigenvectors with a small constant, and as we have seen, $K$ is less than $N$ in situations of interest;
\item Forming spectral coefficients at time $t = 0$ and their counterpart at time $t$ is obviously a $O(N)$ operation.
\item It is known that solving an $\ell_1$ problem with $N$ degrees of freedom converges geometrically using the methods explained below \cite{Bredies} and is therefore a $O(N)$ operation.
\end{itemize}

Compared to the QR method, or even to traditional timestepping methods for (\ref{eq:wave}), the compressive scheme has a complexity that scales favorably with $N$. The control that we have over the size of $K$ comes from the choice of sparse initial conditions, and also from the choice of time $t$ at which the solution is desired. If the initial conditions are not sparse enough, they can be partitioned into adequately narrow bumps by linearity; and if $t$ is too big, the interval $[0,t]$ can be divided into subintervals over which the compressive strategy can be repeated.

More details on the implementation can be found in Section \ref{sec:implementation}.





\subsection{Significance}

The complexity count is favorable, but we believe the most important feature of the compressive solver, however, is that the computation of the eigenvectors is ``embarrassingly parallel", i.e., parallelizes without communication between the nodes of a computer cluster. This is unlike both traditional timestepping and Helmholtz-via-QR methods. The $\ell_1$ solver---proximal iterative thresholding---also parallelizes nicely and easily over the different eigenvectors, without setting up any domain decomposition method. We leave these ``high-performance computing" aspects to a future communication.

Another decisive advantage of the compressive strategy is the ability to freely access the solution at different times. In inverse problems involving the wave equation, where an adjoint-state equation is used to form the gradient of a misfit functional, one is not just interested in solving the wave equation at a given time $t$, but rather in forming combinations such as
\[
\int_0^T \frac{\pd^2 u}{\pd^2 t}(x,t) q(x,t) \, dt,
\]
where $u$ solves an initial-value problem, and $q$ solves a final-value problem---the adjoint-state wave equation. A classical timestepping method requires to keep large chunks of the history in memory, because $u(x,t)$ is formed by time stepping up from $t=0$, and $q(x,t)$ is formed by time stepping down from $t = T$. The resulting memory overhead in the scope of reverse-time migration in reflection seismology is well documented in \cite{Symes-checkpoint}. The compressive scheme in principle alleviates these memory issues by minimizing the need for timestepping.

Let us also comment on the relief provided by the possibility of solving for incomplete sets of eigenvectors in the scope of compressive computing. As we saw, using an iterative power-type method has big advantages over a QR method. Could an iterative method with randomized shifts be used to compute all eigenvectors? The answer is hardly so, because we would need to ``collect" all eigenvectors from repeated random sampling. This is the problem of coupon collecting in computer science; the number of realizations to obtain all $N$ coupons, or eigenvectors, with high probability, is $O(N \log N)$ with a rather large constant. If for instance we had already drawn all but one coupons, the expected number of draws for finding the last one is a $O(N)$---about as much as for drawing the first $N/2$ ones! For this reason, even if the required number of eigenvectors is a significant fraction of $N$, the compressive strategy would remain attractive.

Finally, we are excited to see that considerations of compression in the eigenfunction domain requires new kinds of quantitative estimates concerning wave equations. We believe that the questions of sparsity, incoherence, and eigenvalue gaps are completely open in two spatial dimensions and higher, for coefficients $\sigma(x)$ that present interesting features like discontinuities.

\subsection{Related Work}

As mentioned in the abstract, and to the best of our knowledge, $\ell_1$ minimization using Helmholtz eigenfunctions as measurement basis vectors was first investigated in the context of seismic imaging by Lin and Herrmann \cite{LH} in 2007. Data extrapolation in seismology is typically done by marching the so-called single square root equation in depth; Herrmann and Lin show that each of those steps can be formulated as an optimization problem with an $\ell_1$ sparsity objective in a curvelet basis \cite{FDCT}, and constraints involving incomplete sets of eigenfunctions of the horizontal Helmholtz operator. The depth-stepping problem has features that make it simpler than full wave propagation---extrapolation operators with small steps are pseudodifferential hence more prone to preserving sparsity than wave propagators---but \cite{LH} deals with many practical considerations that we have idealized away in this paper, such as the choice of basis, the higher dimensionality, and the handling of real-life data.

Sparsity alone, without using eigenfunctions or otherwise incoherent measurements, is also the basis for fast algorithms for the wave equation, particularly when the same equation needs to be solved several times. Special bases such as curvelets and wave atoms have been shown to provide sparse representations of wave propagators \cite{Smith-wave, CD}. Speedups of a factor as large as 20, over both spectral and finite difference methods, have been reported  for a two-dimensional numerical method based on those ideas \cite{watuwe}. See also \cite{EOZ} for earlier work in one spatial dimension. Computational harmonic analysis for the solution of PDE has its roots in the concept of representing singular integral operators and solving elliptic problems using wavelets \cite{Mey1, BCR, CDD}. 

From a pure complexity viewpoint, methods based solely on decomposing the equation in a fixed basis are not entirely satisfactory however, because of the heavy tails of supposedly nearly-diagonal matrices. The resulting constants in the complexity estimates are cursed by the dimensionality of phase-space, as documented in \cite{watuwe}. There is also the question of flexibility vis-a-vis special domains or boundary conditions. We anticipate that representing the equation in an incoherent domain, instead, followed by a sparse recovery like in this paper may not suffer from the same ills.

Of course, sparsity and incoherence ideas have a long history in imaging problems unrelated to computation of PDE, including in geophysics \cite{santosa-sparse-spikes}. As mentioned earlier, from the mathematical perspective this line of work has mostly emerged from the work of Donoho and collaborators \cite{donoho-uncertainty, chen-basis-pursuit}. See \cite{Mal3} for a nice review.

Finally, previous mathematical work related to the theorems in Section \ref{sec:three} are discussed at the end of the respective proofs.

\section{The Compressive Point of View: from Sampling to Computing}\label{sec:two}

In this section we expand on the reasoning of Section \ref{sec:why} to justify Theorem \ref{teo:main}. We first introduce the quoted recent results of  sparse signal recovery.

\subsection{Elements of Compressed Sensing}\label{sec:CS}


Consider for now the generic problem of recovering a sparse vector $f_0$ of $\R^N$ from noisy measurements $y = \Phi f_0 + z \in \R^K$, with $K \leq N$, and $\| z \|_2 \leq \eps$. The $\ell_1$ minimization problem of compressed sensing is
\[
\qquad\qquad\qquad\qquad\qquad\qquad \min \| f \|_{\ell_1}, \qquad \mbox{s.t.} \qquad \| \Phi f - y \|_2 \leq \eps. \qquad\qquad\qquad\qquad\qquad\qquad (\mbox{P}_1)
\]
At this level of generality we call $\Phi$ the measurement matrix and take its rows $\phi_k$ to be orthonormal. Accurate recovery of (P${}_1$) is only possible if the vectors $\phi_k$ are incoherent, i.e., if they are as different as possible from the basis in which $f_0$ is sparse, here Dirac deltas. Cand\`{e}s, Romberg and Tao make this precise by introducing the $S$-restricted isometry constant
    $\de_S$, which is the smallest $0<\de<1$ such that
        \eql{\label{eq-defn-rih}
            (1-\de) \norm{c}_{\ell_2}^2 \leq \norm{ \Phi_T c }_{\ell_2}^2 \leq (1+\de) \norm{c}_{\ell_2}^2.
        }
        for all subsets $T$ of $\{ 1, \ldots, N \}$ such that $|T| \leq S$, for all vectors $c$ supported on $T$, and
        where $\Phi_T$ is the sub-matrix extracted from $\Phi$ by selecting the columns in $T$. Equation (\ref{eq-defn-rih}) is called \emph{$S$-restricted isometry property}.

Cand\`{e}s, Romberg and Tao proved the following result in \cite{candes-stable}.



\begin{theorem}\label{teo:stab}
For $f_0 \in \R^N$, call $f_{0,S}$ the best approximation of $f_0$ with support size $S$. Let $S$ be such that $\delta_{3S} + 3 \delta_{4S} < 2$. Then the solution $f$ of (P${}_1$) obeys
\begin{equation}\label{eq:errorCS}
\| f - f_{0} \|_2 \leq C_{1} \, \eps + C_{2} \, \frac{\| f_0 - f_{0,S} \|_{\ell_1}}{\sqrt{S}}.
\end{equation}
The constants $C_1$ and $C_2$ depend only on the value of $\delta_{3S}$ and $\delta_{4S}$.
\end{theorem}

Sparsity of $f_0$ is encoded in the rate of decay of $\| f_0 - f_{0,S} \|_{\ell_1}$ as $S \to \infty$.

The problem of linking back the restricted isometry property to incoherence of rows was perhaps first considered by Cand\`{e}s and Tao in \cite{candes-near-optimal}. In this paper we will use the following refinement due to Rudelson and Vershynin \cite{rudelson-sparse}. It relates the allowable value of $S$ to the number $K$ of measurements and the dimension $N$ of the vector to be recovered.

\begin{theorem}\label{teo:RV}
Let $A$ be a $N$-by-$N$ orthogonal matrix, and denote $\mu = \sqrt{N} \max_{ij} |A_{ij}|$. Extract $\Phi$ from $A$ by selecting $K$ rows uniformly at random. For every $\epsilon$ there exists a constant $C_\epsilon > 0$ such that if
\begin{equation}\label{eq:Kgeq}
K \geq C_\epsilon \cdot \mu^2 \cdot S \log N \cdot \log^2(S) \log(S \log N)),
\end{equation}
then $\Phi$ obeys the $S$-restricted isometry property (\ref{eq-defn-rih}) with $\delta _S < \epsilon$ and with very high probability.
\end{theorem}

Randomness plays a key role in this theorem; no instance of a corresponding deterministic result is currently known. Above, ``very high probability" means tending to $1$ as $O(N^{-m})$, where $m$ is an affine increasing function of $C$ in (\ref{eq:Kgeq}). The proof of Theorem \ref{teo:RV} in \cite{rudelson-sparse} uses arguments of geometric functional analysis and theory of probability in Banach spaces. Note that \cite{rudelson-sparse} prove the theorem for $\mu = 1$, but it is straightforward to keep track of the scaling by $\mu$ in their argument\footnote{For the interested reader, \cite{rudelson-sparse} denotes $\mu$ by $K$. The scalings of some important quantities for the argument in \cite{rudelson-sparse}, in their notations, are $C_{10} \sim K$, $E_1 \sim K$, $k_1 \sim K$, and $k \sim K^2$.}.


\subsection{Discretization}
\label{sec:disc}

Equations (\ref{eq:errorCS}) and (\ref{eq:Kgeq}) formally appear to be related to the claims made in Section \ref{sec:why}, but we have yet to bridge the gap with the wave equation and its discretization.

Let us set up a consistent finite element discretization of (\ref{eq:wave}) and (\ref{eq:Helmholtz}), although it is clear that this choice is unessential. The Helmholtz equation with boundary conditions can be approximated as
\[
\sum_j L_{ij} v_{\tilde{\omega}}[j] + \tilde{\omega}^2 \sum_{j} M_{ij} v_{\tilde{\omega}}[j] = 0,
\]
where the stiffness matrix $L$ (a discretization of the second derivative) and the mass matrix $M$ properly include the boundary conditions. Square brackets indicate integer indices. Here $\tilde{\omega}$ is the discrete counterpart to the exact $\omega$, and has multiplicity one, like $\omega$, by the assumption of faithful discretization (Definition \ref{def:faith}). The discrete eigenvector $v_{\tilde{\omega}}[j]$ is an approximation of the eigenfunction samples $v_\omega(j/N)$.


We assume for convenience that the mass matrix can be lumped, and that we can write $M_{ij} = \sigmai^2 \delta_{ij}$. If $\sigma$ is smooth, then $\sigmai \simeq \sigma(i/N)$; and if $\sigma$ lacks the proper smoothness for mass lumping to be a reasonable operation, then all the results of this paper hold with minimal modifications. We also assume that $L_{ij}$ is symmetric negative definite; as a result, so is $\sigmai^{-1} L_{ij} \sigmaj^{-1}$, and the usual conclusions of spectral theory hold: $\tilde{\omega} < 0$, and $\sigmaj v_{\tilde{\omega}}[j]$ are orthogonal for different $\tilde{\omega}$, for the usual dot product in $j$.

The corresponding semi-discrete wave equation (\ref{eq:wave}) is
\[
- \frac{d^2 u[i]}{dt^2}(t) + \sigmai^{-2} \sum_j L_{ij} u[j](t) = 0,
\]
\[
u[j](0) = u_0(j/N), \qquad \frac{d u[j]}{dt}(0) = u_1(j/N)
\]
Because $L_{ij}$ is the same stiffness matrix as above, the solution is
\[
u[j](t) = \sum_{\tilde{\omega}} \left( \, \cos (\tilde{\omega} t) c_{0,\tilde{\omega}} + \frac{\sin (\tilde{\omega} t)}{\tilde{\omega}}  c_{1,\tilde{\omega}} \, \right) v_{\tilde{\omega}}[j].
\]
with
\[
c_{0,\tilde{\omega}} = \sum_j \sigmaj^2  v_{\tilde{\omega}}[j] \; u[j](0), \qquad c_{1,\tilde{\omega}} = \sum_j \sigmaj^2  v_{\tilde{\omega}}[j] \; \frac{d u[j]}{dt}(0).
\]
Call
\[
c_{\tilde{\omega}}(t) = \cos (\tilde{\omega} t) c_{0,\tilde{\omega}} + \frac{\sin (\tilde{\omega} t)}{\tilde{\omega}}  c_{1,\tilde{\omega}}.
\]
A discretization error is incurred at time $t > 0$: we have already baptized it $\tau = \| u[j](t) - u(j/N,t) \|_2$ in the introduction. It is not the purpose of this paper to relate $\tau$ to the grid spacing. A nice reference for the construction of finite elements for the wave equation, in a setting far generalizing the smoothness assumptions made on $\sigma$ in this paper, is in \cite{OZ}.

Let us now focus on the discrete formulation. The ideal $\ell_1$ problem that we would like to solve is
\[
\min \sum_j \sigmaj |u[j](t)|, \qquad \mbox{s.t.} \qquad \sum_j \sigmaj^2 u[j](t) \, v_{\tilde{\omega}}[j] = c_{\tilde{\omega}}(t),
\]
where $\tilde{\omega} \in \tilde{\Omega_K}$ are chosen uniformly at random.

In practice, we must however contend with the error in computing the discrete eigenvalues $\tilde{\omega}$ and eigenvectors $v_{\tilde{\omega}}[j]$ by an iterative linear algebra method. This affects the value of $c_{\tilde{\omega}}(t)$ as well as the measurement vectors in the equality constraints above. We model these errors by introducing the computed quantities $\tilde{v}_{\tilde{\omega}}[j]$ and $\tilde{c}_{\tilde{\omega}}(t)$, and relaxing the problem to
\begin{equation}\label{eq:ell1}
\min \sum_j \sigmaj |u[j](t)|, \qquad \mbox{s.t.} \qquad \| \sum_j \sigmaj^2 u[j](t) \, \tilde{v}_{\tilde{\omega}}[j] - \tilde{c}_{\tilde{\omega}}(t) \|_2 \leq \eps,
\end{equation}
for some adequate\footnote{Here $\eps$ is assumed to be known, but in practice it is permitted to over-estimate it.}, $\eps$, like mentioned throughout in Section \ref{sec:CS}. The $\ell_2$ norm is here over $\tilde{\omega}$.

\subsection{Proof of Theorem \ref{teo:main}}\label{sec:proof}

Let us now explain why the machinery of Section \ref{sec:CS} can be applied to guarantee recovery in the $\ell_1$ problem (\ref{eq:ell1}).  First, it is convenient to view the unknown vector to be recovered in (\ref{eq:ell1}) as $\sigmaj |u[j](t)|$, and not simply $u[j](t)$. This way, we are exactly in the situation of Theorem \ref{teo:stab}: the objective is a non-weighted $\ell_1$ norm, and the measurement vectors $A_{\tilde{\omega}, j} \equiv \sigmaj v_{\tilde{\omega}}[j]$ are orthogonal with respect to the usual dot product in $j$.

Sparsity of the discrete solution $\sigmaj u[j](t)$ is measured by the number of samples $S_\eta$ necessary to represent it up to accuracy $\eta$ in the $\ell_1$ sense, as in equation (\ref{eq:Seta}). If we let $u^S[j](t)$ be the approximation of $u[j](t)$ where only the $S$ largest entries in magnitude are kept, and the others put to zero, then $S_\eta$ is alternatively characterized as the smallest integer $S$ such that
\[
\sum_j \sigmaj | u[j](t)-u^{S}[j](t) |  \leq \eta.
\]
This expression is meant to play the role of the term $\| f_0 - f_{0,S} \|_{\ell_1}$ in equation (\ref{eq:errorCS}).

One last piece of the puzzle is still missing before we can apply Theorems \ref{teo:stab} and \ref{teo:RV}: as explained earlier, the methods for drawing eigenvalues at random do not produce a uniform distribution. The following proposition quantifies the effect of nonuniformity of the random sampling on the number of measurements $K$ in Theorem \ref{teo:RV}. Recall that our method for drawing eigenvectors qualifies as \emph{without replacement} since an eigenvalue can only appear once in the set $\Omega_K$.

\begin{proposition}\label{teo:nonuniform}
Let $A$ be an $N$-by-$N$ orthogonal matrix. Denote by $\Kunif$ the number of rows taken uniformly at random in order to satisfy the accuracy estimate (\ref{eq:accuracy}) with some choice of the constants $C_1$, $C_2$. Now set up a sampling scheme, nonuniform and without replacement for the rows of $A$, as follows:
\begin{itemize}
\item In the first step, draw one row from a distribution $p_n$, and call it $n_1$;
\item At step $k$, draw one row from the obviously rescaled distribution
\[
\frac{p_n}{1- \sum_{j=1}^{k-1} p_{n_j}}
\]
and call it $n_k$. Repeat over $k$.
\end{itemize}
Denote by $K$ the number of rows taken from this sampling scheme. In order to satisfy the estimate (\ref{eq:accuracy}) with the same constants $C_1$ and $C_2$ as in the uniform case, it is sufficient that $K$ compares to $\Kunif$ as
\[
K \geq \frac{\punif}{\min_{n = 1, \ldots, N} p_n} \, \Kunif.
\]
where $\punif$ would be the counterpart of $p_n$ in the uniform case, i.e., $\punif = 1/N$.

\end{proposition}

\begin{proof}
See the Appendix.
\end{proof}

We can now apply the theorems of compressed sensing. Call $\tilde{u}[j](t)$ the solution of (\ref{eq:ell1}). By Theorem \ref{teo:stab}, the reconstruction error is
\begin{align*}
\| \tilde{u}[j](t) - u(j/N,t) \|_2 &\leq \| u[j](t) - u(j/N,t) \|_2 + \| \tilde{u}[j](t) - u[j](t) \|_2 \\
&\leq \qquad\qquad\;\, \tau \qquad\qquad\;\, + \;\;\; C_1 \frac{\eta}{\sqrt{S_\eta}} + C_2 \eps.
\end{align*}
The number $K$ of eigenvectors needed to obtain this level of accuracy with very high probability is given by a combination of Theorem \ref{teo:RV} and Proposition \ref{teo:nonuniform}:
\begin{equation}\label{eq:Kinterm}
K \geq C \cdot \left( \frac{\punif}{\min p_n} \cdot \mu^2 \right) \cdot S_\eta \log N \cdot \log^2(S_\eta) \log(S_\eta \log N)),
\end{equation}
and where the incoherence $\mu$ is given by
\begin{equation}\label{eq:incoherence-mu}
\mu = \sqrt{N} \max_{\tilde{\omega}, j} |\sigmaj v_{\tilde{\omega}}[j]|
\end{equation}
This justifies (\ref{eq:boundK}) and (\ref{eq:accuracy}) with the particular value
\begin{equation}\label{eq:Csigma2}
C(\sigma) = C \cdot \left( \frac{\punif}{\min p_n} \cdot \mu^2 \right).
\end{equation}

It remains therefore to justify the link between $C(\sigma)$ and the smoothness of $\sigma$, given in equation (\ref{eq:Csigma}). This entails showing that
\begin{enumerate}
\item $\mu$ has a bound independent of $N$ (eigenvectors are incoherent); and
\item the ratio of probabilities can be bounded away from zero independently of $N$.
\end{enumerate}

In addition, we would like to argue that $S_\eta$ can be much smaller than $N$, which expresses sparsity. All three questions are answered through estimates about the wave equation, which are possibly new. Although the setup of the recovery algorithm is fully discrete, it is sufficient to focus on estimates the non-discretized wave equation; we explain below why this follows from the assumption of faithful discretization in Definition \ref{def:faith}.

\subsection{Sparsity, Incoherence, and Randomness}\label{sec:proof2}

We address these points in order.

\begin{itemize}

\item \emph{Sparsity}. The quantity $S_\eta$ introduced above depends on time $t$ and measures the number of samples needed to represent the discrete solution $u[j](t)$ to accuracy $\eta$ in (a weighted) $\ell_1$ space. It is hoped that by choosing sparse initial conditions, i.e., with small $S_\eta$ at time $t = 0$, and restricting the time $T$ up to which the solution is computed, $S_\eta$ will remain small for all times $0 \leq t \leq T$. If we expect to have such control over $S_\eta$, it is necessary to first show that the $\ell_1$ norm of the solution itself does not blow up in time, and can be majorized from the $\ell_1$ norm at time zero.  In Section \ref{sec:sparsity} we establish precisely this property, but for the continuous wave equation and in the $L^1$ norm. The only condition required on $\sigma$ for such an estimate to hold is Var$(\log \sigma) < 1$.

\item \emph{Incoherence}. We wish to bound $\mu = \sqrt{N} \max |\sigmaj v_{\tilde{\omega}}[j]|$ by a quantity independent of $N$ when the $v_{\tilde{\omega}}[j]$ are $\ell_2$ normalized in a weighted norm,
\[
\sum_{j=1}^N \sigmaj^2 |v_{\tilde{\omega}}[j]|^2 = 1.
\]
In Section \ref{sec:incoherence}, we show that provided Var$(\log \sigma) < \infty$, we have the continuous estimate
\[
\| \sigma v_\omega \|_{L^\infty} \leq \sqrt{2} \, \exp ( \mbox{Var} (\log \sigma)) \cdot \| v_\omega \|_{L^2_{\sigma^2}}.
\]
At the discrete level, approximating the integral $\int \sigma^2(x) |v_\omega(x)|^2 \, dx$ by the sum $\frac{1}{N} \sum_{j=1}^N \sigmaj^2 |v_\omega[j]|^2$ shows that the quantity independent of $N$ is the incoherence (\ref{eq:incoherence-mu}). More precisely, the assumption of faithful discretization allows to relate continuous and discrete norms, and conclude that
\[
\mu \leq C \cdot \sqrt{N} \| \sigma v_\omega \|_{L^\infty} \frac{\sqrt{\frac{1}{N} \sum_j \sigmaj^2 \tilde{v}^2_{\tilde{\omega}}[j]}}{\| v_\omega \|_{L^2_{\sigma^2}}} \leq C \cdot \sqrt{2} \, \exp ( \mbox{Var} (\log \sigma)),
\]
where $C$ is the numerical constant that accounts for the faithfulness of the discretization, here $C = 8$ for the arbitrary choice we made in Definition \ref{def:faith}.


\item \emph{Randomness}. Finally, for the question of characterizing the probability distribution for picking eigenvectors, recall that it derives from the strategy of picking shifts $w$ at random and then finding the eigenvalue $\lambda_{[k]} = -\omega^2_{[k]}$ such that $\omega_{[k]}$ is closest to $w$. To make sense of a probability distribution over shifts, we consider eigenvalues for the continuous Helmholtz equation (\ref{eq:Helmholtz}) in some large interval $[-W^2, 0]$, or equivalently, $0 \leq k < N$.


If $\sigma = 1$, then $\omega_n = n \pi$, with $n \geq 1$ if Dirichlet, and $n \geq 0$ if Neumann. The corresponding eigenfunctions are of course $\cos(n \pi)$ (Neumann) and $\sin(n \pi)$ (Dirichlet). In this case, a uniform sampling of the shifts $w$ would generate a uniform sampling of the frequencies $\omega_{[k]}$. We regard the spectrum in the general case when $\sigma(x) \ne 1$ as a perturbation of $\omega_n = n \pi$. So we still draw the shifts $w$ uniformly at random, and derive the corresponding distribution on the $\omega_{[k]}$ from the spacing between the $\omega_n$. Namely if we let $\Delta \omega_{\min}$, $\Delta \omega_{\max}$ be the minimum, respectively maximum distance between two consecutive $\omega_n$ in the interval $n \in [0, N-1]$---each eigenvalue is known to have multiplicity one---then the probability of picking any given eigenvalue by this scheme obeys
\[
p_n \geq \punif \, \frac{\Delta \omega_{\min}}{\Delta \omega_{\max}},
\]
where $\punif = \frac{1}{N}$ would be the corresponding probability in the uniform case.

In Section \ref{sec:egv-gaps}, we prove that the spacing between any two consecutive $\omega_n$ obeys
\[
\frac{\pi - \mbox{Var}(\log \sigma)}{\int_0^1 \sigma(x) \, dx} \leq |\omega_{n} - \omega_{n+1}| \leq \frac{\pi + \mbox{Var}(\log \sigma)}{\int_0^1 \sigma(x) \, dx},
\]
provided Var$(\log \sigma) < \infty$. By the assumption made in Definition \ref{def:faith}, the computed eigenvalues satisfy a comparable bound,
\[
\frac{1}{2} \frac{\pi - \mbox{Var}(\log \sigma)}{\int_0^1 \sigma(x) \, dx} \leq |\tilde{\omega}_n - \tilde{\omega}_{n+1}| \leq 2 \frac{\pi + \mbox{Var}(\log \sigma)}{\int_0^1 \sigma(x) \, dx}.
\]
In the continuous case, the gap estimate implies
\[
\frac{p_n}{\punif} \geq \frac{\Delta \omega_{\min}}{\Delta \omega_{\max}} \geq \frac{\pi - \mbox{Var}(\log \sigma)}{\pi + \mbox{Var}(\log \sigma)},
\]
while an additional factor 1/4 is incurred in this lower bound, in the discrete case.

Note that the probabilities of selecting the endpoint eigenvalues $0$ and $- \omega_{\max}^2$ would in principle be halved by the advocated sampling procedure, because their interval is one-sided. This is a non-issue algorithmically since the endpoint eigenvalues are much more easily computed than the others, by a power method without inversion. By default, we can include those eigenvalues in $\Omega_K$. Mathematically, adding measurements (eigenvectors) deterministically cannot hurt the overall performance as we saw in Proposition \ref{teo:nonuniform}.

\end{itemize}

The observations on incoherence and randomness can be combined with (\ref{eq:Csigma2}) to justify the form of $C(\sigma)$ in equation (\ref{eq:Csigma}):
\[
C(\sigma) \leq C \cdot \left( \frac{\punif}{\min p_n} \cdot \mu^2 \right) \leq C \cdot \frac{\pi + \mbox{Var}(\log \sigma)}{\pi - \mbox{Var}(\log \sigma)} \cdot  \exp ( 2 \mbox{Var} (\log \sigma)).
\]


\section{Sparsity, Incoherence, and Gap Estimates for the Wave Equation}\label{sec:three}





Let us first recall the basic results concerning wave equations. We only need to require $\sigma_0 \leq \sigma(x) \leq \sigma_1$ for a.e. $x \in [0,1]$, to obtain existence and uniqueness. In that context, when $u_0 \in H^1(0,1)$ and $u_1 \in L^2(0,1)$, the solution obeys $u \in C_t^0([0,T],H_x^1(0,1)) \cap C^1_t([0,T], L^2_x(0,1))$, as well as the corresponding estimate
\begin{equation}\label{eq:growth}
\| u(t,x) \|_{L^\infty_t H^1_x} \leq C(T) \cdot \left( \| u_0 \|_{H^1} + \| u_1 \|_{L^2} \right),
\end{equation}
where the supremum in time is taken (here an in the sequel) over $[0,T]$. Another background result is the continuous dependence on the coefficients $\sigma(x)$ in $L^\infty$. Let $u_k$, $k=1,2$ solve (\ref{eq:wave2}) with $\sigma_k$ in place of $\sigma$, over the time interval $[0,T]$. Then
\begin{equation}\label{eq:continuity}
\| u_1 - u_2 \|_{L^\infty_t L^\infty_x} \leq C(T,u_0,u_1) \cdot \| \sigma_1 - \sigma_2 \|_{L^\infty}.
\end{equation}
All these results are proved by standard energy estimates. See \cite{Lions} and \cite{Stolk-thesis} for existence and uniqueness; and \cite{BCL, BCL2} for continuity on the parameters.

It seems that we have lost a bit of generality in considering the wave equation (\ref{eq:wave}) with a single parameter $\sigma(x)$, instead of the usual equation of acoustics
\begin{equation}\label{eq:wave2}
\rho(x) \frac{\pd^2 u}{\pd t^2} + \nabla \cdot ( \mu(x) \nabla u) = 0, \qquad x \in \R^d,
\end{equation}
with the two parameters $\rho(x)$ (density) and $\mu(x)$ (bulk modulus). Equation (\ref{eq:wave}) however follows from (\ref{eq:wave2}) if we change the unique spatial variable $z$ into
\[
x = \int_0^z \mu^{-1}(z') dz',
\]
and put $\sigma(x) = \sqrt{\rho(x) \mu(x)}$ the local acoustic impedance. With $\mu$ bounded from above and below a.e., such a change of variables would only alter the constants in the results of sparsity and incoherence proved in this section. It would not alter the eigenvalue gap result of Section \ref{sec:egv-gaps}. Hence we can focus on (\ref{eq:wave}) without loss of generality. For reference, the local speed of sound is $v(x) = \sqrt{\mu(x)/\rho(x)}$.\footnote{Notice that a homeomorphism $z \mapsto x$ is the natural obstruction to the 1D inverse problem of recovering the parameters $\rho$ and $\mu$ from boundary measurements of $u$. As a result only the local impedance is recoverable up to homeomorphism from boundary measurements, not the local speed of sound. This observation, only valid in one spatial dimension, is discussed in great detail in \cite{BCL}.}

In the sequel we further assume that $\log \sigma$ has \emph{bounded variation}. The space $BV([0,1])$ is introduced by defining the seminorm
\[
\mbox{Var}(f)  = \sup_{\{ x_j \}} \sum_{j} | f(x_{j-1}) - f(x_j)|,
\]
where the supremum is over all finite partitions of $[0,1]$ such that $x_j$ is a point of approximate continuity of $f$. Var is called the essential variation\footnote{The variation of $f$ would be defined using the same supremum, but over \emph{all} partitions of $[0,1]$. Requiring that $x_j$ is a point of approximate continuity addresses the problem of inessential discontinuities of where $f(x)$ is not in the interval defined by the left and right limits $f(x^{-})$ and $f(x^{+})$. See \cite{Zie} on p.227 or \cite{DVL} on p.17 for a comprehensive discussion.} of $f$, and also equals the total variation $|f'|([0,1])$ of $f'$ as a signed measure on $[0,1]$. The norm of $BV([0,1])$ is then $\| \sigma \|_{BV} = \| \sigma \|_{L^1} + \mbox{Var}(\sigma)$. If the total variation is taken over the interval $[0,x]$ instead, we will denote it as Var$_x(\sigma)$.

We will need the following result for $BV$ functions in one dimension.

\begin{lemma}\label{teo:BV}
Let $f \in BV([0,1])$, and extend it outside of $[0,1]$ by the constant values $f(0^{+})$ and $f(1^{-})$ respectively. For every $\eps > 0$, consider a mollifier $\rho_\eps(x) = \frac{1}{\eps} \rho \left( \frac{x}{\eps} \right)$, where $\rho \in C^\infty(\R)$, $\rho \geq 0$, supp$(\rho) \subset [-1,1]$, and $\int_{\R} \rho(x) \, dx = 1$. Consider $f_\eps(x) = \int_{\R} \rho_\eps(x-y) f(y) \, dy$ for $x \in \R$. Then

\begin{enumerate}
\item For each $\eps > 0$, $f_\eps \in C^\infty(\R)$;
\item $\lim_{\eps \to 0} \int_0^1 | f(x) - f_\eps(x)| \, dx = 0$;
\item For each $\eps > 0$ and $x \in \R$, $\min_{x \in [0,1]} f(x) \leq f_\eps(x) \leq \max_{x \in [0,1]} f(x)$;
\item For each $\eps > 0$, $\mbox{Var}(f_\eps) = \int_0^1 |f'_\eps(x)| \, dx$;
\item $\lim_{\eps \to 0} \mbox{Var}(f_\eps) = \mbox{Var}(f)$.
\end{enumerate}
\end{lemma}

\begin{proof}
All these facts are proved in \cite{Zie}: points 1 and 2 on p.22, point 3 by elementary majorations (see also p.22), point 4 on p.227, and point 5 on p.225.
\end{proof}


\subsection{Analysis of Sparsity}\label{sec:sparsity}

In this section we prove the following $L^1$ estimate.

\begin{theorem}\label{teo:sparsity}
Let $u \in C_t^0([0,T],H_x^1(0,1)) \cap C^1_t([0,T], L^2_x(0,1))$ solve (\ref{eq:wave2}) with Dirichlet or Neumann boundary conditions, and let $U_1(x) = \int_0^x u_1(y) \, dy$. Assume $\log \sigma \in BV([0,1])$, with
\[
\mbox{\emph{Var}}(\log \sigma) < 1.
\]
Let $t^\sharp = 1/\sigma_{\min}$, a lower bound on the time it takes a fully transmitted bump to travel the length of the interval $[0,1]$. For each $t > 0$, let $n$ be the smallest integer such that $t \leq n t^\sharp$; then
\begin{equation}\label{eq:sparse-Strichartz}
\| u(\cdot,t) \|_{L^1} \leq 2 \left( \frac{\sigma_{\max}}{\sigma_{\min}} \right)^{3/2} \cdot D^n \cdot \left( \| u_0 \|_{L^1}  + \| \sigma U_1 \|_{L^1} \right),
\end{equation}
with $D = \frac{1}{1 - \scriptsize\emph{Var}(\log \sigma)}$. If instead (\ref{eq:wave2}) is posed with periodic boundary conditions, then it suffices that Var$(\log \sigma) < 2$ and the same result holds with $D = \frac{1}{1 - \frac{1}{2} \scriptsize\emph{Var}(\log \sigma)}$.

\end{theorem}

This result calls for a few remarks, which are best expressed after the proof is complete.

\begin{proof}
The proof is divided into five steps.

First, we approximate $\sigma \in BV$ by an $M$-term piecewise constant function $\sigma_{\mbox{pc}}$. Assuming the inequality holds for $\sigma_{\mbox{pc}}$, we then show how to pass to the limit $M \to \infty$.  Second, in order to show the inequality for $\sigma_{\mbox{pc}}$, we approximate the solution $u(\cdot, t)$ by a piecewise interpolant on an equispaced $N$-point grid. This step allows to break up the solution into localized pulses that interact with one discontinuity of the medium at a time. Third, we recall the formulation of reflection and transmission of pulses at interfaces of a piecewise constant medium, and simplify it using a model in which cancellations are absent. This simplification provides an upper bound for a particular weighted $L^1$ norm of the wavefield. Fourth, we present a recursive bump splitting procedure for handling the exponential number of scattering events and homogenizing the corresponding traveltimes. Fifth, we quantify the growth of the weighted $L^1$ norm in terms of combinations of reflection and transmission coefficients. In particular, sums of reflection coefficients are linked back to the total variation of $\log \sigma(x)$.

Let us tackle these points in order.

\begin{enumerate}

\item The result of adaptive $L^\infty$ approximation of $BV$ functions by piecewise constants in one dimension is due to Kahane and goes as follows. For a given partition $\{ x_j ; j = 0, \ldots, M \}$ of $[0,1]$, where $x_0 = 0$ and $x_{M} = 1$, we define an $M$-term approximant as
\begin{equation}\label{eq:Mtermsuperposition}
\sigma_{\mbox{pc}}(x) = \sum_{j = 1}^{M} \sigma_j \chi_{[ x_{j-1}, x_j)}(x).
\end{equation}
Denote  by $\Sigma_M$ the set of all such approximants, i.e., all the possible choices of a partition $\{ x_j \}$ and coefficients $\sigma_j$. Then Kahane's inequality states that
\begin{equation}\label{eq:Kahane}
\inf_{\sigma_{\scriptsize\mbox{pc}} \in \Sigma_M} \| \sigma_{\mbox{pc}} - \sigma \|_{L^\infty} \leq \frac{\mbox{Var}(\sigma)}{2 M}.
\end{equation}
There exists at least one approximant $\sigma_{\mbox{pc}}$ that reaches this bound, and that also satisfies $\mbox{Var}(\sigma_{\mbox{pc}}) \leq \mbox{Var}(\sigma)$. See the nice review article \cite{DeV} by Ron DeVore for a proof.

We saw earlier in equation (\ref{eq:continuity}) that the solution depends continuously on $\sigma$. The same bound holds, trivially, if we use the weaker $L^1$ norm for $u$:
\[
\| u_1 - u_2 \|_{L^\infty_t L^1_x} \leq C(T,u_0,u_1) \cdot \| \sigma_1 - \sigma_2 \|_{L^\infty},
\]
This inequality provides a way of passing to the limit $M \to \infty$ in (\ref{eq:sparse-Strichartz}), provided the constant $D$ in (\ref{eq:sparse-Strichartz}) is shown to be uniform in $M$.

\item Let us now decompose the initial conditions into localized bumps up to a controlled error.  Since we need $u_0 \in H^1(0,1)$ and $u_1 \in L^2(0,1)$ for consistency with the basic theory, we can approximate both $u_0$ and $U_1 = \int^x u_1$ by piecewise linear interpolants built from equispaced samples. The Bramble-Hilbert lemma handles the question of accuracy of polynomial interpolants and is for instance well covered in basic texts on finite elements.\footnote{Tent elements are only a mathematical tool in this section, they are not used in the numerical method. Also note that we are not considering a full discretization here; only the initial conditions are modified.}

\begin{lemma}\label{teo:BH}(Bramble-Hilbert)
Let $v \in H^1(0,1)$. For each positive integer $N$, we let $h = 1/N$ and
\[
v_h(x) = \sum_{i=0}^{N} v(i h) \phi_{i,h}(x),
\]
where $\phi_{i,h}$ are the ``tent" interpolating functions which were defined as
\[
\phi_{0,N}(x) = (1-Nx) \chi_{0 \leq x \leq h}(x), \qquad \phi_{N,N}(x) = (1+N(x-1)) \chi_{1-h < x \leq 1}(x),
\]
\[
\phi_{j,N}(x) = (1+N(x-jh)) \chi_{(j-1)h \leq x \leq jh}(x) + (1-N(x-jh)) \chi_{jh < x \leq (j+1)h}(x),
\]
for $1 \leq j \leq N-1$. Then
\[
\| v - v_h \|_{H^1} \leq C \cdot \| v \|_{H^1} \cdot h, \qquad \mbox{and} \qquad \| v - v_h \|_{L^2} \leq C \cdot \| v \|_{H^1} \cdot h^2.
\]
\end{lemma}

Consider now the solution $u_h(\cdot, t)$ of (\ref{eq:wave2}) with the piecewise linear interpolants $u_{0,h}$ and $U_{1,h}$ substituted for the initial conditions $u_0$ and $U_1$. The growth estimate (\ref{eq:growth}) provides a way to control the discrepancy $(u_h - u)(\cdot, t)$. With $L^\infty_t$ denoting $L^\infty(0,T)$, we have
\begin{align*}
\| u - u_h \|_{L^\infty_t L^1_x} &\leq \| u - u_h \|_{L^\infty_t H^1_x} \\
&\leq C(T) \cdot \left(\| u_0 - u_{0,h} \|_{H^1} + \| U_1 - U_{1,h} \|_{H^1} \right) \qquad \mbox{by }(\ref{eq:growth}) \\
&\leq C(T) \cdot \left(\| u_0 \|_{H^1} + \| U_1 \|_{H^1} \right) \cdot h \qquad\qquad \mbox{by Lemma } \ref{teo:BH},
\end{align*}
which tends to zero as $h \to 0$.


It is sufficient to argue that (\ref{eq:sparse-Strichartz}) holds for localized initial conditions of the type $\phi_{i,h}(x)$. Indeed, if we let $u_{0,h} = \sum u_{0,i} \phi_{i,h}(x)$, and denote by $\phi_{i,h}(x,t)$ the solution of the wave equation with $\phi_{i,h}(x)$ in place of $u_0$ and $0$ for $U_1$, then
\begin{align*}
\| u_h(\cdot, t) \|_{L^1} &\leq \left( \sum_i |u_{0,i}| \right) \cdot \| \phi_{i,h}(\cdot,t) \|_{L^1} \\
&\leq \left( \sum_i |u_{0,i}| \right) \cdot D \| \phi_{i,h} \|_{L^1} \\
&\leq 2 D \cdot \| u_{0,h} \|_{L^1}.
\end{align*}
The last inequality is a simple $L^1$---$\ell_1$ equivalence property for piecewise affine functions on equispaced grids. The corresponding inequality for $u(x,t)$ and $u_0$ then follows after taking the limit $h \to 0$, provided we show that $D$ is independent of $h$.

Next, we explain how to control the $L^1$ norm of traveling bump solutions $u_h(\cdot, t)$.

\item Equation (\ref{eq:wave2}) has an explicit solution when $\sigma_{\scriptsize\mbox{pc}}$ is piecewise constant. Let us start by rehearsing the textbook case of a medium with impedance $\sigma_1$ for $x < x_0$, and $\sigma_2 \ne \sigma_1$ for $x \geq x_0$. Assume that both $u_0$ and $U_1 = \int^x u_1$ are supported on $\{ x: x < 0 \}$, and form the linear combinations
\[
f(x) = u_0(x) - \sigma_1 U_1(x), \qquad g(x) = u_0(x) + \sigma_1 U_1(x).
\]
For small times, $f$ will give rise to right-going waves and $g$ to left-going waves. Without loss of generality, let us assume $g = 0$. Then the solution is given by
\begin{equation}\label{eq:reflection}
u(x,t) = \left\{ \begin{array}{ll}
    f(x-t / \sigma_1) + R f(2 x_0 -x- t / \sigma_1) & \mbox{if } x \leq x_0; \\
    T f(x_0 + \frac{\sigma_2}{\sigma_1}(x - x_0 - t / \sigma_2)) & \mbox{if } x > x_0. \end{array} \right.
\end{equation}
In order for both $u$ and $\frac{\pd u}{\pd x}$ to be continuous at $x = 0$, we need to impose that the reflection and transmission coefficients be determined as
\[
R = \frac{1 - \sigma_2/\sigma_1}{1 + \sigma_2 / \sigma_1}, \qquad T = \frac{2}{1 + \sigma_2/\sigma_1}.
\]
Note that the situation of a wave, initially supported on $\{ x: x > 0\}$, and reflecting at the interface from the right, is entirely analogous. The reflection and transmission coefficients would be obtained by interchanging $\sigma_1$ and $\sigma_2$. Let us call the situation described by (\ref{eq:reflection}) a \emph{single scattering event}.

In order to study the growth of $\int |u(\cdot,t)| dx$, it is important to remove the cancellations that occur in (\ref{eq:reflection}) when $R < 0$. For this purpose, decouple wavefields as
\begin{equation}\label{eq:reflection1}
u_{1}(x,t) = \left\{ \begin{array}{ll}
    f(x-t / \sigma_1) & \mbox{if } x \leq x_0; \\
    T f(x_0 + \frac{\sigma_2}{\sigma_1}(x - x_0 - t / \sigma_2)) & \mbox{if } x > x_0. \end{array} \right.
\end{equation}
\begin{equation}\label{eq:reflection2}
u_{2}(x,t) = \left\{ \begin{array}{ll}
    R f(2 x_0 -x- t / \sigma_1) & \mbox{if } x \leq x_0; \\
    0 & \mbox{if } x > x_0. \end{array} \right.
\end{equation}
We have $u = u_1 + u_2$, but the point of introducing the couple $(u_1,u_2)$ is that there is one particular weighted $L^1$ norm that never decreases in time, namely the $L_{\sigma^{3/2}}^1$ norm defined as
\[
||| v ||| := \int_0^1 |v(x)| \, (\sigma(x))^{3/2} \,  dx, \qquad ||| (v_1,v_2) ||| := ||| v_1 ||| + ||| v_2 |||.
\]
Indeed, $||| u(\cdot, 0) ||| = ||| f |||$ in $\sigma_1$, and it is straightforward to write
\[
||| u_1(\cdot,t) ||| + ||| u_2 (\cdot, t) ||| =  \int_{-\infty}^{-t/\sigma_1} |f(x)| \sigma_1^{3/2} dx + \left[ |R| + \left( \frac{\sigma_2}{\sigma_1} \right)^{1/2} T \right] \int_{-t/\sigma_1}^{\infty} |f(x)| \sigma_1^{3/2} dx.
\]
Since $R^2 + \frac{\sigma_2}{\sigma_1} T^2 = 1$ (conservation of energy), taking the square root of each term in this convex combination yields
\[
|R| + \left( \frac{\sigma_2}{\sigma_1} \right)^{1/2} T \geq 1.
\]
It is convenient to write $\tilde{T} = \left( \frac{\sigma_2}{\sigma_1} \right)^{1/2} T$. Upper and lower bounds for the $L^1_{\sigma^{3/2}}$ norm of the couple $(u_1,u_2)$ follow:
\begin{equation}\label{eq:upper-lower}
||| f ||| \leq ||| u_1(\cdot,t) ||| + ||| u_2 (\cdot, t) ||| \leq \left[ |R| + \tilde{T} \right] \cdot ||| f |||.
\end{equation}


\item We can treat the more general case of a piecewise constant medium by decomposing the initial conditions $u_0$ and $U_1$ into a collection of small bumps, following the preceding discussion of a single interface. Fix a partition $\{ x_j : 0 \leq j \leq M \}$ of $[0,1]$ defining an $M$-term approximant $\sigma_{\scriptsize\mbox{pc}}$, such that $\sigma_{\scriptsize\mbox{pc}}(x) = \sigma_j$ when $x_{j-1} < x \leq x_{j}$ for $j \geq 1$. Let us choose the parameter $h$ in the construction of the tent functions $\phi_{i,h}$ small enough that scattered bumps intersect with at most one discontinuity of $\spc$ at a time. Since the minimum and maximum traveling speed are $1/\sigma_{\max}$ and $1/\sigma_{\min}$ respectively, it suffices to take
\begin{equation}\label{eq:hbound}
h = \frac{1}{2} \frac{\sigma_{\min}}{\sigma_{\max}} \min_{1 \leq j \leq M} |x_{j-1} - x_{j}|
\end{equation}
Let us generically call $\phi(x)$ such a bump, and assume that it is supported on $(x_{j-1}, x_{j})$ for some $0 \leq j \leq M$. It gives rise to left- and right-going waves.

\begin{itemize}
\item Consider $\phi$ as a right-going initial condition of the wave equation; namely, $u_0 = \phi/2$ and $U_1 = - \phi/(2 \sigma_j)$. Choose a time $t^*$ large enough that the first term $\phi(x - t^*/\sigma_j)$ vanishes in (\ref{eq:reflection}), but small enough that no other scattering event than the one at $x_j$ has taken place yet. Then the solution takes the form $u = u_1 + u_2$, with
\begin{equation}\label{eq:2bumps}
u_1(x,t^*) = [\mathcal{R}_{j,r} \phi](x + t^* / \sigma_j), \qquad u_2(x,t^*) = [\mathcal{T}_{j,r} \phi] (x - t^* / \sigma_{j+1}),
\end{equation}
where we have introduced reflection and transmission operators
\[
[\mathcal{R}_{j,r}\phi] (x) = R_{j,r} \phi(2 x_j -x), \qquad [ \mathcal{T}_{j,r}\phi](x) = T_{j,r} \phi(x_j + \frac{\sigma_{j+1}}{\sigma_{j}}(x-x_j)),
\]
with
\[
R_{j,r} = \frac{1 - \sigma_{j+1}/\sigma_{j}}{1 + \sigma_{j+1} / \sigma_{j}}, \qquad T_{j,r} = \frac{2}{1 + \sigma_{j+1}/\sigma_{j}}.
\]
The subscript $r$ refers to the fact that the pulse $\phi$ came from the right.
\item If instead $\phi$ had been chosen to correspond to a left-going bump in $(x_{j-1}, x_j)$, then we would have had
\[
u_1(x,t^*) = [\mathcal{R}_{j,\ell} \phi](x - t^* / \sigma_j), \qquad u_2(x,t^*) = [\mathcal{T}_{j,r} \phi] (x + t^* / \sigma_{j+1}),
\]
where now
\[
[\mathcal{R}_{j,\ell}\phi] (x) = R_{j,\ell} \phi(2x_j -x), \qquad [ \mathcal{T}_{j,\ell}\phi](x) = T_{j,\ell} \phi(x_j + \frac{\sigma_{j-1}}{\sigma_{j}}(x-x_j)),
\]
and
\[
R_{j,\ell} = \frac{1 - \sigma_{j-1}/\sigma_{j}}{1 + \sigma_{j-1} / \sigma_{j}}, \qquad T_{j,\ell} = \frac{2}{1 + \sigma_{j-1}/\sigma_{j}}.
\]
\end{itemize}


In both cases, the original bump disappears at $t = t^*$ and give rise to a couple $(u_1, u_2)$. In the regime of multiple scattering when $t > t^*$, bumps are \emph{recursively split and removed}, using the above characterization. For instance, at time $t^*$, we can restart the wave equation from the left-going bump $[\mathcal{R}_{j,r} \phi](x + t^* / \sigma_j)$ and submit it to scattering at $x_{j-1}$; and independently consider the right-going bump $[\mathcal{T}_{j,r} \phi] (x - t^* / \sigma_{j+1})$ and its scattering at $x_{j+1}$. Applying this procedure recursively after each scattering event, a binary tree in space-time is created, whose nodes are the scattering events and whose edges are the broken bicharacteristics.

We therefore consider a collection of wavefields, i.e., an element $\mathbf{u} \in B(L^1_{\sigma^{3/2}})$ where $B$ is an unordered set of bumps, equipped with the norm $||| \mathbf{u} ||| = \sum_b ||| u_b |||$, and such that the solution of the wave equation is recovered as $u = \sum_b u_b$. Obviously, $||| u ||| \leq ||| \mathbf{u} |||$.

The upper and lower bounds (\ref{eq:upper-lower}) on the $L_{\sigma^{3/2}}^1$ norm of a singly scattered wavefield can be applied recursively to address the multiple scattering situation. Every reflected bump picks up a factor $R_{j,\ell}$ or $R_{j,r}$ as appropriate, and similarly every transmitted bump picks up a factor $\tilde{T}_{j,r} = \sqrt{\sigma_{j+1}/\sigma_j} \, T_{j,r}$ or $\tilde{T}_{j,\ell} = \sqrt{\sigma_{j-1}/\sigma_j} \, T_{j,\ell}$ as appropriate.


For fixed time $t$, evaluating the number of scatterings that have taken place is an overwhelming combinatorial task. Instead, let $t^{\sharp} = \sigma_{\min} = \min_j \sigma_j$ be a lower bound on the total time it takes a nonreflecting bump to travel the interval $(0,1)$.

Consider now the case of periodic boundary conditions. Dirichlet and Neumann boundary conditions will be treated in point 6 below. Define the extended medium
\[
\sigma_{\scriptsize\mbox{ext}}(x) = \left\{ \begin{array}{ll}
        \spc(0) & \mbox{if $x \leq -1$};\\
        \spc(x+1) & \mbox{if $-1 < x \leq 0$};\\
        \spc(x) & \mbox{if $0< x \leq 1$};\\
        \spc(x-1) & \mbox{if $1< x \leq 2$};\\
        \spc(1) & \mbox{if $x \geq 2$}.\end{array} \right.
\]
With the same initial conditions supported inside $[0,1]$, any scattering that takes place in $\spc$ within the time interval $[0,t^\sharp]$ would also take place in $\sigma_{\scriptsize\mbox{ext}}$ within the same time interval. Since by equation (\ref{eq:upper-lower}) the $L^1_{\sigma^{3/2}}$ norm $||| \mathbf{u}(\cdot, t) |||$ always increases during and after scattering, it is safe to bound $||| \mathbf{u}(\cdot, t) |||$ for $t \leq t^\sharp$ by the $L^1_{\sigma^{3/2}}$ norm of the wavefield in the medium $\sigma_{\scriptsize\mbox{ext}}$ for times $t \geq t^\sharp$, in particular $t \to \infty$. (This reasoning is why we needed the lower bound in (\ref{eq:upper-lower}), hence the introduction of the special weighted norm.)

\item  Finally, let us now bound the $L^1_{\sigma^{3/2}}$ norm of the recursively split wavefield $\mathbf{u}$ in the medium $\sigma_{\scriptsize\mbox{ext}}$, and show that it converges to a bounded limit when $t \to \infty$.

Again, assume without loss of generality that the initial bump $\phi$ is right-going, and supported in some interval $[x_{k-1}, x_k] \subset [0,1]$. We extend the partition $\{ x_j \}$ of $[0,1]$ into the partition $\{ x_j - 1 \} \cup \{ x_j \} \cup \{ x_j+1 \}$ of $[-1,2]$, indexed by the single parameter  $-M+1 \leq j \leq 2M$ in the obvious way.

We can identify various contributions to the bound on $||| \mathbf{u} |||$:
\begin{itemize}
\item The fully transmitted bump, with magnitude $\prod_{j = k}^{2M-1} \tilde{T}_{j,r} \leq 1$.

\item The bumps undergoing one reflection, with combined magnitude less than
\[
\sum_{i_1 = k}^{k+M-1} \, (\prod_{k \leq j_1 < i_1} \tilde{T}_{j_1,r}) \, |R_{i_1,\ell}| \, ( \prod_{j_2 < i_1} \tilde{T}_{j_2,\ell} ) \leq \sum_{i = 0}^{M-1} R_{i,\ell}.
\]
\item The bumps undergoing two reflections, with combined magnitude less than
\[
\sum_{i_1 = k}^{k+M-1} \sum_{i_2 = i_1 - M}^{i_1-1} \, (\prod_{k \leq j_1 < i_1} \tilde{T}_{j_1,r}) \, |R_{i_1,\ell}| \, ( \prod_{j_2 < i_1} \tilde{T}_{j_2,\ell} ) \, |R_{i_2,r}| \, (\prod_{j_3 > i_2} \tilde{T}_{j_3,r}) \leq \sum_{i = 0}^{M-1} |R_{i,\ell}| \, \cdot \, \sum_{i = 0}^{M-1} |R_{i,r}|;
\]

etc.
\item The bumps undergoing $2n$ reflections, with combined magnitude less than
\[
\left( \sum_{i=0}^{M-1} |R_{i,\ell}| \right)^n \, \cdot \left( \sum_{i=0}^{M-1} |R_{i,r}| \right)^n.
\]

\end{itemize}

Sums of reflection coefficients can be related to the total variation of $\log \sigma$; by making use of the identity
\[
\frac{|1-x|}{|1+x|} \leq \frac{1}{2} | \log x |, \qquad x > 0,
\]
we easily get
\[
\sum_{i=0}^{M-1} |R_{i,\ell}| \leq \frac{1}{2} \sum_{i=0}^{M-1} |\log \sigma_{j} - \log \sigma_{j-1}| \leq \frac{1}{2} \mbox{Var}(\log \sigma).
\]
The same bound holds for $\sum |R_{i,r}|$. The bound for $||| \mathbf{u} |||$ is a geometric series with sufficient convergence criterion
\[
\mbox{Var}(\log \sigma) < 2,
\]
and value
\[
||| \mathbf{u} ||| \leq \frac{1}{1 - \frac{1}{2}\mbox{Var}(\log \sigma)} ||| \phi |||.
\]
For times beyond $t^\sharp$, of the form $t \leq n t^\sharp$ for some integer $n$, this construction can be iterated and the factor $[1 - \frac{1}{2}\mbox{Var}(\log \sigma)]^{-1}$ needs to be put to the power $n$.

We have taken $\phi$ to be a right-going bump so far, but if instead we consider general initial conditions $u_0$ and $U_1$ over the same support, then we should form two one-way bumps as $\phi_{\pm} = u_0 \pm \sigma_j U_1$ in the interval $[x_{j-1}, x_j]$. We can now 1) go back to $u$ through $||| u ||| \leq ||| \mathbf{u} |||$, 2) pass to the limits $M \to \infty$, $h \to 0$, and 3) use the equivalence of $L^1$ and $L^1_{\sigma^{3/2}}$ norms to gather the final bound as
\[
\| u(\cdot, t) \|_{L^1} \leq \frac{2 \left( \sigma_{\max} / \sigma_{\min} \right)^{3/2}}{\left( 1 - \frac{1}{2}\mbox{Var}(\log \sigma) \right)^{n}} \; ( \| u_0 \|_{L^1} + \| \sigma U_1 \|_{L^1} ), \qquad \mbox{when  } t \leq n t^\sharp.
\]

\item We now return to the case of Dirichlet or Neumann boundary conditions. Bumps meeting $x = 0$ or $x=1$ reflect back inside $[0,1]$ with no modification in the $L^1$ or $L^1_{\sigma^{3/2}}$ norm. The extended medium giving rise to equivalent dynamics should then be defined by mirror extension instead of periodization, as
\[
\sigma_{\scriptsize\mbox{ext}}(x) = \left\{ \begin{array}{ll}
        \spc(1) & \mbox{if $x \leq -1$};\\
        \spc(-x) & \mbox{if $-1 < x \leq 0$};\\
        \spc(x) & \mbox{if $0< x \leq 1$};\\
        \spc(2-x) & \mbox{if $1< x \leq 2$};\\
        \spc(0) & \mbox{if $x \geq 2$}.\end{array} \right.
\]
The reasoning proceeds as previously in this extended medium. The discontinuities of $\sigma_{\scriptsize\mbox{ext}}(x)$ are the points $\{ \tilde{x}_j; -M+1 \leq j \leq 2M \}$ of $[-1,2]$ defined as the proper reindexing of $\{ -x_j \} \cup \{ x_j \} \cup \{ 2 - x_j \}$.

If we define $\tilde{R}_{j,\ell}$ and $\tilde{R}_{j,r}$ as the reflection coefficients at $x = \tilde{x}_j$, $-M+1 \leq j \leq 2M $, then the study of combined amplitudes of reflected bumps involves the quantities
\[
\sum_{i = k}^{k+M-1} \tilde{R}_{i,\ell},        \qquad \mbox{ and } \qquad \sum_{i = k}^{k+M-1} \tilde{R}_{i,r}.
\]
Because the extension is now mirror instead of periodic, each of these sums may involve a given reflection coefficient, $R_{j,\ell}$ or $R_{j,r}$, \emph{twice} within a span of length $M$ of the index $j$. Therefore we can only bound each sum individually by Var$(\log \sigma)$ instead of $\frac{1}{2}$Var$(\log \sigma)$ as previously. The reasoning continues exactly like before, with this loss of a factor 2.

\end{enumerate}

\end{proof}


Let us make a few remarks.
\begin{itemize}

\item For the application to the sparse recovery problem, the weighted $L^1_\sigma$ norm is used instead. In the last steps of the proof we can use the equivalence of $L^1_{\sigma^{3/2}}$ and $L^1_\sigma$ norms to slightly modify the estimate into
\begin{equation}\label{eq:sparse-Strichartz2}
\int_0^1 \sigma(x) |u(x,t) | \, dx \leq 2 \left( \frac{\sigma_{\max}}{\sigma_{\min}} \right)^{1/2} \cdot D^n \cdot \left( \int_0^1 \sigma(x) |u_0(x)| \, dx  + \int_0^1 \sigma^2(x) |U_1(x)| \, dx \right).
\end{equation}

\item The constant in the estimate (\ref{eq:sparse-Strichartz}) depends on time only through the maximum number of rotations around the periodized interval $[0,1]$. That this constant does not tend to one as $t \to 0^{+}$ is the expected behavior, because a single scattering event whereby a bump splits into two or more bumps can happen arbitrarily early.

\item On the other hand we do not known if the condition Var$(\log \sigma) < 1$ (Dirichlet or Neumann), or Var$(\log \sigma) < 2$ (periodic boundary condition) is essential for an $L^1$ estimate to hold. Any proof argument that would attempt at removing a condition of this kind---or explain its relevance---would need to account for the combinatorics of destructive interfererence that occurs in regimes of multiple scattering. This question may offer a clue into localization phenomena.

\item The reader may wonder why we have only included initial conditions and no forcing to the wave equation, as is customary in Strichartz estimates. The presence of an additional forcing $F(x,t)$ in equation (\ref{eq:wave}), however, would spoil sparsity for most choices of $F$. Only well-chosen forcings, properly localized and polarized along bicharacteristics relative to the initial conditions, have any hope of preserving the peaky character of a solution to the wave equation---otherwise energy would be introduced and distributed among too large a set of bicharacteristics.

\item Lastly, an estimate such as (\ref{eq:sparse-Strichartz}) would not generally hold for media that are not of bounded variation. This phenomenon can be illustrated in the small $\eps$ limit of a slab of random acoustic medium with correlation length $O(\eps^2)$, slab width $O(1)$, and impinging pulse width $O(\eps)$. This situation is considered in Chapter 9 of \cite{BookPapanico}, where it is also shown that the intensity of the reflected wave decays like $1/t^2$ in expectation, hence $1/t$ for the wave amplitude. The corresponding picture at fixed $t$ is that of a heavy-tailed wave that decays like $1/x$---hence does not belong to $L^1$.

\end{itemize}

\subsection{Analysis of Incoherence}\label{sec:incoherence}

In this section we prove a result of extension, or incoherence of the eigenfunctions of the operator $ \sigma^{-2}(x) d^2/dx^2$ on the interval $[0,1]$, with Dirichlet ($u(0) = u(1) = 0$) or Neumann ($u'(0) = u'(1) = 0$) boundary conditions. Recall that the natural inner product in this context is
\[
\< f,g \> = \int_0^1 f(x) \overline{g}(x) \sigma^2(x) dx,
\]
with corresponding norm
\[
\| f \|_{L^2_{\sigma^2}} = \sqrt{\< f,f \>}.
\]

\begin{theorem}\label{teo:incoherence}
Let $\log \sigma \in BV([0,1])$, and let $v_\omega(x)$ obey $v_\omega''(x) = - \omega^2 \sigma^2(x) v_\omega(x)$ on $[0,1]$ with Dirichlet or Neumann boundary conditions. Then
\begin{equation}\label{eq:incoherence-continuous}
\| \sigma v_\omega \|_{L^\infty} \leq \sqrt{2}  \, \exp \left( \mbox{\emph{Var}}(\log \sigma) \right) \cdot \| v_\omega \|_{L^2_{\sigma^{2}}}.
\end{equation}
\end{theorem}

The point of this result is that the quantity $\exp \left( \mbox{{Var}}(\log \sigma) \right)$ does not depend on $\omega$.

\begin{proof}
Let us first discuss existence and smoothness of $v_\omega(x)$ in $[0,1]$. The operator $\sigma^{-2}(x) \frac{d^2}{dx^2}$  with Dirichlet or Neumann boundary conditions is not only self-adjoint but also negative semi-definite with respect to the weighted inner product $\< \cdot, \cdot \>_{L^2_{\sigma^2}}$. Hence by spectral theory there exists a sequence of eigenvalues $0 \leq \lambda_1 < \lambda_2 < \ldots$ and corresponding eigenvectors in $L^2_{\sigma^2}$, orthonormal for the same inner product. (Basic material on Sturm-Liouville equations and spectral theory in Hilbert spaces can be found in \cite{DL}.) We denote a generic eigenvalue as $\lambda = \omega^2$, and write
\[
v_\omega''(x) = - \omega^2 \sigma^2(x) v_\omega(x).
\]
This equation in turn implies that $v''_\omega \in L^2_{\sigma^2}(0,1)$, hence also belongs to $L^2(0,1)$, i.e. $v_\omega$ is in the Sobolev space $H^2(0,1)$. Iterating this regularity argument one more time, we can further conclude that $v_\omega$ is in the space of functions whose second derivative is in $BV$.

Let us now fix $\omega$, remove it as a subscript for notational convenience, and consider the quantity
\[
I(x) = |v(x)|^2 + \frac{|v'(x)|^2}{\omega^2 \sigma^2(x)}.
\]
Take $\sigma \in C^1$ for the time being, and notice that a few terms cancel out in the expression of $I'(x)$;
\[
I'(x) = -2 (\log \sigma(x))' \frac{|v'(x)|^2}{\omega^2 \sigma^2(x)}.
\]
We can now bound
\[
I'(x) \geq - 2 |(\log \sigma(x))' | \, I(x),
\]
and use Gronwall's inequality to obtain a useful intermediate result on the decay of $I(x)$,
\begin{equation}\label{eq:interm-decay}
I(x) \geq I(0) \exp \left( - 2 \int_0^x | (\log \sigma(y))' | \, dy \right).
\end{equation}

In general, $\sigma(x)$ is only of bounded variation, but the inequality (\ref{eq:interm-decay}) remains true if written as
\begin{equation}\label{eq:interm-decayBV}
I(x) \geq I(0) \exp \left( - 2 \mbox{Var}_x(\log \sigma) \right),
\end{equation}
(In fact, a slitghly stronger result with the positive and negative variations of $\log(\sigma)$ holds.) For conciseness we justify this result in the Appendix.

The quantity $\sigma^2(x) I(x)$ is a continous function over $[0,1]$, therefore absolutely continuous, and reaches its maximum at some point $x^*$. No special role is played by the origin in the estimate (\ref{eq:interm-decayBV}), hence, for all $x \in [0,1]$, we have
\begin{equation*}
\max_{[0,1]} | \sigma(x) v(x)|^2 \leq \sigma^2(x^*) I(x^*) \leq \exp \left( 2 \mbox{Var}(\log \sigma) \right) \sigma^2(x) I(x).
\end{equation*}
Integrate over $[0,1]$;
\[
\| \sigma v \|_\infty^2 \leq \exp \left( 2 \mbox{Var}(\log \sigma) \right) \| I \|^2_{L^2_{\sigma^2}}.
\]
Now
\[
\int_0^1 \sigma^2(x) I(x) \, dx = \int_0^1 \sigma^2(x) |v(x)|^2 \, dx + \int_0^1 \frac{|v'(x)|^2}{\omega^2} \, dx.
\]
It is easy to verify that both terms in the right-hand side of the above equation are in fact equal to each other, by multiplying the equation $v'' + \omega^2 \sigma^2 v = 0$ with $v$, integrating by parts over $[0,1]$, and using the boundary conditions. Therefore
\[
\| \sigma v \|_\infty^2 \leq 2 \, \exp \left( 2 \mbox{Var}(\log \sigma) \right) \| v \|^2_{L^2_{\sigma^2}},
\]
which is the desired result.

%

\end{proof}

A few remarks on related work and extensions are in order.

\begin{itemize}

\item The argument can be slightly modified to obtain instead
\[
\| v_\omega \|_\infty \leq \sqrt{2} \exp \left( \mbox{Var}(\log \sigma) \right)  \frac{\| v_\omega \|_{L^2_{\sigma^2}}}{\| 1 \|_{L^2_{\sigma^2}}}.
\]

\item The quantity $I(x)$ that appears in the proof is, morally, the time-harmonic counterpart of the sideways energy $E(x) = \int_0^T \left[ \sigma^2(x)|\frac{\pd u}{\pd t}|^2 + |\frac{\pd u}{\pd x}|^2 \right] dt$, where $u$ would now solve (\ref{eq:wave2}). Sideways refers to the fact that integration is carried out in $t$ instead of $x$. This trick of interchanging $t$ and $x$ while keeping the nature of the equation unchanged is only available in one spatial dimension. It was recognized in the 1980s by W. Symes that ``sideways energy estimates" allowed to prove transparency of waves in one-dimensional BV media \cite{Sym, LS}, a result very close in spirit to the eigenfunction result presented here. Independently, E. Zuazua \cite{Zua1}, as well as F. Conrad, J. Leblond, and J.P. Marmorat \cite{CLM}, used similar techniques for proving controllability and observability results for waves in one-dimensional BV media. See \cite{Zua2} for a nice review.

\item Theorem \ref{teo:incoherence} is sharp in the sense that for each $0 < s < 1$, there exists a medium $\sigma(x)$ in the H\"older space $C^s([0,1])$ for which there exists a sequence of eigenfunctions exponentially and arbitrarily localized around, say, the origin. Notice the embedding $C^1 \subset BV$, but $C^s \subsetneq BV$ for $s < 1$. The construction is due to C. Castro and E. Zuazua, see \cite{CZ}. In our setting, it means that (\ref{eq:incoherence-continuous}) cannot hold for such $\sigma(x)$, because the constant in the right-hand side would have to depend on $\omega$.

\item Related results in dimension two and higher, using local norms on manifolds with metric of limited differentiability, e.g. $C^{1,1}$, can be found in \cite{Smith-egf, Sogge-egf, JMS}. Interestingly, the constant in front of the $L^2_{\mbox{\scriptsize loc}}$ norm in general grows like a fractional power law in $\lambda$, allowing the possibility of somewhat localized eigenfunctions in dimensions greater than two, typically near the boundaries of the domain.

\item Physically, one may relate the total variation of $\log \sigma$ to a notion of \emph{localization length} $L$, for instance as the largest $x$ such that Var${}_x (\log \sigma)$ is less than some prescribed constant $C$. Then $L$ dictates the decay rate of eigenfunctions, much in the spirit of Lyapunov exponents for the study of localization for ergodic Schr\"{o}dinger operators.

\end{itemize}

\subsection{Analysis of Eigenvalue Gaps}\label{sec:egv-gaps}

This section contains the eigenvalue gap result. Notice that it is the \emph{square root} of the eigenvalues which obey a uniform gap estimate.

\begin{theorem}\label{teo:egv-gaps}
Let $\log \sigma \in BV([0,1])$ with $\mbox{\emph{Var}}(\log \sigma) < \pi$. Let $\lambda_j = - \omega^2_j$, $j = 1,2$, be two distinct eigenvalues of $\sigma^{-2}(x) \, d^2/dx^2$ on $[0,1]$ with Dirichlet or Neumann boundary conditions. Then
\begin{equation}\label{eq:egv-gaps}
\frac{\pi - \mbox{\emph{Var}}(\log \sigma)}{\int_0^1 \sigma(x) \, dx} \leq |\omega_1 - \omega_2| \leq \frac{\pi + \mbox{\emph{Var}}(\log \sigma)}{\int_0^1 \sigma(x) \, dx}.
\end{equation}
\end{theorem}

\begin{proof}
Consider a generic eigenvalue $\lambda = - \omega^2$ with eigenfunction $u(x)$, and assume that $\sigma \in C^1([0,1])$. This restriction will be lifted by a proper limiting argument.

The quantity $I(x)$ introduced in the proof of Theorem \ref{teo:incoherence}, should be seen as the square of the radius $r(x)$, in a polar decomposition
\[
u'(x) = \omega \sigma(x) r(x) \cos \theta(x), \qquad u(x) = r(x) \sin \theta(x).
\]
If $u' \ne 0$, then $\tan \theta = \omega \sigma \frac{u}{u'}$; and if $u = 0$, then $\cot \theta = \frac{1}{\omega \sigma} \frac{u'}{u}$. (We have already seen that $u$ and $u'$ cannot simultaneously vanish since $r^2(x) = I(x) > 0$ everywhere.) From either of those relations one readily obtains
\begin{equation}\label{eq:thetaprime}
\theta'(x) = (\log \sigma(x))' \, \sin \theta(x) \cos \theta(x) + \omega \sigma(x).
\end{equation}
It is interesting to notice that the radius $r(x)$ does not feed back into this first-order equation for $\theta(x)$. Note also that the second term in the above equation quickly dominates as $\omega \to \infty$; if the nonlinear term is neglected we get back the WKB approximation. The boundary conditions on $\theta$ are:
\begin{align*}
&\mbox{Dirichlet: }  &\theta(0) &= m \pi,  &\theta(1) &= n \pi; \\
&\mbox{Neumann: } &\theta(0) &= \frac{\pi}{2} + m \pi,  &\theta(1) &= \frac{\pi}{2} + n \pi.
\end{align*}
where $m$ and $n$ are arbitrary integers. Without loss of generality, set $m = 0$. There is also a symmetry under sign reversal of both $\theta$ and $\omega$, so we may restrict $n \geq 0$.

Equation (\ref{eq:thetaprime}) with fixed $\theta(0)$ is an evolution problem whose solution is unique and depends continuously on $\omega$. Moreover, the solution is strictly increasing in $\omega$, for every $x$, as can be shown from differentiating (\ref{eq:thetaprime}) in $\omega$ and solving for $d \theta/d\omega$ using Duhamel's formula.

The successive values of $\omega$ that correspond to eigenvalues $\lambda = -\omega^2$ are therefore determined by the (quantization) condition that $\theta(1) = n \pi$ for some $n > 0$ (Dirichlet), or $\theta(1) = \pi/2 + n \pi$ for some $n \geq 0$ (Neumann). By monotonicity of $\theta(1)$ in $\omega$, there is in fact a bijective correspondence between $n$ and $\omega$.

As a result, two distinct eigenvalues $\lambda_j = - \omega^2_j$, $j = 1,2$, necessarily correspond to a phase shift of at least $\pi$ at $x = 1$. Set $\theta_j(x)$ for the corresponding phases. Then, by (\ref{eq:thetaprime}),
\[
(\theta_1 - \theta_2)'(x) = (\log \sigma)'(x) [\sin \theta_1 \cos \theta_1 - \sin \theta_2 \cos \theta_2] + (\omega_1 - \omega_2) \sigma(x).
\]
Integrate in $x$ and use the boundary conditions to find
\[
n \pi = \int_0^1 (\log \sigma)'(x) [\sin \theta_1 \cos \theta_1 - \sin \theta_2 \cos \theta_2] dx + (\omega_1 - \omega_2) \int_0^1 \sigma(x) dx, \qquad n \ne 0.
\]
The factor in square brackets is bounded by 1 in magnitude, therefore
\[
|\omega_1 - \omega_2| \geq \frac{1}{\int_0^1 \sigma(x) dx}(\pi - \int_0^1 |(\log \sigma)'(x)| dx),
\]
and also
\[
|\omega_1 - \omega_2| \leq \frac{1}{\int_0^1 \sigma(x) dx}(\pi + \int_0^1 |(\log \sigma)'(x)| dx),
\]
This proves the theorem in the case when $\sigma \in C^1([0,1])$. A standard limiting argument shows that the properly modified conclusion holds when $\sigma \in BV([0,1])$; we leave this justification to the Appendix.

\end{proof}

A few remarks:

\begin{itemize}


\item The polar decomposition used in the proof of Theorem \ref{teo:egv-gaps} is a variant of the so-called Pr\"{u}fer transformation \cite{Atk}, which is more often written as $u'(x) = r(x) \cos \theta(x)$, $u(x) = r(x) \sin \theta(x)$. This simpler form does not appear to be appropriate in our context, however. Notice that such polar decompositions are a central tool in the study of waves in random media \cite{BookPapanico}.

\item It is perhaps interesting to notice that Var$(\log \sigma) < \pi$ is a sufficient condition identified by Atkinson in \cite{Atk} for the convergence of the Bremmer series for the Sturm-Liouville problem.

\end{itemize}


\newcommand{\tom}{\tilde\om}
\newcommand{\tOm}{\tilde\Om}

\section{Algorithms}
\label{sec:implementation}

This section discusses some of the finer points of the implementation. Basic discretization issues were exposed in Section \ref{sec:disc}.




\subsection{Extraction of the Eigenvectors}
\label{subsec-extraction-eigen}

The first part of the algorithm consists in extracting a random set of eigenvectors $\{ \tilde v_{\tom} \}_{\tom}$ of the discretized operator $\Ll = \Si^{-2} L$ where $\Si = \diag_j( \sigmaj )$.  The discretized Laplacian $L$ over $\RR$ is computed spectrally 
\eq{
	\widehat{L v}[m] = - 4 \pi^2 m^2 \hat v[m] 
}
where $m \in \{ -N/2+1,\ldots,N/2 \}$ indexes the frequencies of the discrete Fourier transform. Fourier transforms are computed in $O(N \log N)$ operation with the FFT.



Since we are interested in extracting only a few eigenvectors chosen at random, we use an iterative method \cite{templates-eigen} that parallelizes trivially on multiple processors. Each processor computes and stores independently from the others a few eigenvectors using an iterative method. 


The simplest way to compute an eigenvector $\tilde v_{\tom}$ whose eigenvalue $-\tom^2$ is closest to a given $-\tom_0^2$ is to compute iterative inverse powers
\eq{
	\tilde v_{\tom}^{(k+1)} = (\Ll + \tom_0^2 \Id)^{-1} \tilde v_{\tom}^{(k)},
}
with an adequate starting guess $\tilde v_{\tom}^{(0)}$, typically white noise. In practice we use a variant of this power iteration called the restarted Arnoldi method, and coded in Matlab's \texttt{eigs} command.
At each iteration we approximately solve the linear system $(\Ll + \tom_0^2 \Id) \tilde v_{\tom}^{(k+1)} = \tilde v_{\tom}^{(k)}$ with a few steps of stabilized bi-conjugate gradient \cite{templates-linsyst}. Recent work \cite{ErlNab} suggests that a shift in the reverse direction $\Ll - \tom_0^2 \Id$ or a complex shift $\Ll + \imath\tom_0^2 \Id$ are good preconditionners for this linear system resolution. Such preconditoners are applied efficiently using multigrid, or alternatively and as used in this paper, using discrete symbol calculus. In this framework, it is the whole symbol of the operator $(\Ll - \tom_0^2 \Id)^{-1}$ which is precomputed in compressed form, and then applied iteratively to functions on demand. The resulting preconditioners are quite competitive. See \cite{DSC} for more information.


Each shift $\tom_0^2$ should be chosen according to an estimate of the true (but unknown) eigenvalues repartition to sample as uniformly as possible the set of eigenvectors of $\Ll$. The eigenvalues of the discrete Laplacian in a constant medium $\sigmaj = \si_0$ are $\{ - \om_{\max}^2 (2 m/N)^2 \}_{m=-N/2+1}^{N/2}$ where $\om_{\max}^2 = \pi^2 N^2 / \si_0^2$. Treating the general case as a perturbation of this constant setting leads draw $\tom_0$ uniformly at random in $[0,\om_{\max}]$ where $\om_{\max}$ defined as the maximum eigenvalue of $\Ll$. The value of $\om_{\max}$ is readily available and computed using power iterations on $\Ll$. We have seen in Section \ref{sec:egv-gaps} that the departure from uniformity is under control when the medium has a reasonable total variation. We also explained that the sampling should be without replacement: in practice the implementation of ``replacement" carefully accounts for the multiplicity two of each eigenspace in the case of periodic boundary conditions.

\subsection{Iterative Thresholding for $\lun$ Minimization}

At the core of the compressive wave computation algorithm is the resolution of the optimization problem \eqref{eq:ell1} involving the $\lun$ norm. We introduce the operator $\Phi : \RR^{N} \mapsto \RR^{\Om}$ such that 
\eq{
	\Phi u[\tom] = \sum_j u[j] \tilde v_{\tom}[j]
}
where $\tom \in \Om$ indexes $K = |\Om|$ eigenvectors $\{ \tilde v_{\tom} \}_{\tom}$ of the discretized Laplacian and $\Si = \diag_j(\sigmaj)$. 
Discarding the time dependency, the $\lun$ optimization \eqref{eq:ell1} is re-written in Lagrangian form as
\begin{equation}\label{eq-lagrangian-optim}
	\underset{u}{\min} \;
	\frac{1}{2} \norm{ \Phi \Si^2 u - \tilde c }^2
	+ \la \sum_j \sigmaj |u[j]|.
\end{equation}
The Lagrangian parameter $\la$ should be set so that $\norm{ \Phi \Si^2 u - \tilde c } \leq \epsilon$.


As described in Section \ref{sec:disc}, $\epsilon$ account for the discretization error, and it can also reflects errors in computation of the eigenvectors. This quantity can be difficult to estimate precisely, and it can be slightly over-estimated, which increases the sparsity of the computed approximation. 


Iterative algorithms solves the minimization \eqref{eq-lagrangian-optim} by sequentially applying a gradient descent step to minimize $\norm{ \Phi \Si^2 u - \tilde c }$ and a soft thresholding to impose that the solution has a low weighted $\lun$ norm $\sum_j \sigmaj |u[j]|$. This algorithm was proposed independently by several researcher, see for instance \cite{daubechies-iterated,combettes-proximal,figueiredo-nowak-em}, and its convergence is proved in  \cite{daubechies-iterated,combettes-proximal}. 

The steps of the algorithm are detailed in Table \ref{listing-it-thresh}. They correspond to the application of the iterative thresholding algorithm to compute the iterates $\Si u^{(k)}$ with the measurement matrix $\Phi \Si$. Since this matrix satisfies $\norm{\Phi\Si u} \leq \norm{u}$ by Plancherel, these iterates converge to a minimizer of \eqref{eq-lagrangian-optim}.

Since the correspondence between $\epsilon$ and $\la$ is a priori unknown, $\la$ is modified iteratively at step 4 of the algorithm so that the residual error converges to $\epsilon$, as detailed in \cite{chambolle-algo-tv}.


An important feature of the iterative algorithm detailed in Table \ref{listing-it-thresh} is that it parallelizes nicely on clusters where the set of eigenvectors $\{ v_{\tom} \}_{\tom \in \tOm}$ are distributed among several nodes. In this case, the transposed operator $\Phi^*$ is pre-computed on the set of nodes, and the application of $\Phi\Si^2$ and $\Si^2 \Phi^*$ is done in parallel during the iterations.


The iterative thresholding algorithm presented in Table \ref{listing-it-thresh} might not be the fastest way to solve \eqref{eq-lagrangian-optim}. Recent contributions to sparse optimization include for instance primal-dual schemes \cite{zhu-primal-dual}, gradient pursuit \cite{blumensath-grad-pursuit}, gradient projection \cite{figueiredo-grad-projection}, fixed point continuation \cite{hale-fixed-point-cont}, gradient methods \cite{nesterov-smooth}, Bregman iterations \cite{yin-bregman} and greedy pursuits \cite{needell-cosamp}. These methods could potentially improve the speed of our algorithm, although it is still unclear which method should be preferred in practice. 


Another avenue for improvement is the replacement of the $\lun$ norm by non-convex functionals that favor more strongly the sparsity of the solution. Non-convex optimization methods such as FOCUSS \cite{gorodnitsky-focuss}, re-weighted $\lun$ \cite{candes-reweighted-l1} or morphological component analysis with hard thresholding \cite{starck-mca} can lead to a sub-optimal local minimum, but seem to improve over $\lun$ minimization in some practical situations.

\begin{listing}
\begin{enumerate}
	\item \textit{Initialization:} set $u^{(0)} = 0$ and $k=0$. 
	\item \textit{Update of the solution:} compute a step of descent of $\norm{ \Phi \Si^2 u - \tilde c }^2$
		\eq{
			\bar u^{(k)} = u^{(k)} + \Phi^* \pa{ \tilde c - \Phi \Si^2 u^{(k)} },
		}
	\item \textit{Minimize $\lun$ norm:} threshold the current update
		\eq{
			\foralls j, \quad u^{(k+1)}[j] = S_{\la / \sigmaj}( \bar u^{(k)}[j] ),
		}
		where the soft thresholding operator is defined as
		\eq{
			S_{\la}(\al) = \choice{
				0 \qifq |\al|<\la,\\
				\al - \sign(\al)\la \quad \text{otherwise}.
			}
		}
	\item \textit{Update the Lagrange multiplier:} set 
	\eq{
		\la \leftarrow \la \frac{\epsilon}{\norm{ \Phi \Si^2 u - \tilde c }}
	}
	\item \textit{Stop:} while $\norm{u^{(k+1)}-u^{(k)}} >$ tol, set $k \leftarrow k+1$ and go back to 2.
\end{enumerate}\vspace{-3mm}
    \caption{Iterative thresholding algorithm to solve \eqref{eq-lagrangian-optim}. \label{listing-it-thresh}}
\end{listing}

\subsection{Sparsity Enhancement}\label{sec:sparsity-enhancement}

The success of the compressive method for wave propagation is directly linked to the sparsity of the initial conditions $u_0$ and $u_1$. To enhance the performance for a fixed set of eigenvectors, the initial data can be decomposed as $u_0 = \sum_{\ell=0}^{\ell-1} u_0^{k}$ where each of the $L$ components $\{u_0^\ell\}_\ell$ is sufficiently sparse, and similarly for $u_1$. The algorithm is then performed $L$ times with each initial condition $u_0^\ell$ and the solution is then recomposed by linearity. It would be interesting to quantify the slight loss in the probability of success since $L$ simulations are now required to be performed accurately, using the same set of eigenvectors.

Since the solution might become less sparse with time $t$ increasing, one can also split the time domain into intervals $[0,t] = \bigcup_i [t_i,t_{i+1}]$, over each of which the loss of sparsity is under control. The algorithm is restarted over each interval $[t_i,t_{i+1}]$ using a decomposition of the wavefields at time $t_i$ into a well chosen number $L = L_{t_i}$ of components to generate sparse new initial conditions. 

\section{Numerical Experiments}

\newcommand{\Err}{\text{\upshape Err}}

\subsection{Compressive Propagation Experiments}

We perform simulations on a 1D grid of $N=2048$ points, with impedance $\si(x)$ of various smoothness and contrast $\si_{\max} / \si_{\min}$. The result of the compressive wave computation is an approximate discrete solution $\{\tilde u[j](t) \}_{j=0}^{N-1}$ at a fixed time step $t$ of the exact solution $\{u[j](t)\}_j$ of the discretized wave equation.

The performance of the algorithm is evaluated using the $\ldeux$ recovery error in space and at several time steps $t_i = T i/n_t$ for $i=0,\ldots,n_t-1$ uniformly distributed in  $[0,T]$, where $n_t=100$. The final time is evaluated such that $\int_0^T \si = 1$, so that the initial spike at $t=0$ propagates over the whole domain. This error is averaged among a large number of random sets $\Om \in \Om_K$ of $K$ eigenvectors
\eql{\label{eq-error-measure}
	\Err(\si,K/N)^2 = \frac{1}{N n_t  |\Om_K| \, \norm{u}} \sum_{\Om \in \Om_K} \sum_{i=0}^{n_t-1} \sum_{j=0}^{N-1} | u[j](t_i) - \tilde u[j](t_i) |^2.
}
Each set $\Om \in \Om_K$ is drawn at random using the procedure described in Section \ref{subsec-extraction-eigen}.

This error depends on the sub-sampling factor $K/N$, where $K=|\Om|$ is the number of computed eigenvectors, and on the impedance $\si$ of the medium. Numerical evaluation of the decay of $\Err$ with $K$ are performed for two toy models of acoustic media that could be relevant in seismic settings: smooth $\si$ with an increasing number of oscillations, and piecewise smooth $\si$ with an increasing number of discontinuities. For each test, the initial condition $u_0$ is a narrow gaussian bump of standard deviation $7/N$. 

\paragraph{Smooth oscillating medium.}

A uniformly smooth impedance $\si_\ga$ parameterized by the number of oscillations $\ga \in [1,20]$ is defined as
\eql{\label{eq-smooth-medium}
	\sigma_\ga[j] = \frac{\si_{\max}+1}{2} + \frac{\si_{\max}-1}{2} \pa{ \sin(2\pi \ga j/N) + 3 }.
}
The contrast $\si_{\max}/\si_{\min}=\si_{\max} \in [1,10]$ is also increased linearly with the complexity $\ga$ of the medium, according to $1+\frac{9}{19}(\gamma-1)$.



Figure \ref{fig-sin-all-error} shows how the recovery error $\Err(\si_{\ga},K/N)$ scales with complexity $\ga$ of the medium and the number of eigenvectors $K$. For media with moderate complexity, one can compute an accurate solution with $N/10$ to $N/5$ eigenvectors.

\myfigure{
	\begin{tabular}{ccc}
		\includegraphics[width=.33\linewidth]{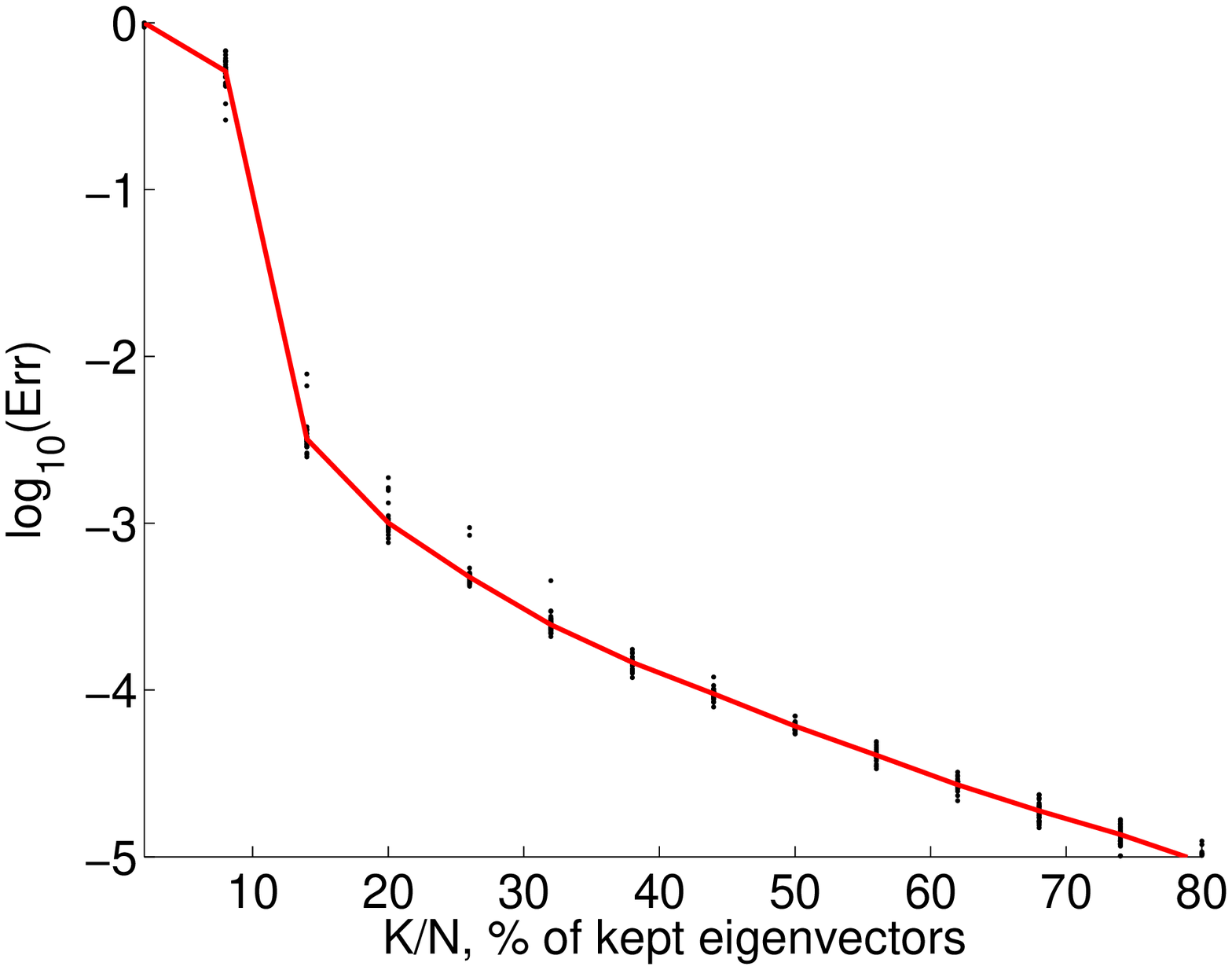}&\hspace{-6mm}
		\includegraphics[width=.33\linewidth]{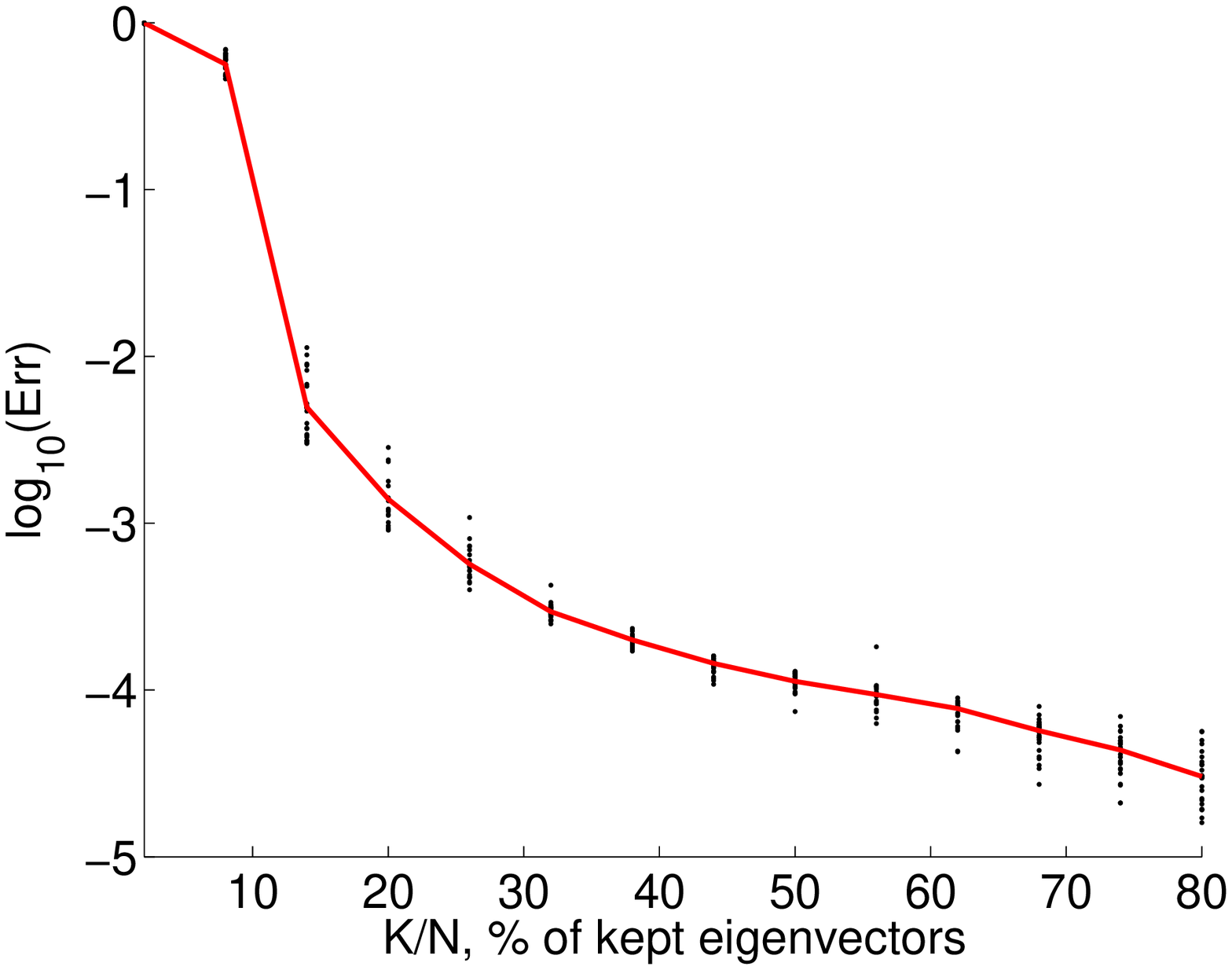}&\hspace{-6mm}
		\includegraphics[width=.33\linewidth]{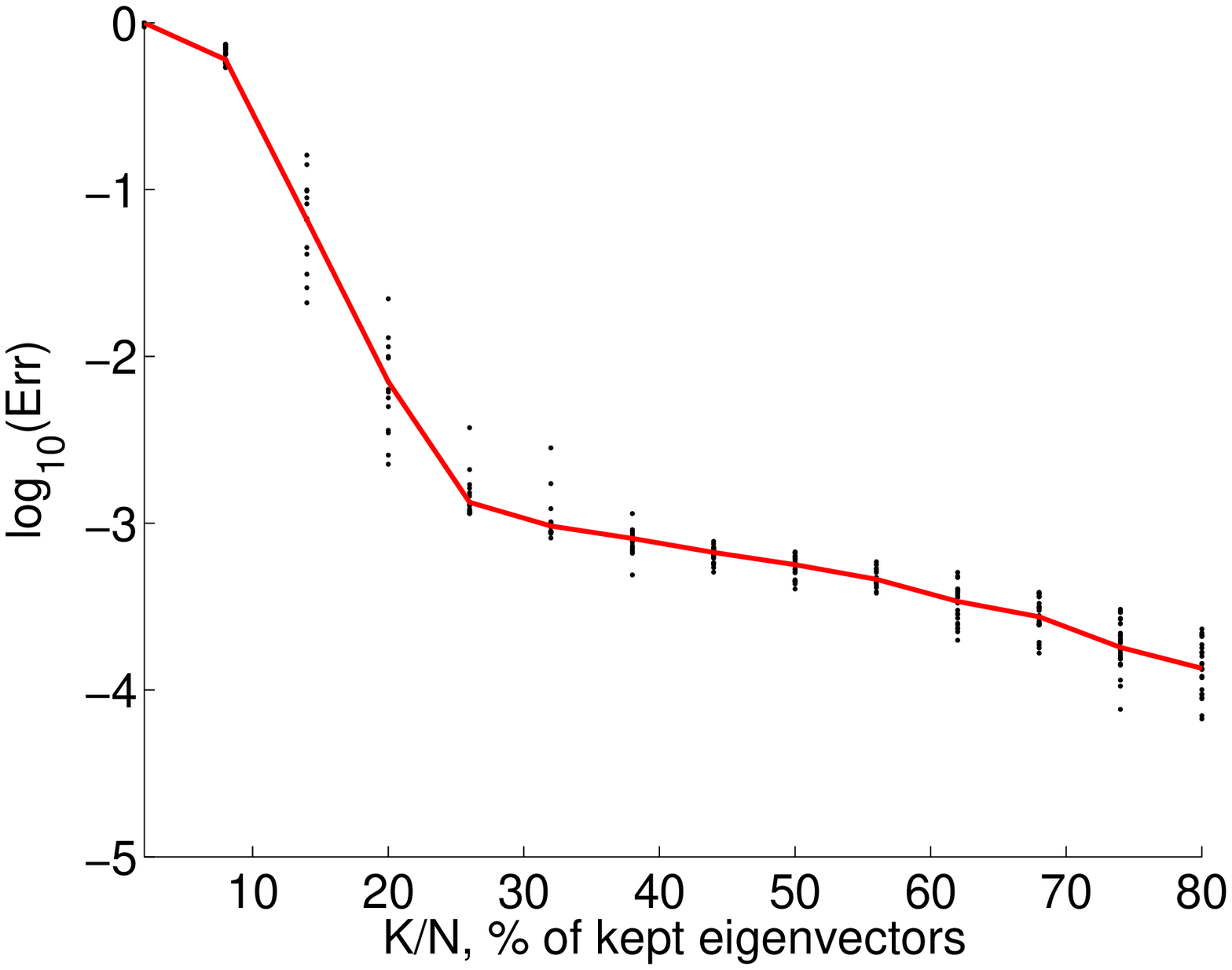}\\
		$\ga=2$ & $\ga=4$ & $\ga=8$	
	\end{tabular}
	\begin{tabular}{c}
    \includegraphics[width=.4\linewidth]{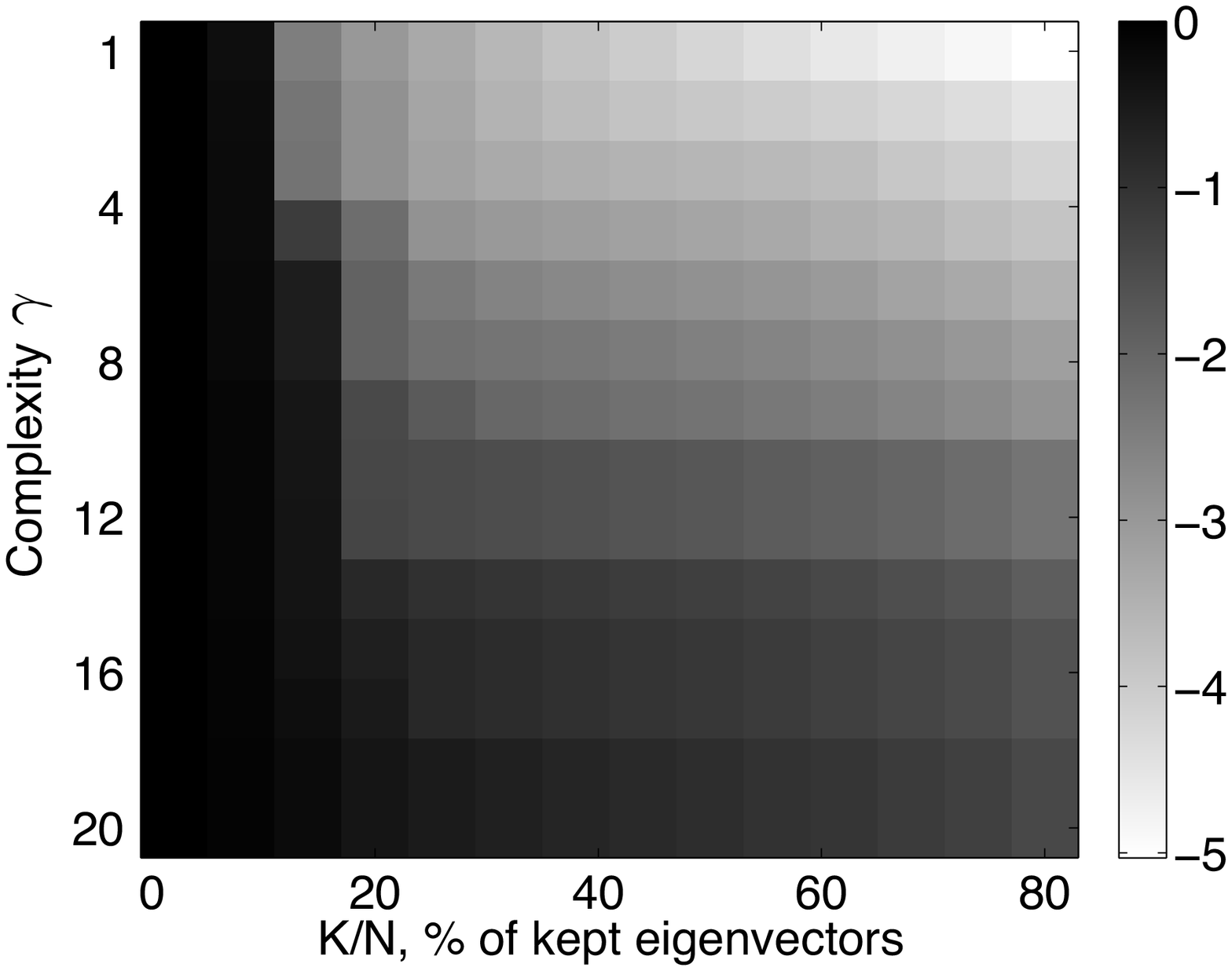}\\
    $\log_{10}(\Err(\si_{\ga},K/N))$
    \end{tabular}
}{ 
	Compressive wave propagation in a smooth medium. Top row: recovery error decay $\log_{10}(\Err(\si_{\ga},K/N))$ as a function of the sub-sampling $K/N$ for various complexity $\ga$ of the medium. Each black dots corresponds to the error of a given random set $\Om$ (the red curve is the result of the averaging among these sets). Bottom row: 2D display of the error $\log_{10}(\Err(\si_{\ga},K/N))$ as a function of both $K/N$ (horizontal axis) and $\ga $ (vertical axis).
}{fig-sin-all-error}

\paragraph{Piecewise smooth medium.}

Jump discontinuities in the impedance $\si$ reflect the propagating spikes and thus deteriorate the sparsity of the solution when $t$ increases, as shown on Figure \ref{bv-steps}, top row. Figure \ref{bv-steps} shows that compressive wave computation is able to recover the position of the spikes with roughly $N/5$ to $N/4$ eigenvectors. 

\newcommand{\bvspeed}[1]{ \includegraphics[width=.3\linewidth]{bv-steps/bv-steps-rough4-contrast30-#1} }
\newcommand{\myrot}[1]{\rotatebox{90}{\quad\;#1}}
\newcommand{\prespace}{\hspace{-7mm}}
\newcommand{\interspace}{\hspace{-5mm}}
\myfigure{
	\begin{tabular}{cc}
	\includegraphics[width=.33\linewidth]{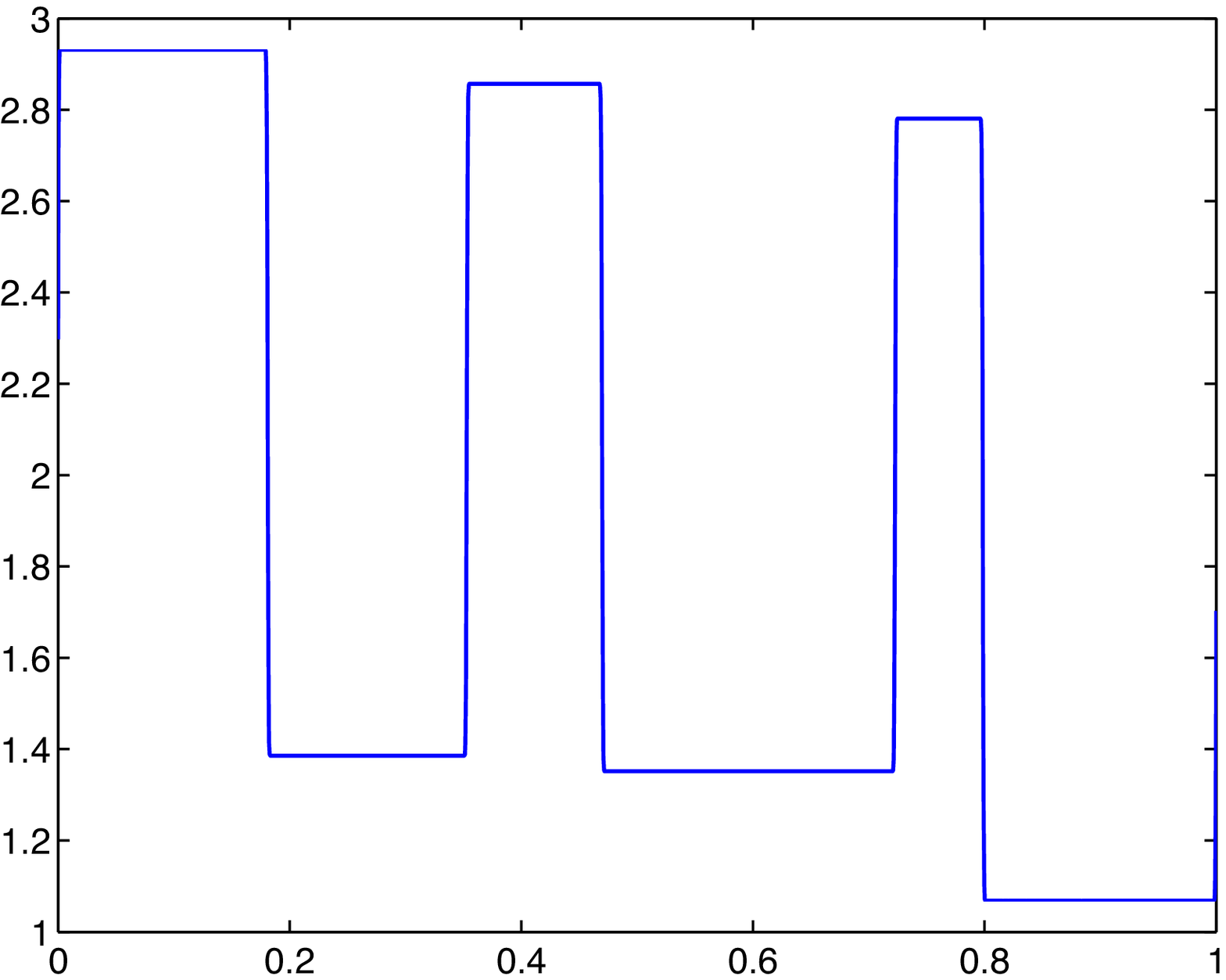}& 
	\includegraphics[width=.38\linewidth]{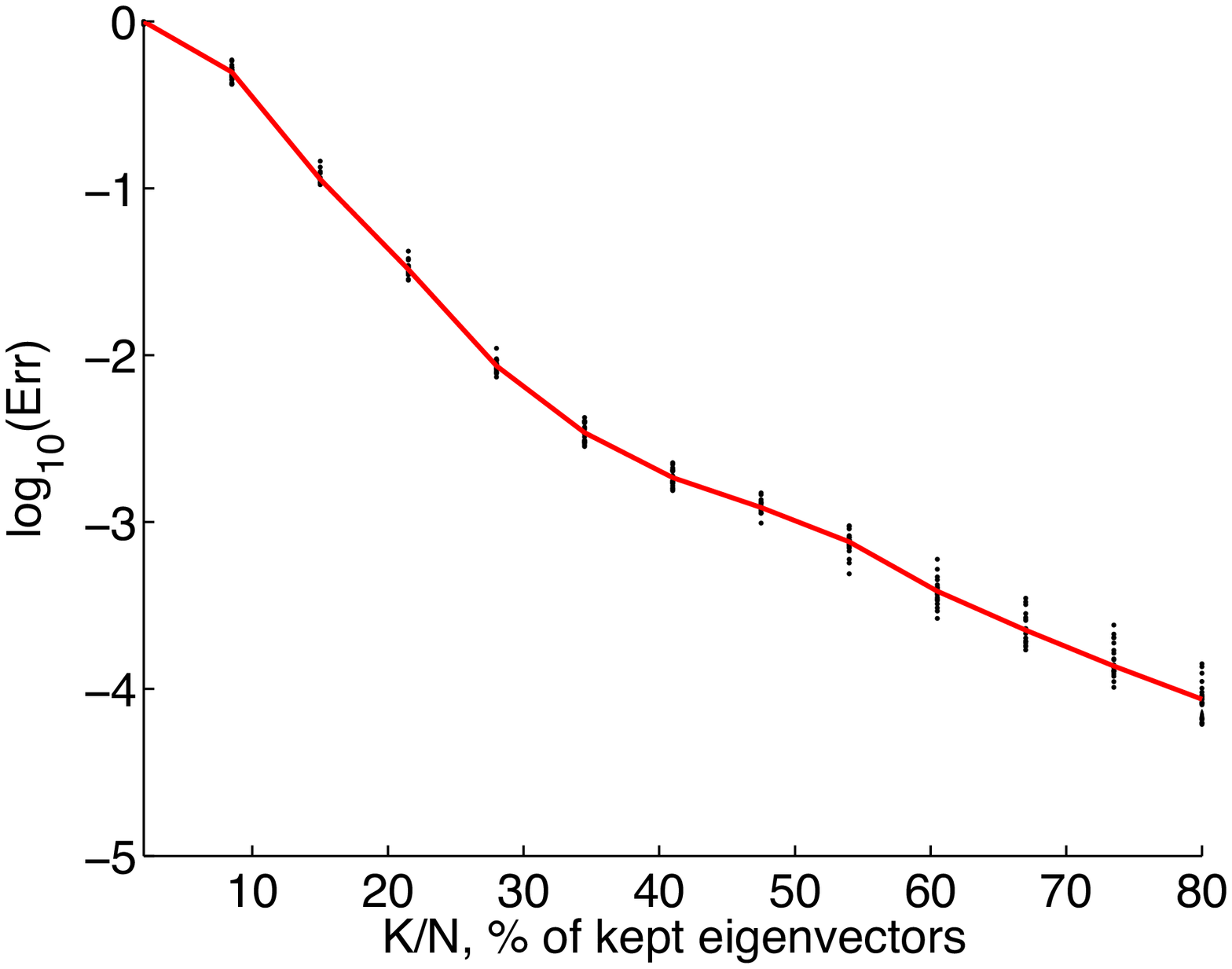}\\
	Speed $\si^{-1}$ & $\log_{10}(\Err(\si_{\ga},K/N))$
	\end{tabular}
	\begin{tabular}{cccc}
		\myrot{\qquad $u$}& \prespace{}
		\bvspeed{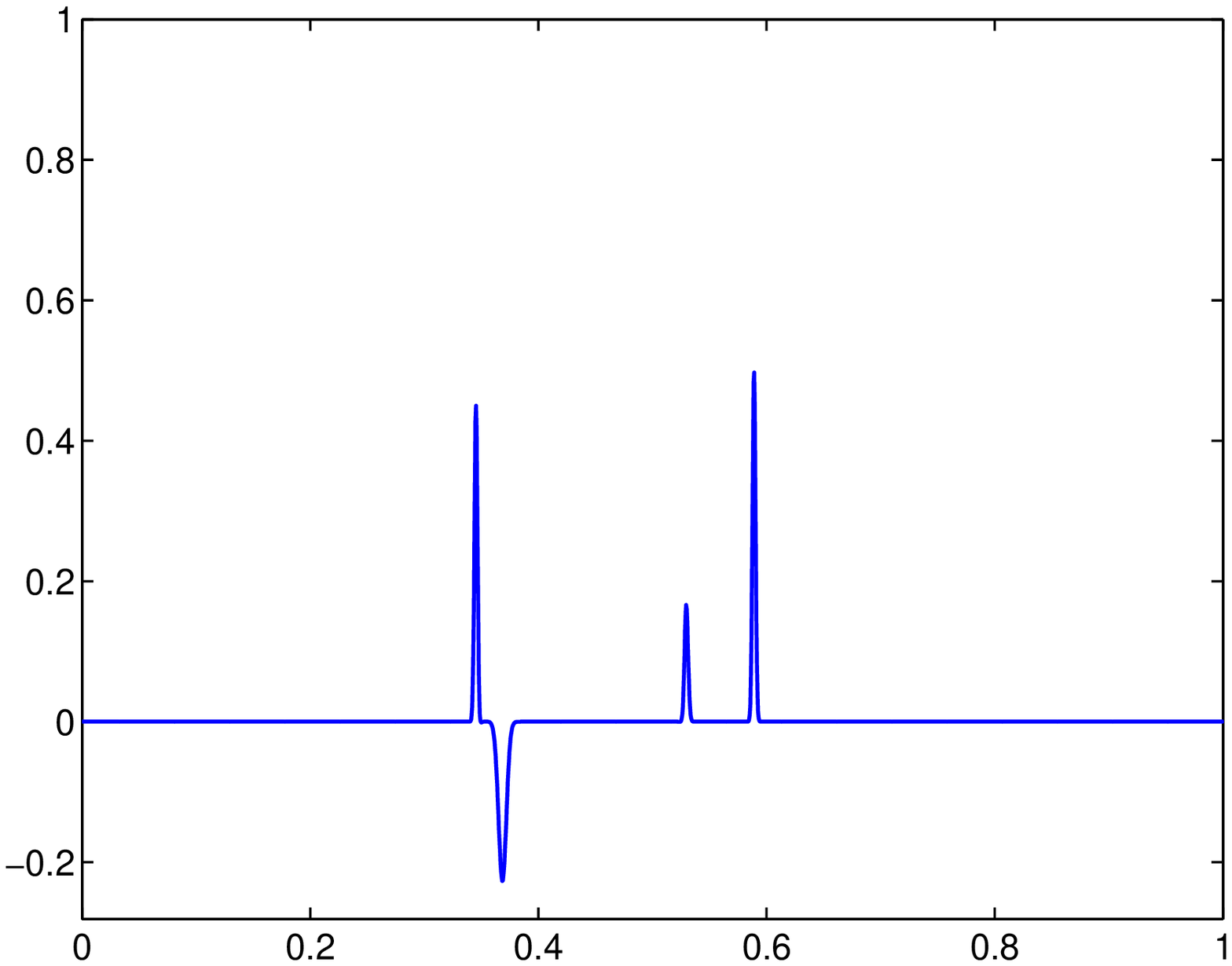}& \interspace{}
		\bvspeed{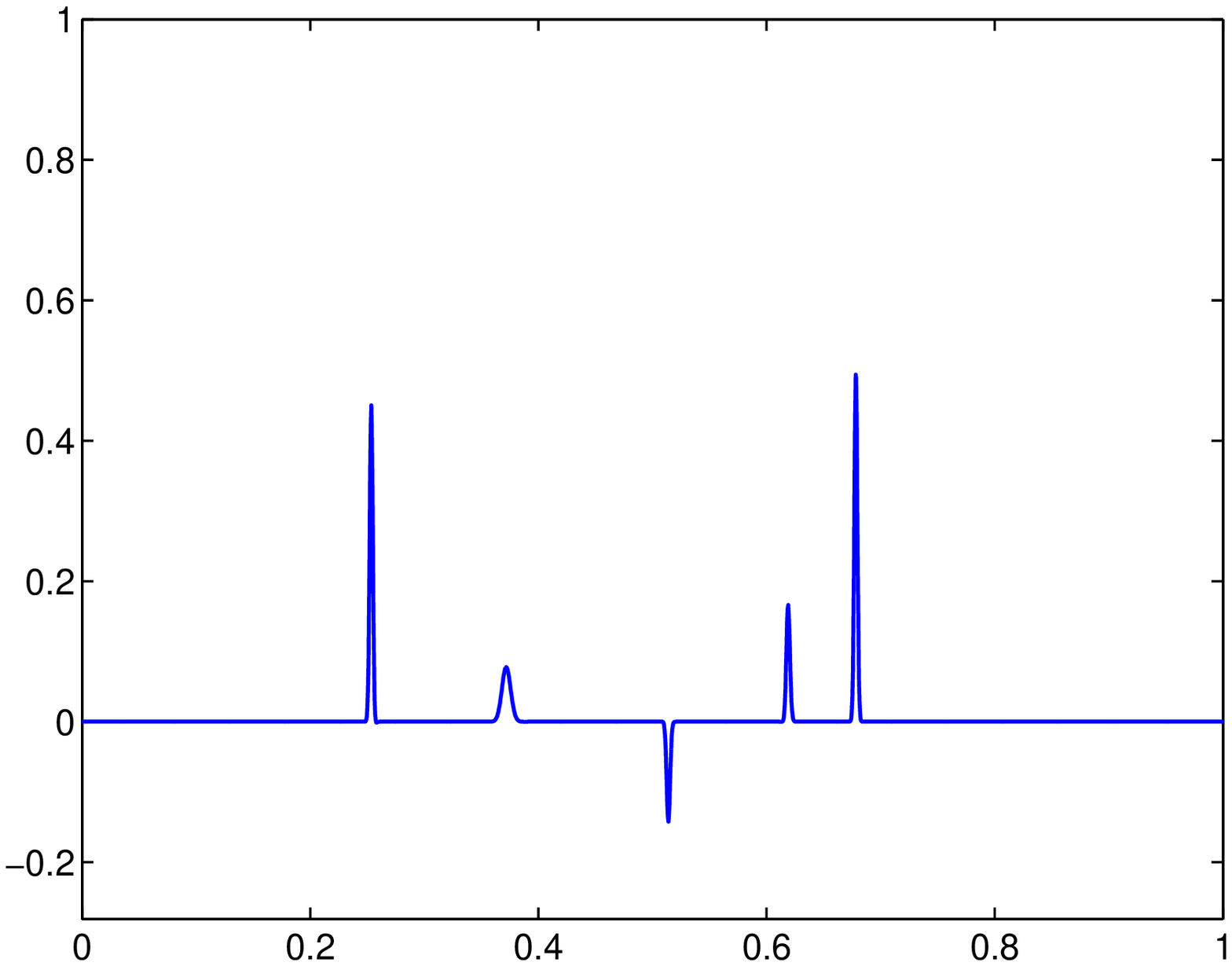}& \interspace{}
		\bvspeed{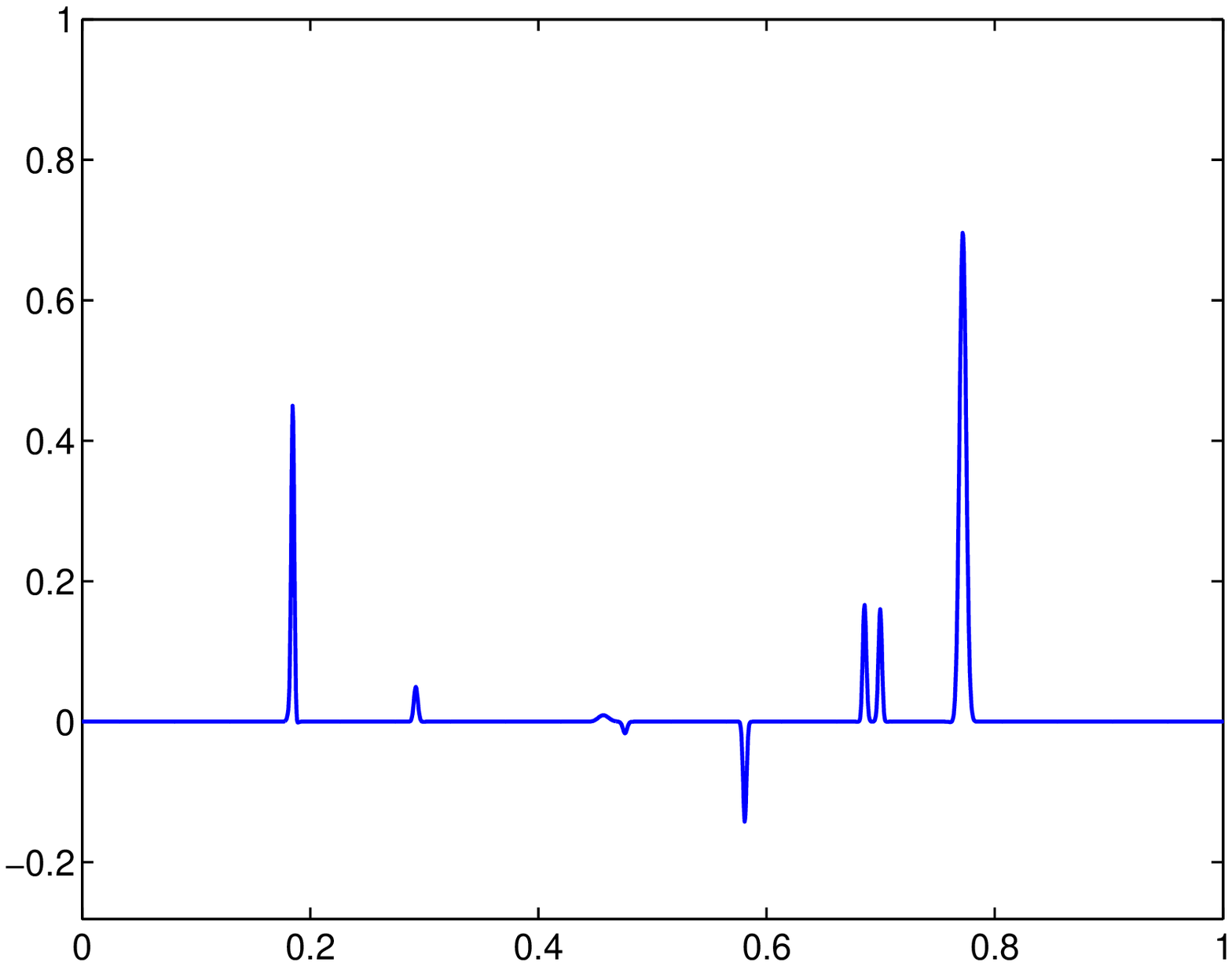}\\ 
		\myrot{$K/N=0.2$}& \prespace{}
		\bvspeed{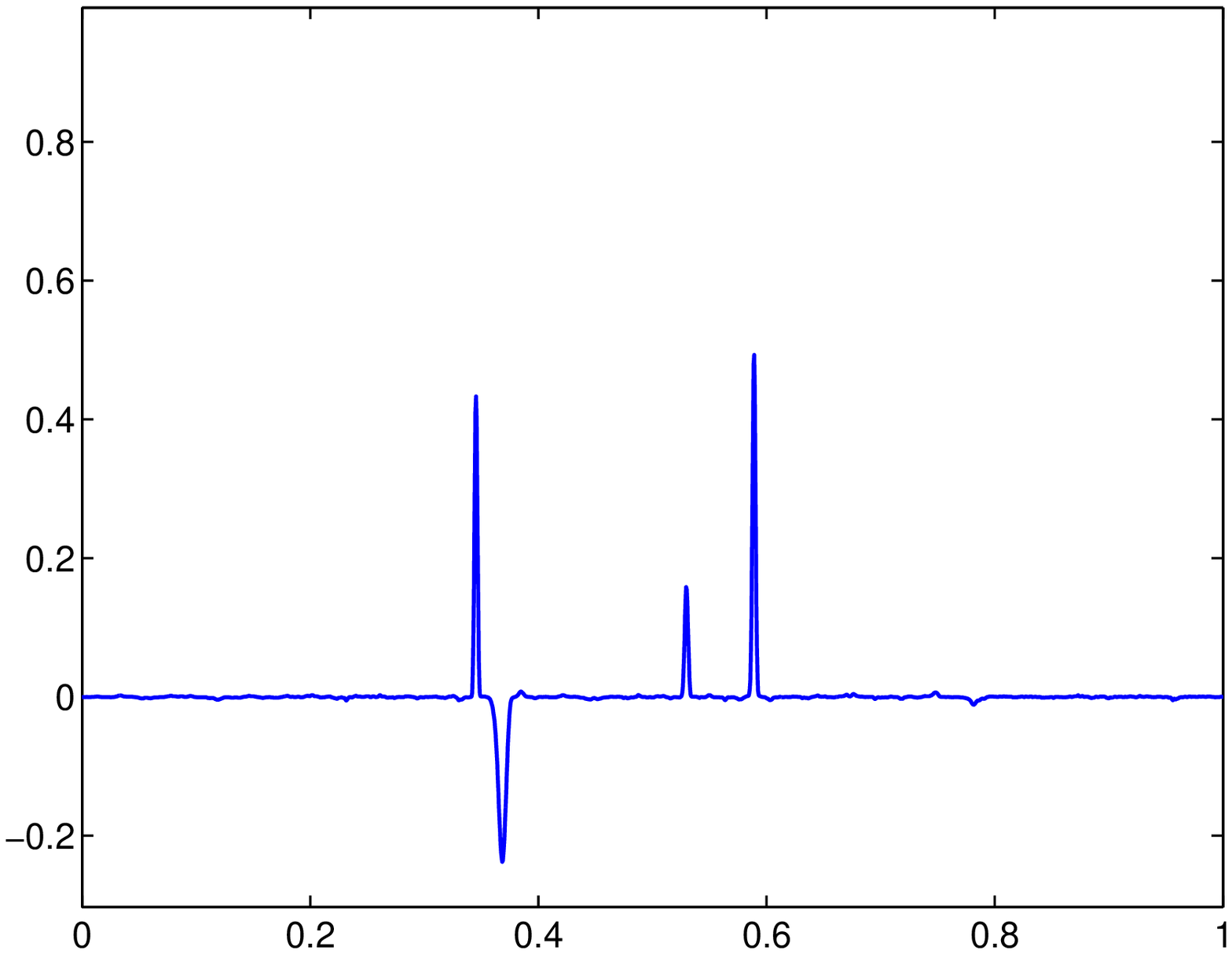}& \interspace{}
		\bvspeed{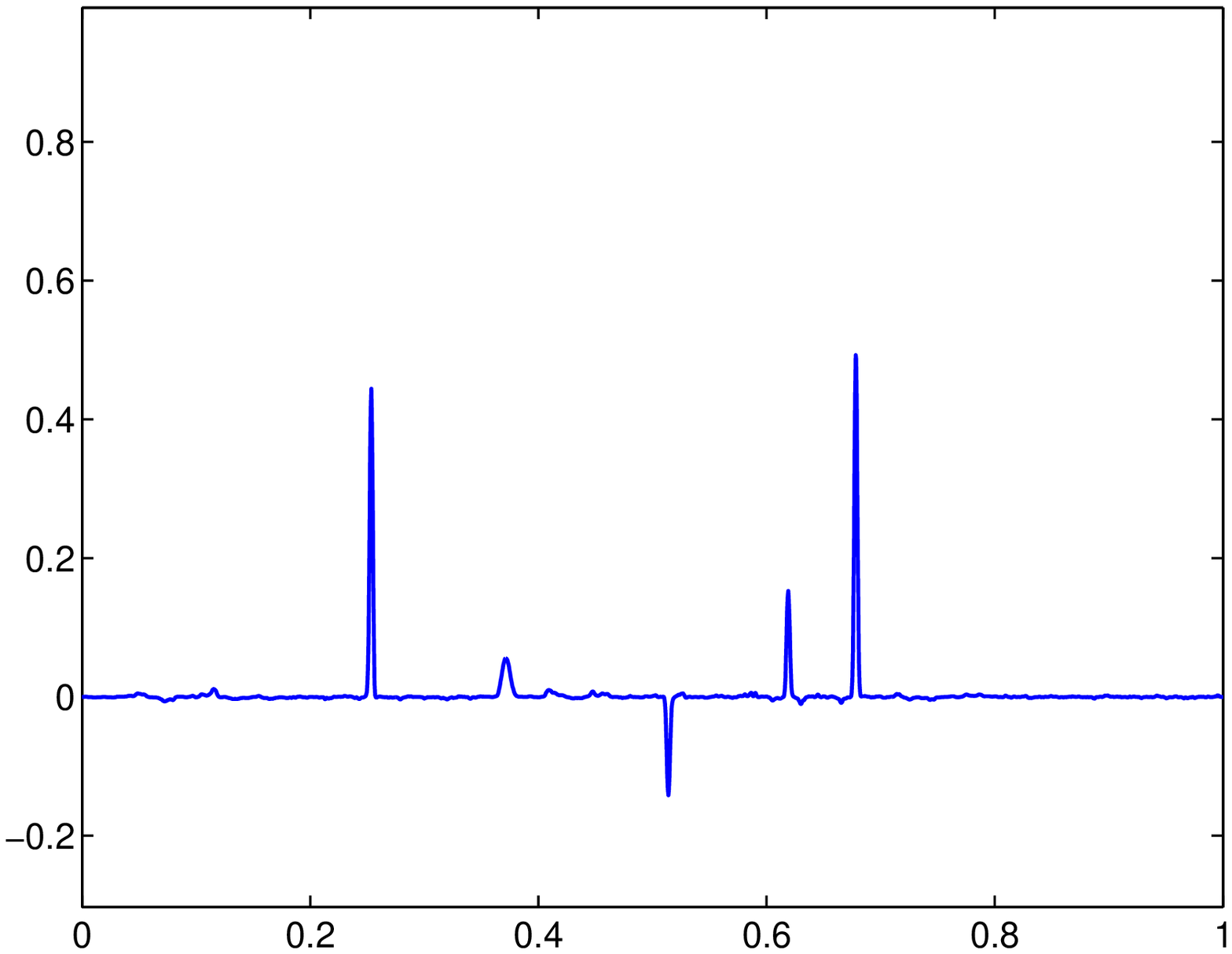}& \interspace{}
		\bvspeed{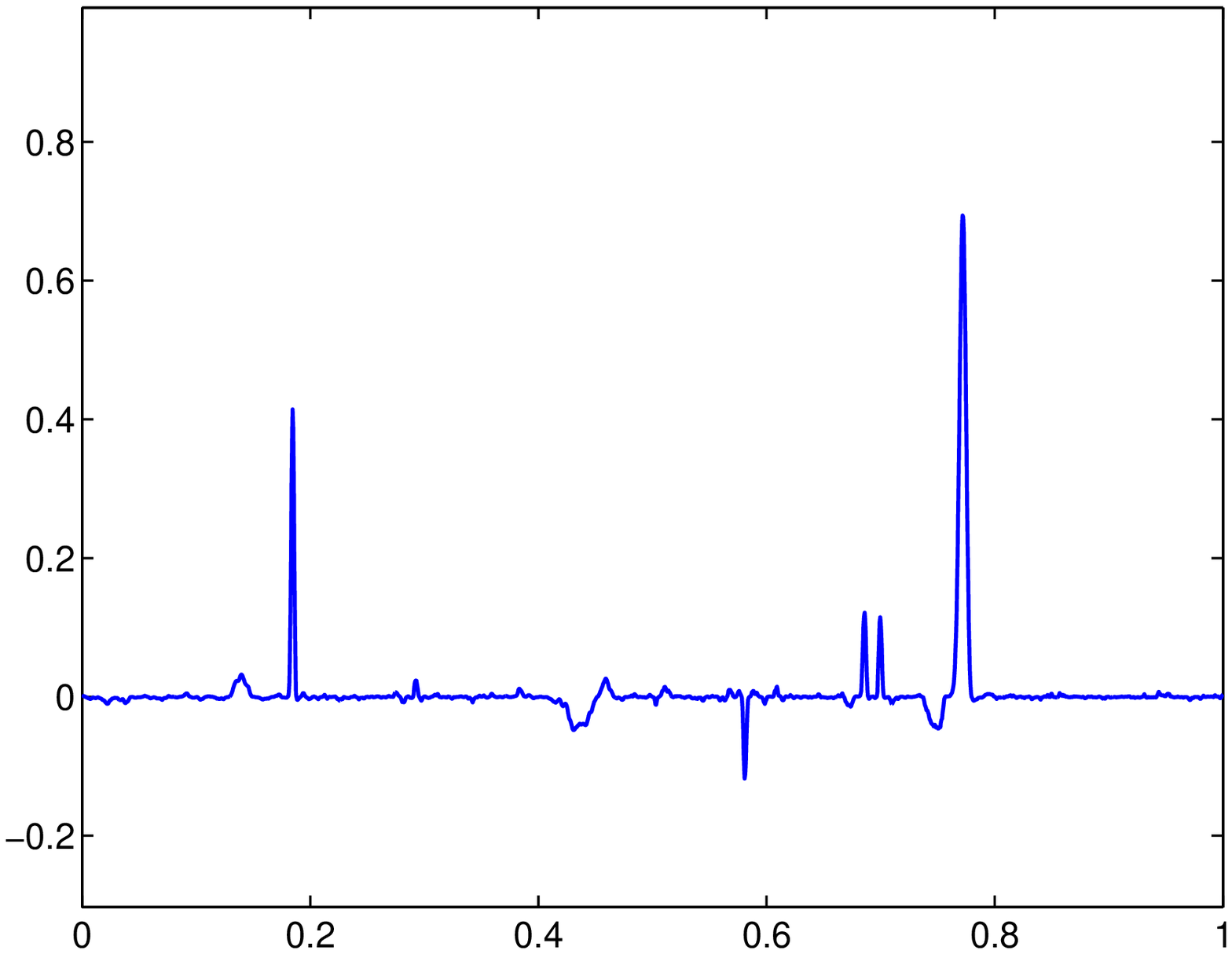}\\ 
		\myrot{$K/N=0.15$}& \prespace{}
		\bvspeed{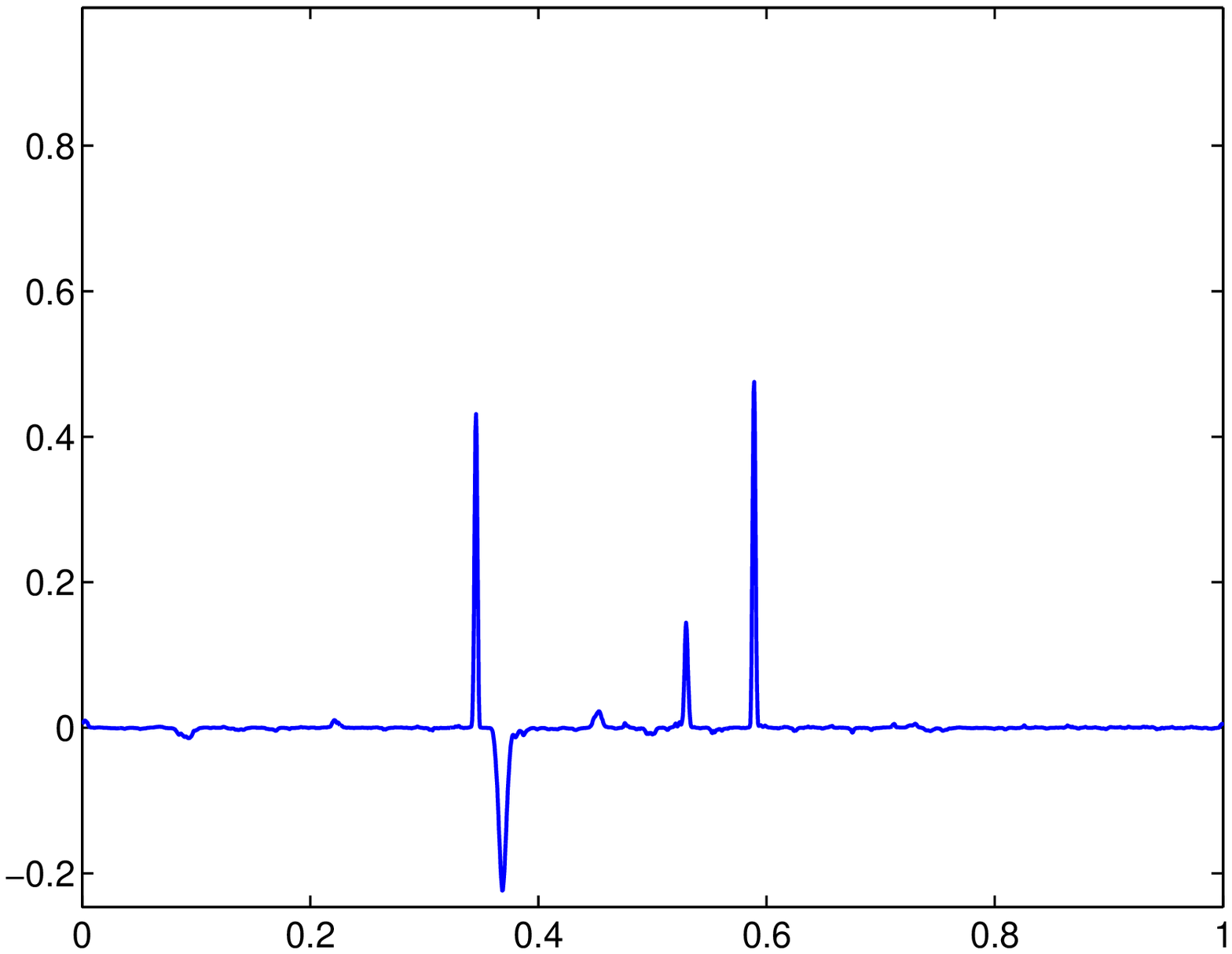}& \interspace{}
		\bvspeed{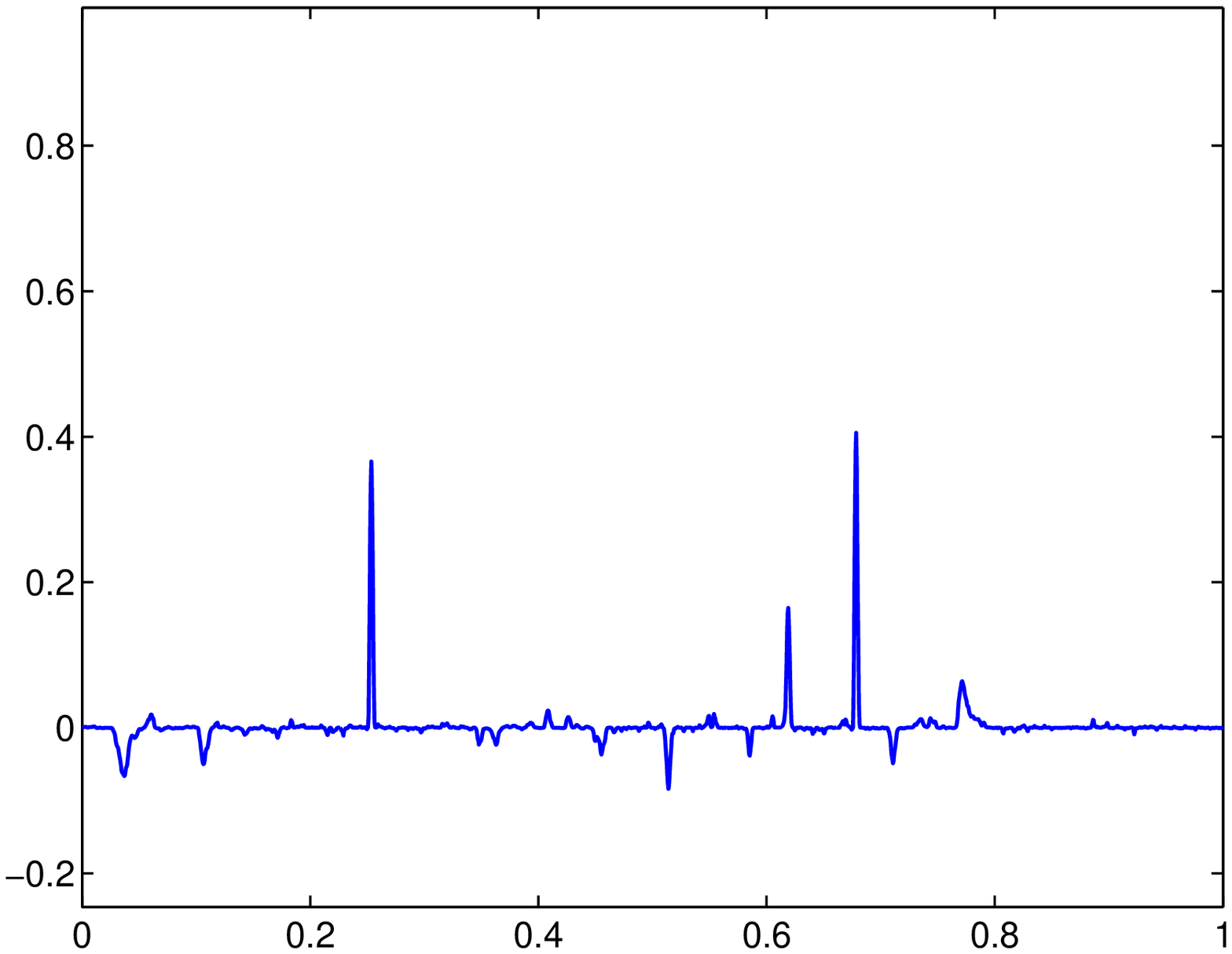}& \interspace{}
		\bvspeed{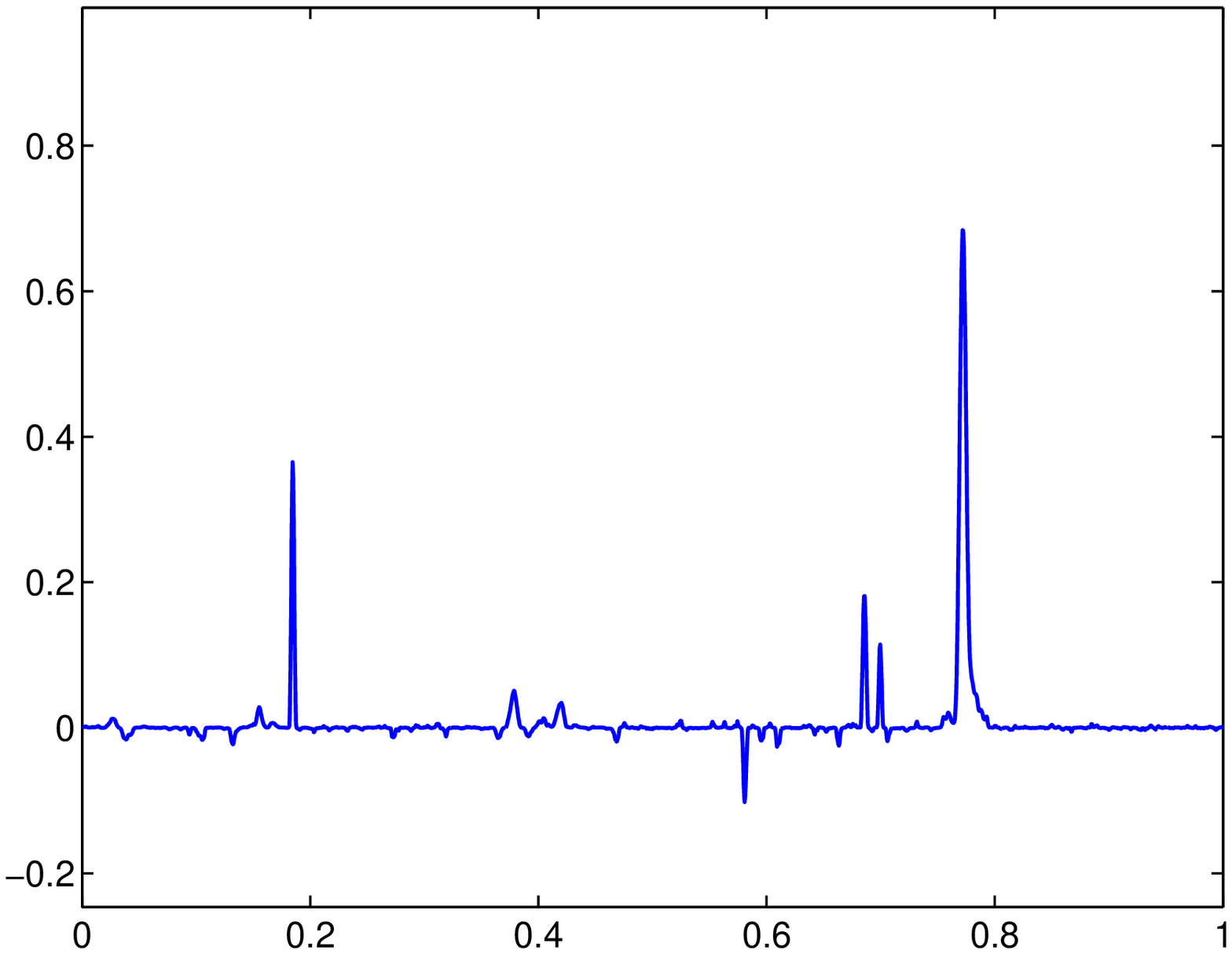}\\ 
		& $t=.2$ & $t=.4$ & $t=.6$ 
	\end{tabular}
}{ 
	Examples of approximate solution $\tilde u[j](t)$ for a piecewise smooth medium.
}{bv-steps}

To quantify more precisely the recovery performance, we consider a family of piecewise smooth media $\si_\ga$ parameterized by its number of step discontinuities $\ga \in [1, 20]$. The contrast $\si_{\max}/\si_{\min} \in [1,10]$ is also increased linearly with the complexity $\ga$ of the medium. The discontinuities are uniformly spread over the spatial domain $[0,1]$. The piecewise smooth impedance $\si_\ga$ is slightly regularized by a convolution against a Gaussian kernel of standard deviation $5/N$. This tends to deteriorate the sparsity of the solution when $t$ increases, but helps to avoid numerical dispersion due to the discretization of the Laplacian. Figure \ref{fig-piecewise-all-error} shows how the recovery error $\Err(\si_{\ga},K/N)$ scales with complexity of the medium and the number $K$ of eigenvectors. 

\myfigure{
	\begin{tabular}{ccc}
		\includegraphics[width=.33\linewidth]{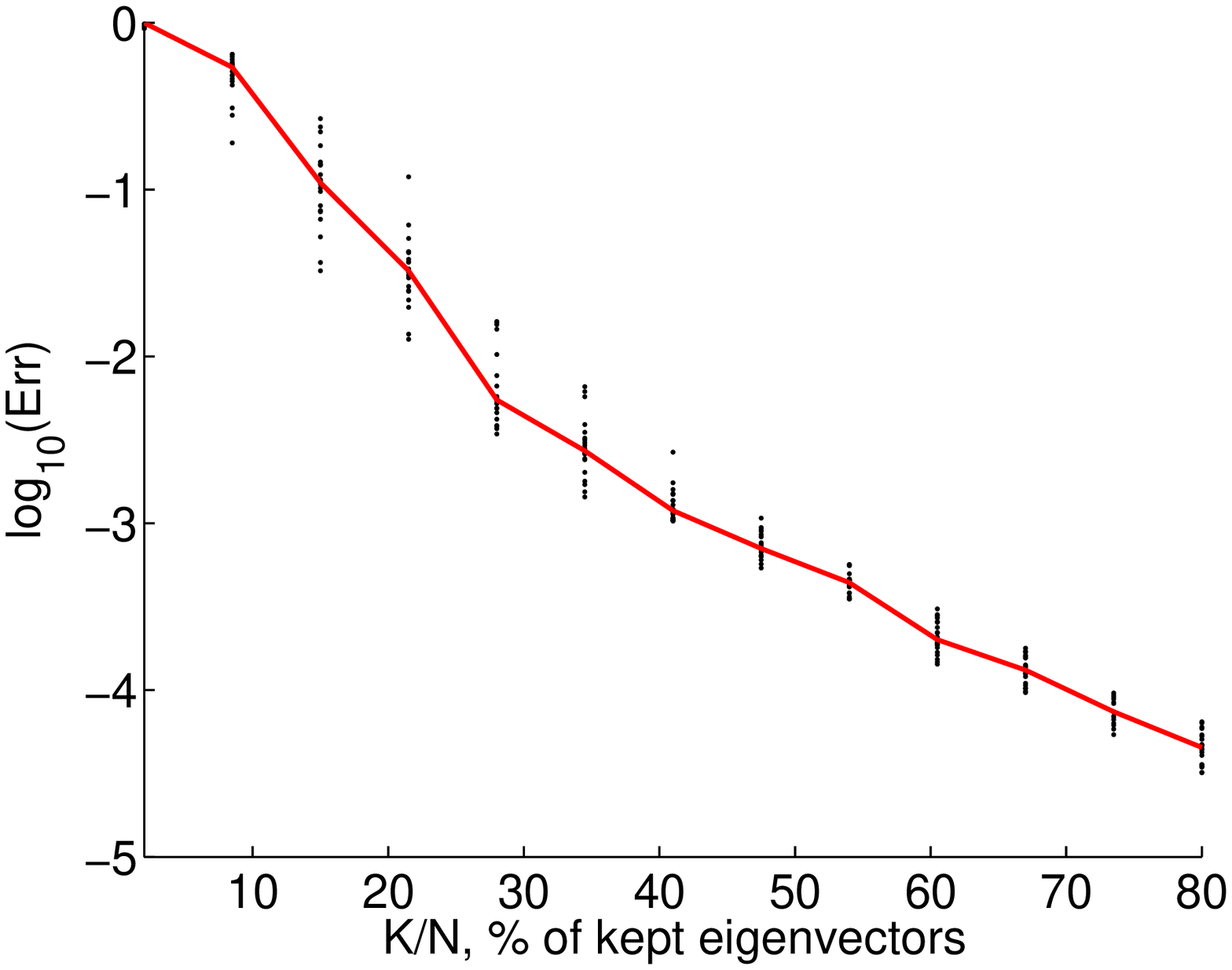}&\hspace{-6mm}
		\includegraphics[width=.33\linewidth]{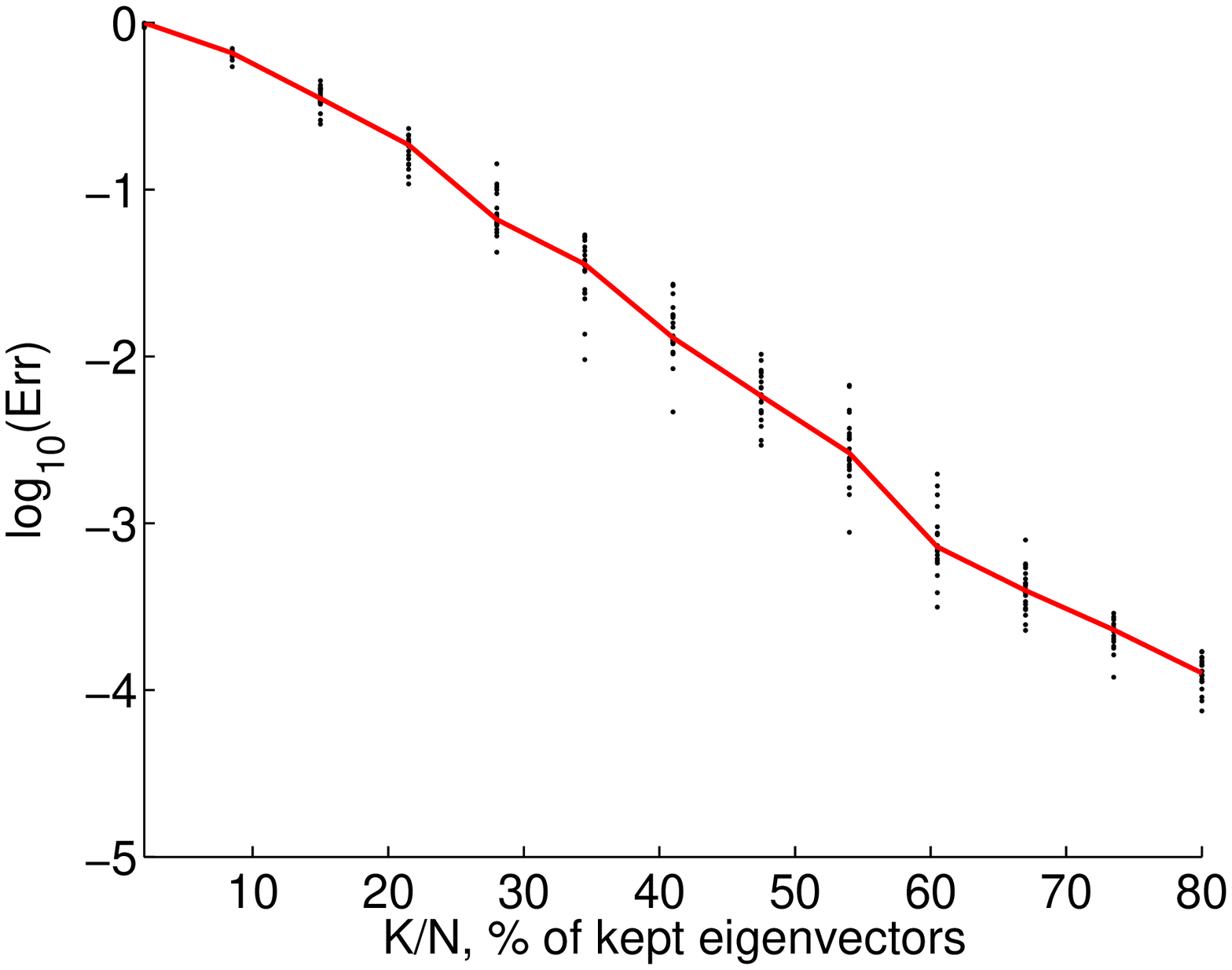}&\hspace{-6mm}
		\includegraphics[width=.33\linewidth]{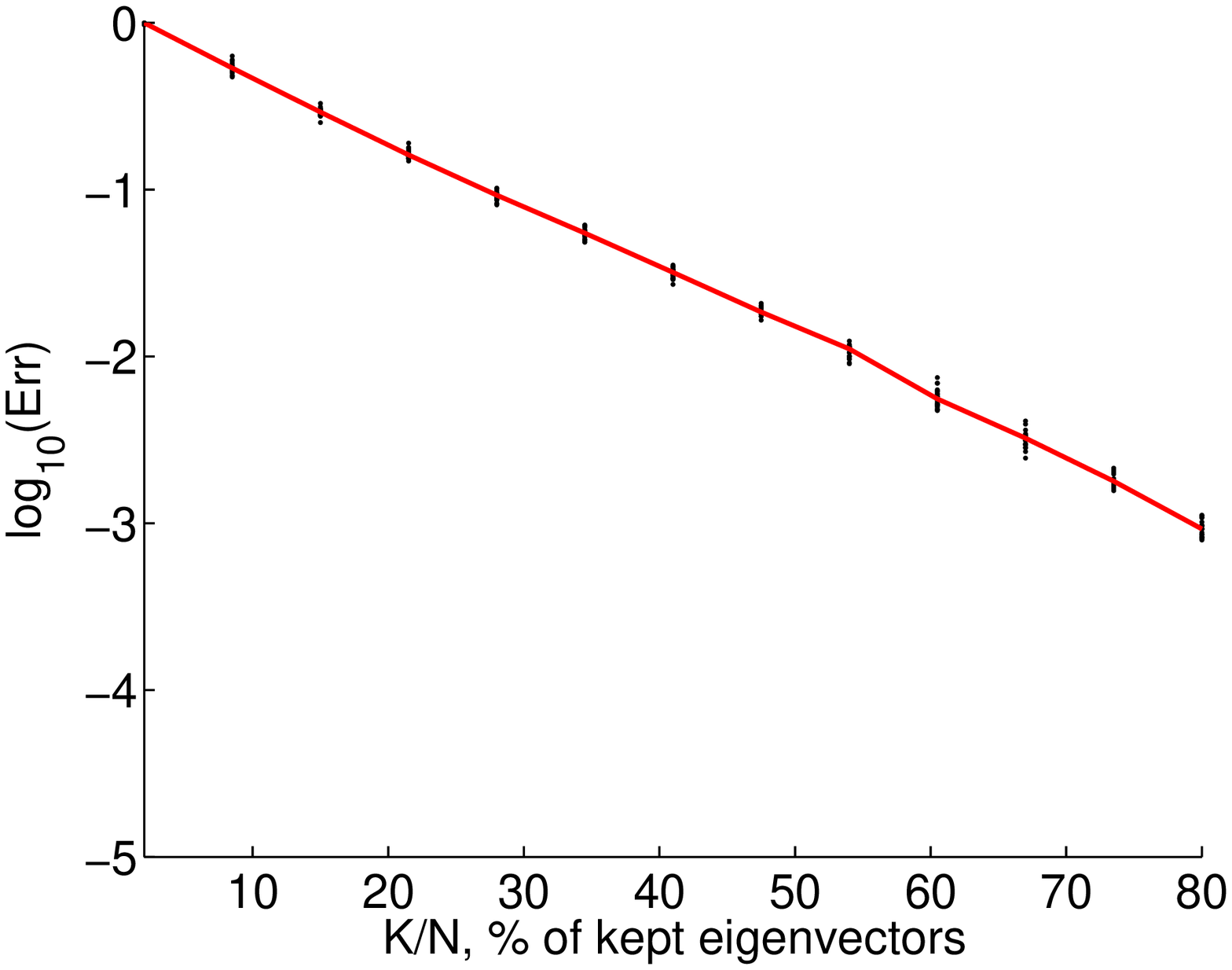}\\
		$\ga=2$ & $\ga=4$ & $\ga=8$	
	\end{tabular}
	\begin{tabular}{c}
    \includegraphics[width=.4\linewidth]{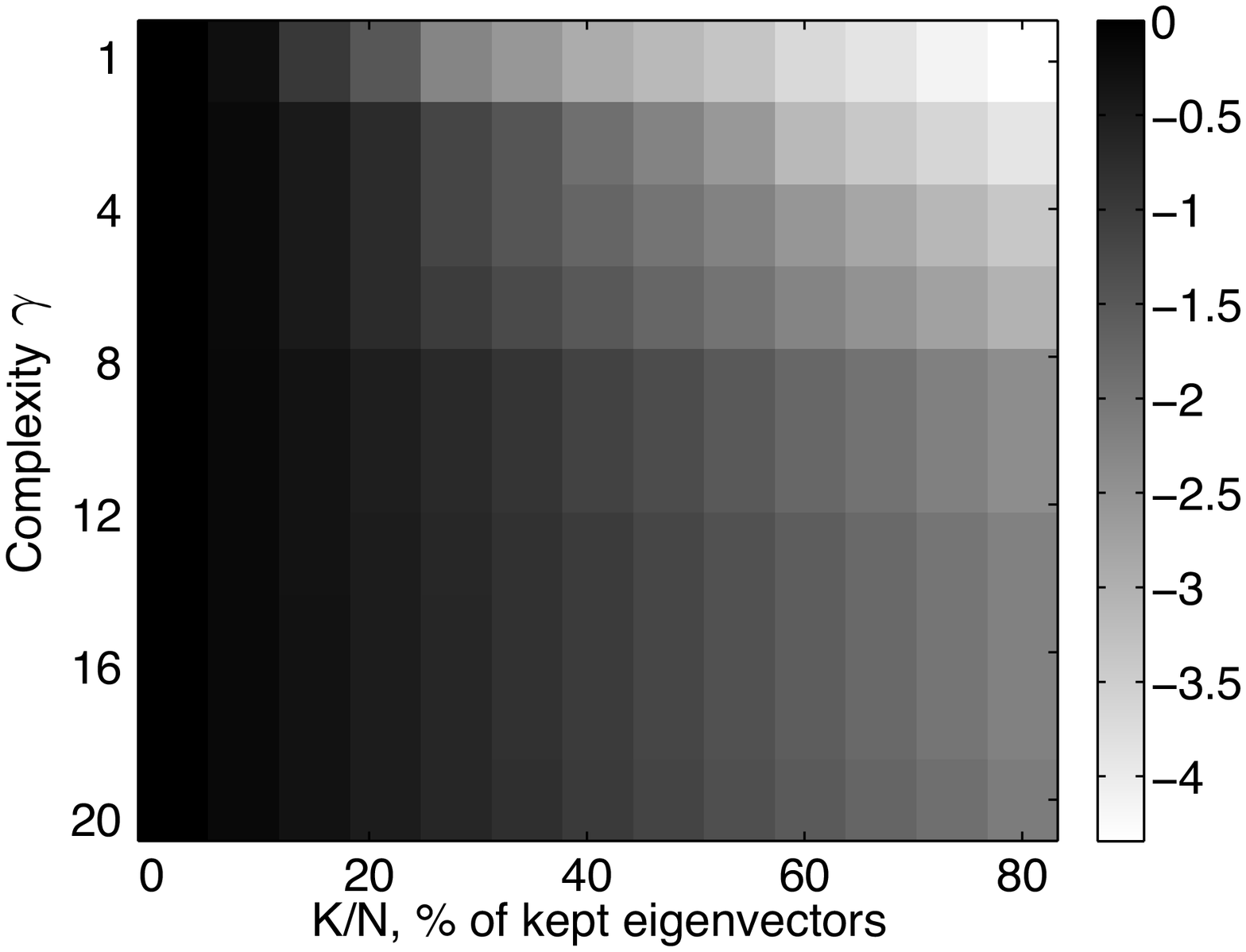}\\
    $\log_{10}(\Err(\si_{\ga},K/N))$
    \end{tabular}
}{ 
	Compressive wave propagation in a piecewise smooth medium. 
}{fig-piecewise-all-error}

\subsection{Compressive Reverse-Time Migration}\label{sec:crtm}

In this example, we consider an idealized version of the inverse problem of reflection seismology, where wavefield measurements at receivers are used to recover some features of the unknown impedance $\sigma(x)$. Our aim is to show that compressive wave computation offers significant memory savings in the implementation of a specific adjoint-state formulation that we call \emph{snapshot reverse-time migration}.

\paragraph{Review of 1D reverse-time migration.}

We will assume a one-dimensional setup as in the rest of this paper, i.e.,
\begin{equation}\label{eq:wave-original}
\sigma^2(x) \frac{\pd^2 u}{\pd t^2} - \frac{\pd u}{\pd x^2} = 0,
\end{equation}
with a known localized initial condition $u(x,0) = u_0(x)$, $\pd u / \pd t(x,0) = u_1(x)$, and over a domain sufficiently large that the choice of boundary conditions does not matter. The $x$ coordinate plays the role of depth. As in classical treatments of reflection seismology, the squared impedance is perturbed about a smooth, known reference medium $\sigma_0^2$, as
\[
\sigma^2(x) = \sigma_0^2(x) + r(x),
\]
where the high-frequency perturbation $r(x)$ has the interpretation of ``reflectors''. The wave equation is then linearized as
\begin{equation}\label{eq:wave-linearized}
\sigma_0^2(x) \frac{\pd^2 u}{\pd t^2} - \frac{\pd u}{\pd x^2} = - r(x) \frac{\pd^2 u_{\mbox{inc}}}{\pd t^2},
\end{equation}
where the incoming field $u_{\mbox{inc}}$ solves the wave equation in the unperturbed medium $\sigma_0^2$. An important part of the seismic inversion problem---the ``linearized problem"---is to recover $r(x)$ from some partial knowledge of $u(x,t)$, interpreted as solving (\ref{eq:wave-linearized}). We will assume the availability of \emph{snapshot} data, i.e., the value of $u$ and $\pd u / \pd t$ at some fixed time $T$,
\[
(d_1(x), d_2(x)) = (u(x,T), \frac{\pd u}{\pd t}(x,T)),
\]
possibly restricted in space to a region where the waves are reflected, by opposition to transmitted. We then let $F[\sigma^2_0]$ for the forward, or modeling operator, to be inverted:
\begin{equation}\label{eq:forward-op}
\begin{pmatrix} d_1 \\ d_2 \end{pmatrix} = F[\sigma^2_0] r.
\end{equation}

A more realistic seismic setup would be to assume the knowledge of $u(0,t)$ for positive $t$, but the step of going from $(d_1, d_2)$ to $u(0,t)$, or vice-versa, is a depth-to-time conversion that should present little numerical difficulty. While assuming trace data $u(0,t)$ would give rise to an adjoint-state wave equation with a right-hand side, working with snapshot data has the advantage of casting the adjoint-state equation as a final-value problem without right-hand side.

More precisely, a simple argument of integration by parts (reproduced in the Appendix) shows that the operator $F$ in (\ref{eq:forward-op}) is transposed as
\begin{equation}\label{eq:imaging-op}
F^*[\sigma^2_0] \begin{pmatrix} d_1 \\ d_2 \end{pmatrix} = - \int_0^T q(x,t) \frac{\pd^2 u_{\mbox{inc}}}{\pd t^2}(x,t)  dt,
\end{equation}
where $q$ solves the adjoint-state equation
\begin{equation}\label{eq:wave-adjoint}
\sigma_0^2(x) \frac{\pd^2 q}{\pd t^2} - \frac{\pd q}{\pd x^2}   = 0,
\end{equation}
with final condition
\[
q(x,T) = \frac{1}{\sigma^2_0(x)} d_2(x), \qquad \frac{\pd q}{\pd t}(x,T) = \frac{-1}{\sigma^2_0(x)} d_1(x).
\]
(notice the swap of $d_1$ and $d_2$, and the minus sign.)

In nice setups, the action of $F^*[\sigma_0^2]$, or $F^*$ for short, called imaging operator, is kinematically equivalent to that of the inverse $F^{-1}$ in the sense that the singularities of $r$ are in the same location as those of $F^*  \begin{pmatrix} d_1 \\ d_2 \end{pmatrix}$. In other words the normal operator $F^* F$ is pseudodifferential. We also show in the Appendix that $F^*$ is the negative Frechet derivative of a misfit functional for snapshot data, with respect to the medium $\sigma^2$, in the tradition of adjoint-state methods \cite{Ple}. Hence applying the imaging operator is a useful component of solving the full inverse problem.

A standard timestepping method is adequate to compute (\ref{eq:imaging-op}); the adjoint wave equation is first solved until time $t = 0$, then both $u$ and $q$ are evolved together by stepping forward in time and accumulating terms in the quadrature of (\ref{eq:imaging-op}). However, this approach is not without problems:
\begin{itemize}
\item The CFL condition restricting the time step for solving the wave equation is typically smaller  than the time gridding needed for computing an accurate quadrature of (\ref{eq:imaging-op}). The potential of being able to perform larger, upscaled time steps is obvious.

\item The backward-then-forward technique just discussed assumes time-reversibility of the equation in $q$ (or conversely of the equation in $u$), a condition that is not always met in practice, notably when the wave equation comes with either absorbing boundary conditions or an additional viscoelastic term. For non-reversible equations, computation of (\ref{eq:imaging-op}) comes with a big memory overhead due to the fact that one equation is solved from $t = 0$ to $T$, while the other one is solved from $t = T$ to $0$. One naive solution is to store the whole evolution; a more sophisticated approach involves using checkpoints \cite{Symes-checkpoint}, where memory is traded for CPU time, but still does not come close to the ``working storage" in the time-reversible case. In this context, it would be doubly interesting to avoid or minimize the penalty associated with time stepping.
\end{itemize}

\paragraph{Numerical validation of compressive reverse time migration.}

As a proof of concept, we now show how to perform snapshot reverse-time migration (RTM) without timestepping in the case of the reversible wave equation, on a 1D grid of $N=2048$ points. The approach here is simply to compute independently each term of a quadrature of (\ref{eq:imaging-op}) using the compressive wave algorithm for $u$ and $q$. We leave to a future project the question of dealing with 2D and 3D non-reversible examples, but we are confident that most of the ideas will carry through.

The smooth medium $\si_0^2$ is defined as in \eqref{eq-smooth-medium} with $\ga = 1$ (one oscillation) and a contrast $\si_{\max}^2/\si_{\min}^2 = 1.4$. The reflectors $r(x)$ is a sum of two Gaussian bumps of standard deviation $7/N$ and amplitude respectively $-0.6$ and $0.6$, see figure \ref{fig-rtm-error}, top row.

We first compute the input $d_1$ and $d_2$ of the RTM by computing the solution $u(x,t)$ of the wave equation in the perturbed medium $\si^2 = \si_0^2 + r$, and then evaluate $d_1(x) = u(x,T)$ and $d_2(x) = \frac{\partial u}{\partial t} (x,T)$. The initial condition $u(x,0)$ is a second derivative of a Gaussian of standard deviation $7/N$. This simulates seismic observations at time $t=T$, see figure \ref{fig-rtm-error}, top row.  

The algorithm proceeds by computing approximations $\tilde q(x,t_i)$ and $\tilde u_{\text{inc}}(x,t_i)$ of the forward and backward propagations $q$ and $u_{\text{inc}}$ at $N_t=N/10$ equispaced times $\{ t_i \}_{i=0}^{N_t-1}$. These approximations are computed for each $t_i$ independently, without time stepping, by using the compressive algorithm with a small set of $K < N$ eigenvectors. The RTM estimation of the residual $r(x)$ is obtained by discretizing \eqref{eq:imaging-op}
\eq{
	\tilde r(x) = 
	- \frac{1}{n_t} \sum_{i=0}^{n_t-1} \tilde q(x,t_i) \frac{\pd^2 \tilde  u_{\mbox{inc}}}{\pd t^2}(x,t_i),
}
where the derivative is computed using finite differences. 

The success of compressive RTM computations is measured using an error measure $\Err(K/N)$ obtained similarly to \eqref{eq-error-measure} by averaging over several randomizations for the sets $\Om \in \Om_K$ of $|\Om|=K$ eigenvectors
\eq{
	\Err(K/N)^2 = \frac{1}{N |\Om_K| \, \norm{u}} \sum_{\Om \in \Om_K} \sum_{j=0}^{N-1} | \tilde r_0[j] - \tilde r[j] |^2 
}
where $\tilde r_0$ is the RTM estimation obtained with the full set of $N$ eigenvectors.

Figure \ref{fig-rtm-error}, bottom row, displays the decay of $\Err(K/N)$ with the number of eigenvectors used for the compressive computations. This shows that roughly $20\%$ of eigenvectors are needed to reach 1 digit of accuracy, and $30\%$ to reach 2 digits of accuracy.

Note that this simple RTM method cannot be expected to recover the original $r(x)$ accurately, mainly because we have not undone the action of the normal operator $F^* F$ in the least-square treatment of the linearized problem.

Also note that the compressive algorithm was run ``as is", without any decomposition of the initial condition, or split of the time interval over which each simulation is run, as in Section \ref{sec:sparsity-enhancement}.

\myfigure{
	\begin{tabular}{cc}
    \includegraphics[width=.35\linewidth,height=.25\linewidth]{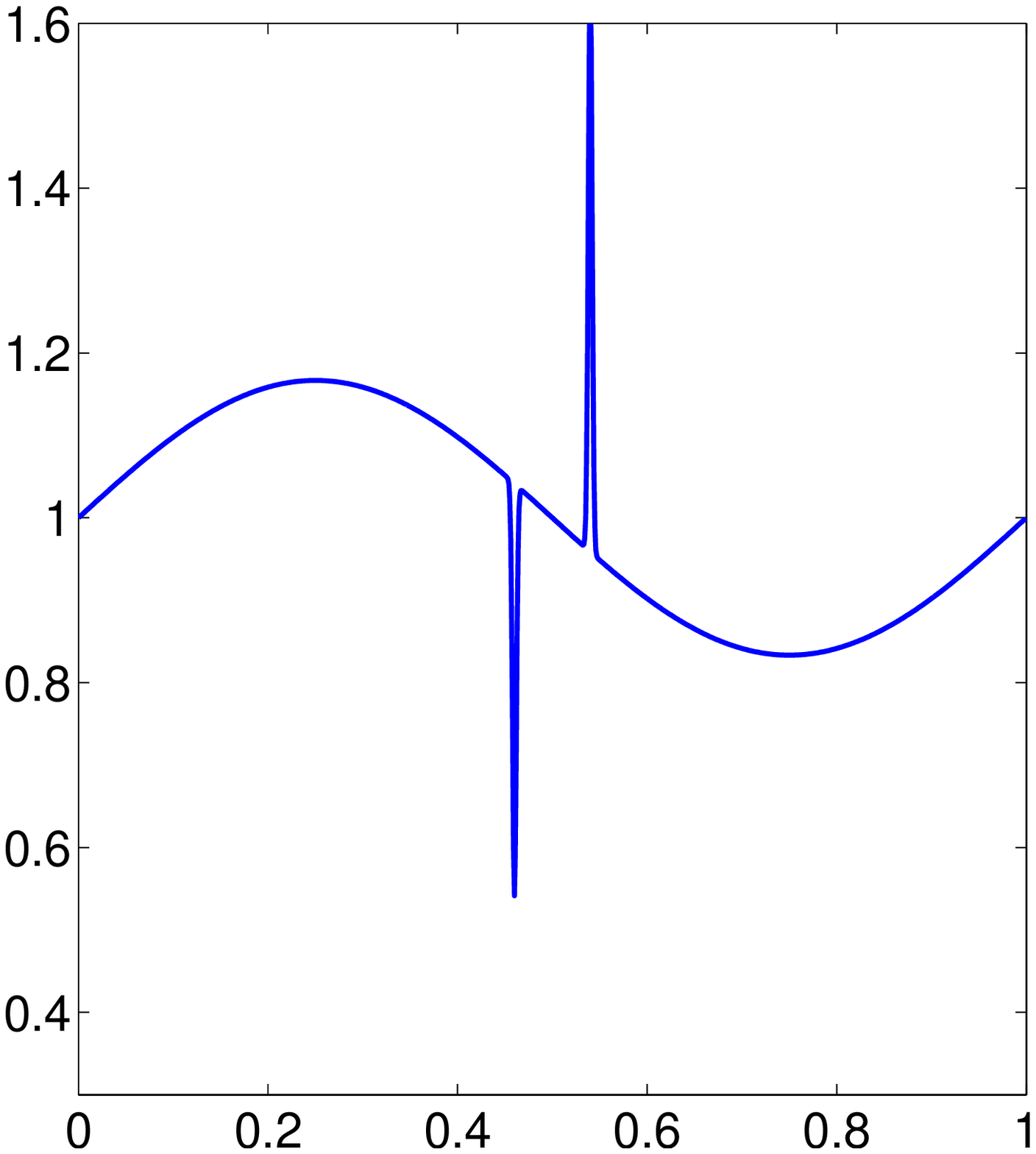}&
    \includegraphics[width=.35\linewidth,height=.25\linewidth]{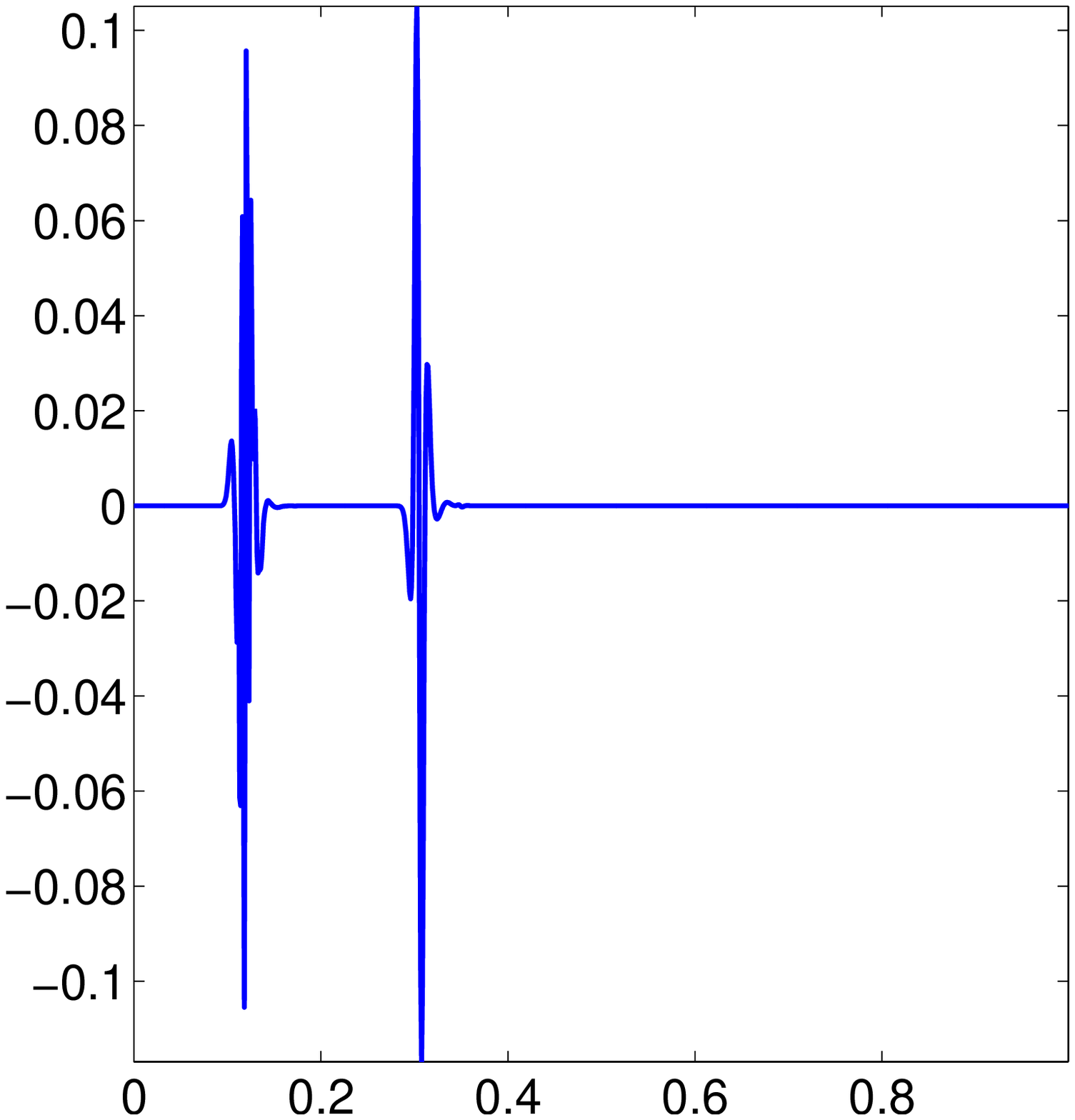}\\
    $\si^2 = \si_0^2 + r$ & $d_1$
    \end{tabular}
    \\
	\begin{tabular}{cc}
    \includegraphics[width=.35\linewidth,height=.25\linewidth]{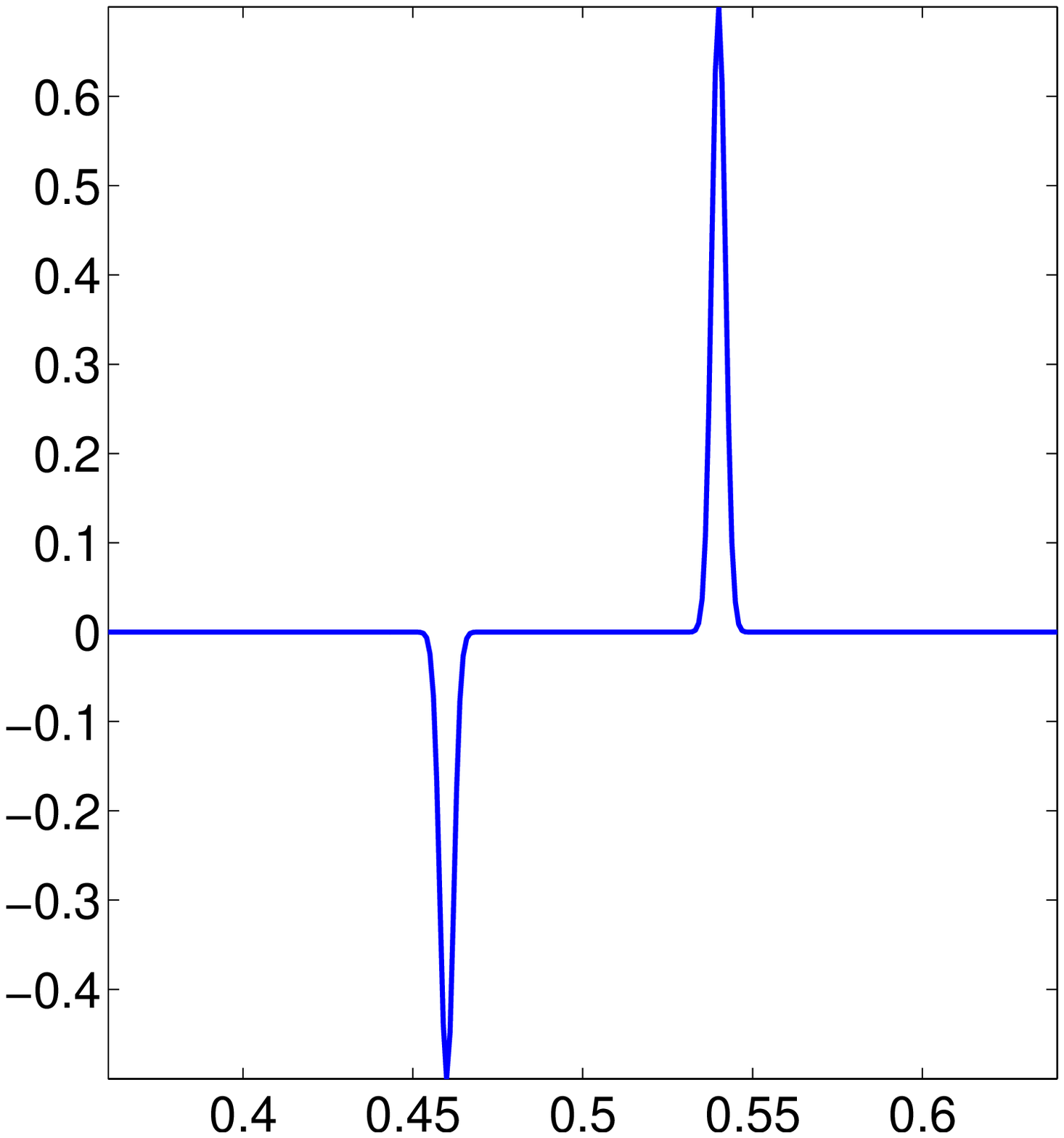}&
    \includegraphics[width=.35\linewidth,height=.25\linewidth]{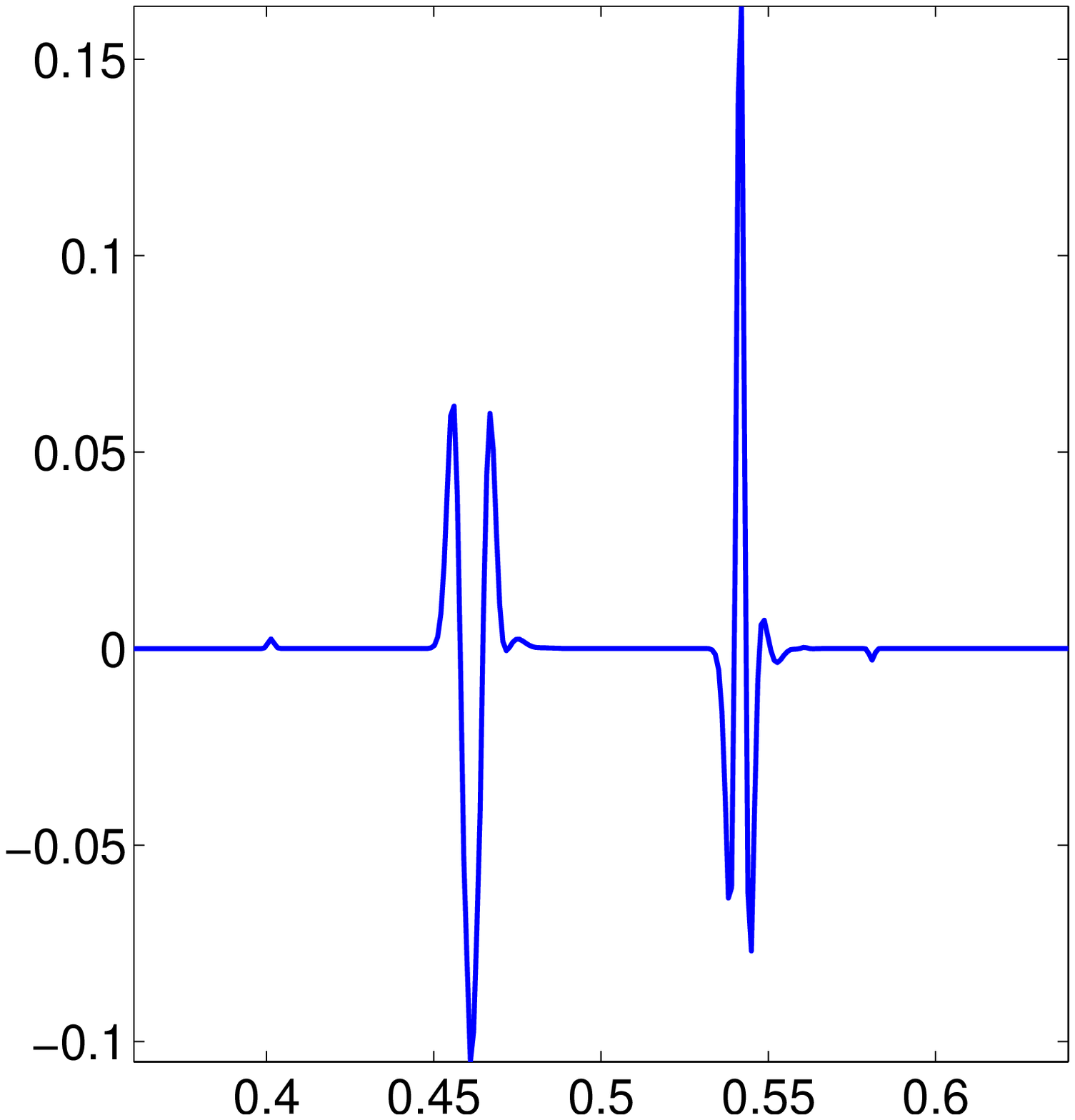}\\
    $r(x)$ (zoom) & $\tilde r(x)$ (zoom)
    \end{tabular}
    \\
	\begin{tabular}{c}
    \includegraphics[width=.4\linewidth]{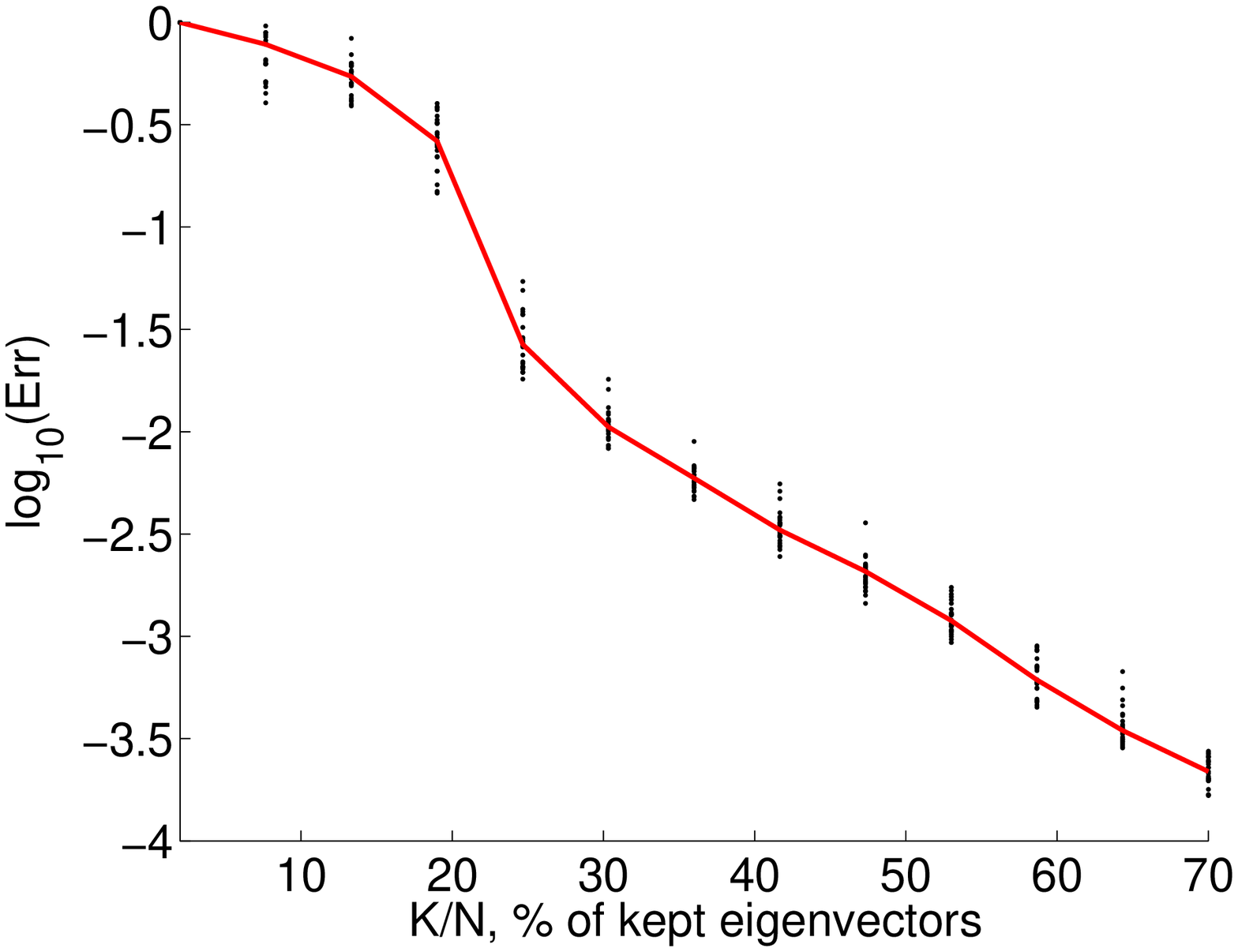}\\
    $\log_{10}(\Err(K/N))$
    \end{tabular}
}{ 
	Top row: perturbed speed $\si^2$ and input $d_1$ of the RTM computations.
	Middle row: reference solution $r_{\text{RTM}}$ and approximated solution $\tilde r_{\text{RTM}}$ for $K/N = 0.2$.
	Bottom row: recovery error decay $\log_{10}(\Err(K/N))$ as a function of the sub-sampling $K/N$. Each black dot corresponds to the error of a given random set $\Om$ (the red curve is the result of the averaging among these sets).
}{fig-rtm-error}

\section{Discussion}

A number of questions still need to be addressed, from both a mathematical and a practical viewpoints.

\begin{itemize}
\item \emph{Number of eigenvectors}. While it is easy to check the accuracy of a compressive solver relative to a standard scheme, an important question is a posteriori validation
without comparison to an expensive solver. Were enough eigenvectors sampled?
One way to answer this question could be to compare simulation results to those of a coarse-grid standard finite difference scheme. Another possibility is to check for convergence in $K$, as the recovered solution stabilizes as more and more eigenvectors are added. To the best of our knowledge compressed sensing theory does not yet contain a theoretical understanding of a posteriori validation, though.

\item \emph{Two and three-dimensional media}. Passing to a higher-dimensional setting will require using adequate bases for the sparsity of wavefields. Curvelets is one choice, and wave atoms is another one, as both systems were shown to provide $\ell_p$ to $\ell_p$ boundedness for all $p > 0$, of the wave equation Green's functions in $C^\infty$ media \cite{CD, Smith-wave}. (Wavelets or ridgelets would only provide $\ell_p$ to $\ell_1$ boundedness.) The sparsity question for wave equations in simple nonsmooth media, e.g. including one interface, is to our knowledge completely open. The question of incoherence of the eigenfunctions with respect to such systems will be posed scale-by-scale, in the spirit of wavelets vs. Fourier as in \cite{chen-basis-pursuit} for instance. Incoherence or extension questions for generalized eigenfunctions may also be complicated by the presence of infinite-dimensional eigenspaces in unbounded domains.

\item \emph{Parallel computing}. An exciting outlook is that of scaling the compressive approach to clusters of computers where properly preconditioned and domain-decomposed Helmholtz equations are solved in an entirely parallel fashion over $\omega$. The $\ell_1$ solver should be parallelized too, and for the iterative schemes considered the two bottleneck operations are those of applying the incomplete matrix of eigenvectors and its transpose. For reverse-time migration where several times need to be considered, an outstanding question is that of being able to reliably predict the location of large basis coefficients at time $t$ from the knowledge of wavefields already solved for at neighboring times $t - \Delta t$ and $t + \Delta t$. Note that the advocated ``snapshot" variant of reverse-time migration in Section \ref{sec:crtm} contains a preliminary time-to-depth conversion step.

\item \emph{Other equations and random media}. Any linear evolution equation that obeys interesting sparsity and incoherence properties can in principle be studied in the light of compressive computing. This includes the heat equation in simple media; and possibly also the (one-particle) Schr\"{o}dinger equation in smooth potentials when a bound on wave numbers of the kind $|\xi| \lesssim 1/h$ is assumed, where $h$ is Planck's constant. Interestingly, the compressive method also makes sense for mildly random media as the $\ell_1$ reconstruction empirically filters out the trailing ``coda" oscillations when they follow behind a coherent wavefront.

\end{itemize}

\appendix
\section{Additional proofs}

{\bf Complement to the proof of Theorem \ref{teo:incoherence}.}

In order to justify (\ref{eq:interm-decayBV}), consider the quantity
\[
I_\eps(x) = |u(x)|^2 + \frac{|u'(x)|^2}{\omega^2 \sigma_\eps^2(x)},
\]
where $\sigma_\eps$ is adequately regularized in the sense of Lemma \ref{teo:BV}, but $u(x)$ still solves the equation with $\sigma(x)$. Then by point 1 of Lemma \ref{teo:BV}, $I'_\eps(x)$ makes sense, and
\[
I'_\eps(x) = \left( \frac{\sigma_\eps^2(x) - \sigma^2(x)}{\sigma_\eps^2(x)} \right) \, 2 \mbox{Re}  \, (u'(x) \overline{u}(x)) - 2 (\log \sigma_\eps(x))' \frac{|u'(x)|^2}{\omega^2 \sigma_\eps^2(x)}.
\]
The inhomogeneous Gronwall inequality\footnote{If $\dot{y}(t) \leq f(t) + g(t) y(t)$, then $y(t) \leq y(0) \exp \left( \int_0^t g(s) ds \right) + \int_0^t f(s) \exp \left( \int_s^t g(\tau) d\tau \right) ds$, and vice-versa with $\geq$ in both equations instead of $\leq$.} can be applied and yields
\begin{equation}\label{eq:Ieps}
I_\eps(x) \geq I_\eps(0) \,\exp \left( - 2 \int_0^x | (\log \sigma_\eps(y))' | \, dy \right) - \int_0^x F_\eps(y) \exp \left( - 2 \int_y^x | (\log \sigma_\eps(z))' | \, dz \right) \, dy,
\end{equation}
where
\[
F_\eps(x) = \left( \frac{| \sigma_\eps^2(x) - \sigma^2(x) |}{\sigma_\eps^2(x)} \right) \, 2 | u'(x) u(x) |.
\]
We can now invoke the different points of Lemma \ref{teo:BV}. Points 2 and 3, in conjunction with uniform boundedness of $u$ and $u'$ over $[0,1]$, allow to conclude that $\int_0^1 F_\eps(x) \, dx \leq C \cdot \eps$ for some constant $C$. By points 4 and 5, the second term in the right-hand side of equation (\ref{eq:Ieps}) tends pointwise to zero as $\eps \to 0$.  Hence by point 2 we have $\lim_{\eps \to 0} I_\eps(x) = I(x)$; and the first term in the right-hand side of equation (\ref{eq:Ieps}) is handled again by points 4 and 5. The inequality obtained in the limit is
\[
I(x) \geq I(0) \exp \left( - 2 \mbox{Var}_x(\log \sigma) \right),
\]
which is what we sought to establish.

\bigskip

{\bf Complement to the proof of Theorem \ref{teo:egv-gaps}.}

Let us show that (\ref{eq:egv-gaps}) holds when $\sigma \in BV([0,1])$ and not just $C^1([0,1])$. Let $\sigma_\eps$ be a regularization of $\sigma$, in the usual sense. Define $\theta_\eps(x)$ through
\[
\cot \theta_\eps(x) = \frac{1}{\omega \sigma_\eps(x)} \frac{u'(x)}{u(x)}, \qquad \mbox{or} \qquad \tan \theta_\eps(x) = \omega \sigma_\eps(x) \frac{u(x)}{u'(x)},
\]
where $u(x)$ solves the Sturm-Liouville problem with the un-regularized $\sigma$. Then a short calculation shows that
\[
\theta'_\eps(x) = \frac{\sigma'_\eps(x)}{\sigma_\eps(x)} \, \sin \theta_\eps(x) \cos \theta_\eps(x) + \omega \sigma(x) \, \frac{\sigma(x)}{\sigma_\eps(x)} \frac{{\sin}^2 \theta_\eps(x)}{{\sin}^2 \theta(x)}.
\]
The ratio of sine squared can be recast in terms of a familiar quantity:
\[
\frac{{\sin}^2 \theta_\eps(x)}{{\sin}^2 \theta(x)} = \frac{1 + {\cot}^2 \theta_\eps(x)}{1 + {\cot}^2 \theta(x)} = \frac{1 + \frac{1}{\omega^2 \sigma^2_\eps(x)} |\frac{u'(x)}{u(x)}|^2}{1 + \frac{1}{\omega^2 \sigma^2(x)} |\frac{u'(x)}{u(x)}|^2} = \frac{I_\eps(x)}{I(x)}.
\]
The same result follows from a similar formula with cosines in case $\sin \theta = 0$. Consider now two different frequencies $\omega_1$, $\omega_2$, and integrate to get
\begin{align*}
\theta_{\eps,1}(1) - \theta_{\eps,2}(1) = \int_0^1 \frac{\sigma'_\eps(x)}{\sigma_\eps(x)} \, &[ \sin \theta_{\eps,1} \cos \theta_{\eps,1} - \sin \theta_{\eps,2} \cos \theta_{\eps,2} ] dx \\
&\qquad +  (\omega_1 - \omega_2) \int_0^1 \sigma(x) \frac{\sigma(x)}{\sigma_\eps(x)} \frac{I_\eps(x)}{I(x)} \, dx.
\end{align*}
The first term is bounded by $\int_0^1 | \frac{\sigma'_\eps(x)}{\sigma_\eps(x)}| dx$, which tends to Var$(\log\sigma)$ as $\eps \to 0$ by point 5 of Lemma \ref{teo:BV}. The second integral can be compared to $\int_0^1 \sigma(x) dx$ by studying
\begin{align*}
| \int_0^1 \sigma \frac{\sigma}{\sigma_\eps} \frac{I_\eps}{I} \, dx - \int_0^1 \sigma \, dx | &= | \int_0^1 \frac{\sigma}{\sigma_\eps I_\eps} (\sigma I_\eps - \sigma_\eps I) \, dx \\
&\leq \int_0^1 \frac{\sigma I_\eps}{\sigma_\eps I} | \sigma - \sigma_\eps | \, dx + \int_0^1 \frac{\sigma}{I} | I_\eps - I | \, dx.
\end{align*}
All four factors $\sigma$, $I$, $\sigma_\eps$ and $I_\eps$ are bounded and bounded away from zero, by Theorem \ref{teo:incoherence} for $I$ and point 3 of Lemma \ref{teo:BV} for $\sigma_\eps$ and $I_\eps$. Point 2 of Lemma \ref{teo:BV} then implies that $\int_0^1 |\sigma_\eps - \sigma| \, dx \to 0$. We have already argued in the proof of Theorem \ref{teo:incoherence} that $\int_0^1 | I_\eps - I | \, dx \to 0$. The limiting inequalities are (\ref{eq:egv-gaps}), as desired.

\bigskip

{\bf Proof of Proposition \ref{teo:nonuniform}}

Let $\underline{p} = \min p_n$, where $n = 1, \ldots, N$, and let $\punif = 1/N$. The basic heuristic is that, even if the probability of drawing a given measurement decreases by a factor $p_n / \punif$ which may be less than 1, this effect is at least more than offset by sampling a correspondingly larger number of measurements, namely $K_{\unif} \; \punif/\underline{p}$ instead of $K_{\unif}$.

First, let us see why the accuracy bound (\ref{eq:accuracy}) still holds if we increase the probability for any particular $\omega_n$ to be included in the list $\Omega_K$.\footnote{Interestingly, the statement that the accuracy of recovery is improved by adding measurements to the $\ell_1$ minimization problem, is in general false. It is only the error bound that does not degrade.} In our context, assume for the moment that
\begin{equation}\label{eq:assump-prob}
\Pr(\omega_n \in \Omega_K) \geq \Kunif \; \punif,
\end{equation}
uniformly over $n$. Then we may reduce this setup to a situation where all $\Pr(\omega_n \in \Omega_K) = \Kunif \; \punif$ by a rejection sampling procedure, where each $\omega_n \in \Omega_K$ is further labeled as primary, or ``accepted", with probability $(\Kunif \; \punif)/\Pr(\omega_n \in \Omega_K)$, and labeled as secondary, or ``rejected", otherwise. The set of primary measurements alone would be adequate to call upon the general theory; passing to obvious linear algebra notations the corresponding constraints are denoted as $\| Ax - b \|_2 \leq \eps$ where $A$ satisfies the $S$-RIP with high probability. Adding the secondary measurements cannot fundamentally hurt the recovery of $\ell_1$ minimization, because the constraints become
\[
\| Ax - b \|_2^2 + \| B x - c \|_2^2 \leq (\eps')^2
\]
for some $\eps' \asymp \eps$, hence in particular include $\| Ax - b \|_2 \leq \eps'$. This latter ``tube" condition, together with the cone condition that $\| x \|_1$ should be minimized, suffice to obtain the recovery estimate on $x$ as was shown in \cite{candes-stable}. Adding a compatible constraint involving $\| Bx - c \|_2$ only decreases the feasibility set, hence does not alter the recovery estimate.

Second, it remains to show (\ref{eq:assump-prob}). An event such as $\omega_n \in \Omega_K$ refers to the result of a sequential sampling of $K$ objects among $N$, without replacement, and according to a certain distribution, either $p_n$ or $\punif$. This setup is different that the one considered in the classical papers on compressed sensing. It is shown below that
\begin{equation}\label{eq:Punif}
\Pr(\omega_n \in \Omega_{\Kunif}) = \Kunif \; \punif \qquad\qquad \mbox{(uniform probabilities $\punif$)}
\end{equation}
in the uniform case, and
\begin{equation}\label{eq:Pnonunif}
\Pr(\omega_n \in \Omega_K) \geq K \underline{p} \qquad\qquad \mbox{(nonuniform probabilities $p_n$)}
\end{equation}
in the nonuniform case. In order for $\Pr(\omega_n \in \Omega_K)$ to be greater than $\Kunif \;  \punif$ as in equation (\ref{eq:assump-prob}), it suffices therefore that $K \geq \Kunif \;  \punif/\underline{p}$, which proves the proposition.

The following proof of (\ref{eq:Punif}) and (\ref{eq:Pnonunif}) is due to Paul Rubin, who kindly allowed us to reproduce the argument here.

In the case of uniform probabilities, write $\punif = p$ for short. Denote by $\Omega_k$ the set of $\omega_n$ drawn during the first $k$ iterations. By Bayes theorem applied to a sequence of nested (random) events, we have
\begin{align*}
\Pr(\omega_n \notin \Omega_k) &= \Pr(\omega_n \notin \Omega_1) \, \cdot \, \Pr(\omega_n \notin \Omega_2 \vert \omega_n \notin \Omega_1) \, \ldots \, \Pr(\omega_n \notin \Omega_k \vert \omega_n \notin \Omega_{k-1} ) \\
&= (1 - p) \, \cdot \, \left( 1 - \frac{p}{1-p} \right) \, \ldots \, \left(1 - \frac{p}{1-(k-1)p} \right) \\
&= (1 - p) \, \left( \frac{1- 2p}{1-p} \right) \, \ldots \, \left( \frac{1-kp}{1-(k-1)p} \right).
\end{align*}
The factors telescope in this product, with the result that $\Pr(\omega_n \notin \Omega_k) = 1 - kp$, hence $\Pr(\omega_n \in \Omega_K) = kp$ as desired. 

In the case of nonuniform probabilities $p_n$, consider the quantity
\[
q_k = \sum_{n : \omega_n \in \Omega_k} p_n.
\]
The $q_k$ are themselves random variables, since they depend on the history of sampling up to step $k$. Denote by $\hat{\Omega}_{k}$ one realization of $\Omega_k$, and $\hat{q}_k$ the corresponding realization of the random process $q_k$. Then the reasoning essentially proceeds as previously, only by induction on $k$. First, $\Pr(\omega_n \notin \Omega_1 ) = 1 - p_n \leq 1 - \underline{p}$. Assume $\Pr(\omega_n \notin \Omega_{k-1}) \leq 1 - (k-1) \underline{p}$. Then
\[
\Pr(\omega_n \notin \Omega_k \vert \omega_n \notin \hat{\Omega}_{k-1} ) = 1 - \frac{p_n}{1 - \hat{q}_{k-1}} = \frac{1 - \hat{q}_{k-1} - p_n}{1 - \hat{q}_{k-1}}.
\]
We claim that
\begin{equation}\label{eq:Rubinlemma}
\frac{1 - \hat{q}_{k-1} - p_n}{1 - \hat{q}_{k-1}} \leq \frac{1 - k \underline{p}}{1 - (k-1) \underline{p}}.
\end{equation}
To justify this inequality, write the sequence of equivalences
\begin{align*}
\frac{1 - \hat{q}_{k-1} - p_n}{1 - \hat{q}_{k-1}} &\leq \frac{1 - k \underline{p}}{1 - (k-1) \underline{p}} \\
(1 - (k-1) \underline{p})(1 - \hat{q}_{k-1} - p_n) &\leq (1 - k \underline{p}) (1 - \hat{q}_{k-1}) \\
- p_n + (k-1) \underline{p} p_n &\leq - \underline{p} + \hat{q}_{k-1} \underline{p} \\
0 &\leq (p_n - \underline{p}) + \underline{p} (\hat{q}_{k-1} - (k-1) \underline{p}).
\end{align*}
The last inequality is obviously true. Therefore  
\[
\Pr(\omega_n \notin \Omega_k \vert \omega_n \notin \hat{\Omega}_{k-1} ) \leq \frac{1 - k \underline{p}}{1 - (k-1) \underline{p}}.
\]
Because this bound is uniform over $\hat{\Omega}_{k-1}$, it is easy to see that the same inequality holds when the conditioning is on $\omega_n \notin \Omega_{k-1}$ instead of being on $\omega \notin \hat{\Omega}_{k-1}$.\footnote{Indeed, if a random event $B$ is the disjoint union $\bigcup B_i$, and $\Pr(A \vert B_i) \leq C$ for all $i$, then $\Pr(A \vert B) \leq C$ as well. This fact is a simple application of Bayes's theorem: $\Pr(A \vert B) = \Pr(A \& B)/\Pr(B) = \sum_i \Pr(A \& B_i)/\Pr(B) \leq C \sum_i \Pr(B_i)/\Pr(B) = C$.}

We conclude by writing
\begin{align*}
\Pr(\omega_n \notin \Omega_{k}) &= \Pr(\omega_n \notin \Omega_{k} | \omega_n \notin \Omega_{k-1}) \Pr(\omega_n \notin \Omega_{k-1}) \\
 &\leq  \frac{1 - k \underline{p}}{1 - (k-1) \underline{p}} (1 - (k-1) \underline{p}) \\
 &= 1 - k \underline{p}.
\end{align*}

%

\bigskip

{\bf Proof and discussion of equation (\ref{eq:imaging-op})}

In this proof, a subscript $t$ denotes a time differentiation. Consider the following quantity,
\[
\int_0^T \int u \, \left[ \sigma_0^2 \frac{\pd^2}{\pd t^2} - \frac{\pd^2}{\pd x^2} \right] q \, dx dt,
\]
where $u$ solves (\ref{eq:wave-linearized}), and $q$ solves (\ref{eq:wave-adjoint}) with final data $(d_1(x), d_2(x))$ not necessarily equal to $(u(x,T), u_t(x,T))$. Because of (\ref{eq:wave-adjoint}), this integral is zero. Two integrations by parts in $x$ and $t$ reveal that
\[
0 = \int \sigma_0^2 (q_t u) \vert_0^T \, dx - \int \sigma_0^2 (q u_t) \vert_0^T \, dx + \int_0^T \int q \,\left[ \sigma_0^2 \frac{\pd^2}{\pd t^2} - \frac{\pd^2}{\pd x^2} \right] u \, dx dt.
\]
The boundary terms in $x$ have been put to zero from assuming, for instance, free-space propagation. In the first two terms, we can use $u(x,0) = u_t(x,0) = 0$, and recognize that $\sigma_0^2 q_t(x,T) = - d_1(x)$ and $\sigma_0^2 q(x,T) = d_2(x)$. For the last term, we may use (\ref{eq:wave-linearized}). The result is
\[
\int d_1(x) u(x,T) \, dx + \int d_2(x) u_t(x,T) \, dx =  - \int r(x) \int_0^T q(x,t) \frac{\pd^2 u_{\mbox{\scriptsize inc}}}{\pd t^2}(x,t) \, dt \, dx.
\]
On the left, we recognize $\< \begin{pmatrix} d_1 \\ d_2 \end{pmatrix},  \, F[\sigma_0^2] r \>$, where the inner product is in $L^2(\R^n, \R^2)$. The right-hand-side is a linear functional of $r(x)$, from $L^2(\R^n, \R)$ to $\R$, hence we identify
\[
F^*[\sigma_0^2] \begin{pmatrix} d_1 \\ d_2 \end{pmatrix} = - \int_0^T q(x,t) \frac{\pd^2 u_{\mbox{\scriptsize inc}}}{\pd t^2}(x,t) \, dt,
\]
as required.

Notice that this formula for the adjoint operator is motivated by a standard optimization argument that we now reproduce. Consider the misfit functional $J[\sigma^2, u]$ for the data $(d_1, d_2)$, defined as
\[
J[\sigma^2, u] = \frac{1}{2} \int (u(x,T) - d_1(x))^2 + (u_t(x,T) - d_2(x))^2 \, dx,
\]
and under the constraint that $u$ solves the original wave equation (\ref{eq:wave-original}). Then the form of $F^*$ is inspired by the negative Frechet derivative $- \frac{\delta J}{\delta \sigma^2}[\sigma_0^2, u_{\mbox{\scriptsize inc}}]$. In order to see this, we can define a dual variable $q(x,t)$ corresponding to the constraint (\ref{eq:wave-original}), dual variables $\mu_0(x)$ and $\mu_1(x)$ corresponding to the initial conditions for $u$, and write the Lagrangian
\[
L[\sigma^2; u, q, \mu_0, \mu_1] = J[\sigma^2, u] - \int_0^T \int q \, \left[ \sigma^2 \frac{\pd^2}{\pd t^2} - \frac{\pd^2}{\pd x^2} \right] u \, dx dt + \int \mu_0 u(x,0) dx + \int \mu_1 u_t(x,0) dx.
\]
The same integrations by parts as earlier give
\[
L[\sigma^2; u, q, \mu_0, \mu_1] = J[\sigma^2, u] - \int_0^T \int u \, \left[ \sigma^2 \frac{\pd^2}{\pd t^2} - \frac{\pd^2}{\pd x^2} \right] q \, dx dt
\]
\[
\qquad\qquad\qquad + \int \sigma^2 (q u_t) \vert_0^T \, dx - \int \sigma^2 (q_t u) \vert_0^T \, dx + \int \mu_0 u_0 dx + \int \mu_1 u_1 dx.
\]
The Frechet derivatives of $L$ with respect to $u(x,t)$, $u(x,T)$, and $u_t(x,T)$ reveal the adjoint-state equations:
\[
\frac{\delta L}{\delta u(x,t)} = - \left[ \sigma^2 \frac{\pd^2}{\pd t^2} - \frac{\pd^2}{\pd x^2} \right] q(x,t) = 0;
\]
\[
\frac{\delta L}{\delta u(x,T)} = u(x,T) - d_1(x) - \sigma^2 q_t(x,T) = 0;
\]
\[
\frac{\delta L}{\delta u_t(x,T)} = u_t(x,T) - d_2(x) + \sigma^2 q(x,T) = 0.
\]
The derivative of $L$ with respect to $\sigma^2$ gives the desired formula
\[
\frac{\delta J}{\delta \sigma^2} =  \int_0^T q(x,t) \frac{\pd^2 u}{\pd t^2}(x,t) \, dt.
\]
The derivatives with respect to $q$ and $\mu_0$, $\mu_1$ replicate the constraints, and finally the derivatives with respect to $u(x,0)$ and $u_t(x,0)$ are not interesting because they involve $\mu_0$, $\mu_1$, which do not play a role in the expression of $\delta J/\delta \sigma^2$.


Notice that the final conditions for $q$ involve the predicted wavefields $u$ and $u_t$. If we put $\sigma = \sigma_0$ however, and $\sigma_0$ is smooth, then $u = u_{\mbox{\scriptsize inc}}$. In that case, essentially no reflections occur and the wavefields $u_{\mbox{\scriptsize inc}}(x,T)$, $\pd u_{\mbox{\scriptsize inc}} / \pd t (x,T)$ are zero at time $T$, in the region of interest where the data lies. It is only when making corrections to the guess $\sigma_1$ that the predicted wavefields are important in the final condition for $q$.

\bibliographystyle{plain}
\bibliography{cwc}

\end{document}